\documentclass{article}
\usepackage{latexsym}
\usepackage{amsmath,amsthm}

\usepackage{amssymb}
\usepackage{upgreek}

\usepackage[margin=1.1in,dvips]{geometry}

\newcommand{\vol}{\textnormal{vol}}

\def\f12{\frac 1 2}

\def\a{\alpha}
\def\b{\beta}
\def\ga{\gamma}
\def\Ga{\Gamma}

\def\ep{\epsilon}
\def\la{\lambda}

\def\si{\sigma}
\def\Si{\Sigma}
\def\om{\omega}
\def\Om{\Omega}

\def\th{\theta}

\def\Lb{\underline{L}}

\def\pa{\partial}
\def\les{\lesssim}

\def\f12{\frac 1 2}

\newcommand{\lap}{\mbox{$\Delta \mkern-13mu /$\,}}

\newcommand{\nabb}{\mbox{$\nabla \mkern-13mu /$\,}}
\newcommand{\paL}{\mbox{$\pa \mkern-13mu /$\,}}
\newtheorem{thm}{Theorem}

\newtheorem{prop}{Proposition}

\newtheorem{lem}{Lemma}
\newtheorem{cor}{Corollary}
\newtheorem{remark}{Remark}
\begin{document}

\title{On the quasilinear wave equations in time dependent inhomogeneous media
\footnote{This work is part of the author's Ph.D. thesis at Princeton University.}}
\author{Shiwu Yang}
\date{}
\maketitle
\begin{abstract}
We consider the problem of small data global existence for quasilinear wave equations with null condition on a class of Lorentzian
manifolds $(\mathbb{R}^{3+1}, g)$ with \textbf{time dependent} inhomogeneous metric. We show that sufficiently small data give rise to a unique global solution for metric
which is merely $C^1$ close to the Minkowski metric inside some large cylinder $\{\left.(t, x)\right||x|\leq R\}$
and approaches the Minkowski metric weakly as $|x|\rightarrow \infty$. Based on this result, we give weak but sufficient
conditions on a given large solution of
quasilinear wave equations such that the solution is globally stable under perturbations of initial data.
\end{abstract}

\section{Introduction}

In this paper, we study the Cauchy problem for the quasilinear wave equations
\begin{equation}
 \label{QUASIEQ}
\begin{cases}
 \Box_{g} \phi+g^{\mu\nu\ga}\pa_\ga\phi\cdot \pa_{\mu\nu }\phi=A^{\mu\nu}\pa_\mu\phi\pa_\nu\phi+F(\phi, \pa\phi),\\
\phi(0, x)=\phi_0(x), \quad \pa_t\phi(0, x)=\phi_1(x)
\end{cases}
\end{equation}
on a Lorentzian manifold $(\mathbb{R}^{3+1}, g)$, where $\Box_g$ is
the covariant wave operator for the metric $g$. The nonlinearities
are assumed to satisfy the null condition: $g^{\mu\nu\ga}$,
$A^{\mu\nu}$ are constants such that
$g^{\a\b\ga}\xi_\a\xi_\b\xi_\ga=0$, $A^{\mu\nu}\xi_\mu\xi_\nu=0$
whenever
 $\xi_0^2=\xi_1^2+\xi_2^2+\xi_3^2$ and $F$ is at least cubic in terms of $\phi$, $\pa\phi$ for small $\phi$, $\pa\phi$.

\bigskip

The Cauchy problem for nonlinear wave equations with general
quadratic nonlinearities on $\mathbb{R}^{n+1}$ has been studied
extensively. In $4+1$ or higher dimensions, the decay rate of the
solution to a linear wave equation is sufficient to
obtain the small data global existence result, see e.g. \cite{klgex},
\cite{klinvar}, \cite{kl-ponce}, \cite{ge-shatah} and reference therein. However, in $3+1$
dimensions, one can only show the almost global existence result
\cite{kl-johnalmostge}, \cite{klinvar}. In fact, in
\cite{johnblowup}, F. John showed that any nontrivial $C^3$ solution
of the equation
$$\Box \phi=(\pa_t\phi)^2
$$
with compactly supported initial data blows up in finite time.
Nevertheless, a sufficient condition on the quadratic
nonlinearities, which guarantees the small data global existence
result, is the celebrated null condition introduced by S. Klainerman
\cite{klNullc}. Under this condition, D. Christodoulou
\cite{ChDNull} and S. Klainerman \cite{klNull} independently proved
the small data global existence result.

\bigskip

The approach of \cite{ChDNull} used the conformal method, which
relies on the conformal embedding of Minkowski space to the Einstein cylinder
$R\times S^{3}$. S. Klainerman used the vector field method \cite{klinvar}, which
connects the symmetries of the flat $\mathbb{R}^{n+1}$ with the
quantitative decay properties of solutions of linear wave
equations. The vector fields, used as commutators or multipliers, are
the killing and conformal killing vector fields in
$\mathbb{R}^{n+1}$ and can be given explicitly
\begin{equation}
\label{Lorenzinv}
\Ga=\{ \Om_{ij}=x_i\pa_j-x_j\pa_i,\quad
L_i=t\pa_i+x_i\pa_t, \quad \pa_\a, \quad K=(t^2+r^2)\pa_t +
2tr\pa_r, \quad S=t\pa_t+r\pa_r\}.
\end{equation}
Based on this vector field method, there have been an extensive
literature on generalizations and variants of D. Christodoulou and
S. Klainerman's work, in particular on the multiple speed problems
e.g. \cite{klmulti}, \cite{sideris-multispeed}, \cite{soggemulti} and
obstacle problems e.g. \cite{sogge-metcalfe-nakamura}, \cite{soggestar},
\cite{soggestaralm}, \cite{sogge-metcalfe}. All these works used the scaling vector field $S$.

\bigskip

Another application of the vector field method is to the wave
equations on a Lorentzian manifold $(\mathcal{M}, g)$ with metric
$g$, which may also depend on the solution of the equation. The
motivation for studying such problems arises from studying the
problem of global nonlinear stability of Minkowski space in wave
coordinates. The stability of Minkowski space was first established
by Christodoulou-Klainerman by recasting the problem as a system of
Bianchi equations for the curvature tensor \cite{kcg}. Later,
Lindblad-Rodnianski \cite{SMigor} obtained a different proof in wave coordinates, in which the problem was
formulated as a system of quasilinear wave equations for the
components of the metric perturbation. This is an example that the
background metric $g(\phi)$ depends on the solution $\phi$ of the
wave equation, where $g(0)=m_0$ (the Minkowski metric). The
quasilinear part of such equations
 $g^{\a\b}(\phi)\pa_{\a\b}\phi$ never satisfies the null condition defined in the original work of
 S. Klainerman \cite{klNullc}, \cite{klNull}. Nevertheless, besides the global nonlinear stability of Minkowski space aforementioned
, the nonlinear wave equations
\begin{equation*}
 g^{\a\b}(\phi)\pa_{\a\b}\phi=0
\end{equation*}
on $\mathbb{R}^{3+1}$, which was studied by H. Lindblad in
\cite{gx-lindblad}, \cite{gx-lindblad2}, also admits small data
global solutions. A particular case
\[
 \pa_{tt}\phi-(1+\phi)^2\Delta \phi=0
\]
has been investigated by S. Alinhac
\cite{alinhac-example}.

\bigskip

The linear and nonlinear wave equations on a Lorentzian manifold
with given metric $g$ have also received considerable attention, in
particular on black hole spacetimes. For the linear wave equation on
$\mathbb{R}^{3+1}$, S. Alinhac \cite{alinhac-freedecay} showed that
the solution has the decay properties similar to those of a solution
of a linear wave equation on Minkowski space provided that the
metric $g$ approaches the Minkowski metric suitably as
$t\rightarrow \infty$. In \cite{lodecTataru}, D. Tataru proved the
local decay of the solution but with the assumption that the
background metric is stationary or time independent. For the decay
of solution of linear wave equations on Kerr spacetimes (including
Schwarzschild spacetimes), we refer the readers to \cite{dr3},
\cite{improvLuk}, \cite{jnotes}, \cite{blueHidden} and references therein.
For the nonlinear equations, J. Luk \cite{smallglobLuk} proved the
small data global existence result for semilinear wave equations
with derivatives on slowly rotating Kerr spacetimes. In a recent
work \cite{glsoChengbo}, Wang-Yu proved the small data global
existence result for quasilinear wave equations on static
spacetimes, which are more restrictive than stationary ones.

\iffalse
\cite{jnotes}(igor's lecture notes decay on black hole space; vector fields K)
\cite{improvLuk}(improved decay on Schwarzschild; vector fields S)
\cite{smallglobLuk}(small data global solution to semilinear wave equation on slowly rotating kerr; S)
\cite{blueHidden}(symmetries, decay, K)
\cite{dr3}(decay of linear solution on Schwarzschild spacetime, K)
\fi

\bigskip

A common feature of these problems is that the background metric $g$
settles down to a stationary metric either by the assumptions (Kerr
spacetimes are stationary) or, for the case when the metric
$g(\phi)$ depends on the solution, by the assumption $g(0)=m_0$ and
the expected convergence $\phi(t, x)\rightarrow 0$ as
$t\rightarrow\infty$. The need for such convergence,
or at least convergence of the time derivative of the metric to $0$, is dictated by the vector field method. All applications of the vector field method require commutations
with some generators of the conformal symmetries of Minkowski space. In particular, we note that all the applications have used the scaling vector field $S=t\pa_t +r\pa_r$ or the
conformal killing vector field $K=(t^2+r^2)\pa_t+2tr \pa_r$ as
commutators. For the problem
\[
 \Box_g \phi=F,
\]
the error term coming from the commutation with $S$ or $K$ would be of the form
$t \pa_t g^{\a\b}\pa_{\a\b}\phi$ or $t^2 \pa_t g^{\a\b}\pa_{\a\b}\phi$ which leads to the requirement that $t\pa_t g$ is at least bounded and thus the time derivative of the metric decays to $0$ as $t\rightarrow \infty$.

\bigskip

To our knowledge, the first work on nonlinear wave equations on time
dependent inhomogeneous background is the author's work
\cite{yang1}. It was shown that if the background metric is merely
$C^1$ close to the Minkowski metric inside the cylinder $\{\left.(t,
x)\right||x|\leq R\}$ and is identical to the Minkowski metric
outside, then the semilinear wave equations with derivatives
satisfying the null condition admit small data global solutions.
This result relies on a new method for proving the decay of the
solutions of linear wave equations developed by Dafermos-Rodnianski in \cite{newapp}. This new approach is a blend of an
integrated local energy inequality and a $p$-weighted energy
inequality in a neighborhood of the null infinity, also see
applications in \cite{Igorsurvey}, \cite{newapp3}, \cite{yang2}.

\bigskip

The aim of the present work is to extend the result in \cite{yang1}
to quasilinear wave equations on time dependent inhomogeneous
backgrounds with metric which is not merely a perturbation of the Minkowski
metric inside some cylinder but can be a perturbation of the
Minkowski metric on the whole spacetime. We show that if the metric
is merely $C^1$ close to the Minkowski metric inside some large cylinder with radius $R$ and approaches the Minkowski
metric outside with some weak rates and if the initial data are sufficiently small,
then the solutions of the quasilinear wave equations satisfying the
null condition are global in time. In particular, the metric does
not necessarily settle down to any particular stationary metric.

\bigskip

Before stating the main theorems, we now introduce some necessary notations.
We use the coordinate system $(t, x)=(t, x_1, x_2, x_3)$
of Minkowski space. We may also use the standard polar local coordinate system $(t, r,
\om)$ and the null coordinates $(u, v, \om)$
\[
 u=\frac{t-r}{2}, \quad v=\frac{t+r}{2}.
\]
Let $\nabb$ denote the induced covariant
 derivative, $\lap$ the induced Laplacian on the spheres of constant $r$, $\Om$ the angular
 momentum with components $\Om_{ij}=x_i\pa_j-x_j\pa_i$. Here $\pa_i$ is the partial derivative $\pa/\pa_{x_i}$.
We may use $\pa$ to abbreviate $(\pa_t,
\pa_{1}, \pa_{2}, \pa_{3}) =(\pa_t, \nabla)$. The vector
fields, used as commutators, are
\[
 Z=\{ \Om_{ij}, \pa_t\}.
\]
We use the convention that Greek indices run from 0 to 3 while the Latin indices run from
1 to 3.

Following the setup in \cite{igorStabvacuum}, we
now introduce a null frame $\{L, \Lb, S_1, S_2\}$, which is locally
a basis of the tangent space at any point $(t, x)$ of the Minkowski
space for $r>0$. We let
\[
L=\pa_v=\pa_t+\pa_r,\quad \Lb=\pa_u=\pa_t-\pa_r.
\]
We then let $S_1$, $S_2$ be an orthonormal basis of the spheres with
constant radius $r$. We use $\overline{\pa_v}$ to denote the ``good'' derivatives
\[
 \overline{\pa_v}=(L, S_1, S_2)=(\pa_v, \nabb).
\]
For any symmetric two tensor $k^{\mu\nu}$, relative to the null frame $\{L,\Lb, S_1, S_2\}$, we have
\[
 k^{\Lb\Lb}=\frac{1}{4}k^{\mu\nu}\Lb_{\mu}\Lb_\nu,\quad \Lb_0=1,\quad \Lb_i=-\frac{x_i}{r}.
\]
In our argument, we estimate the decay of the solution with respect
to the foliation $\Si_{\tau}$, defined as follows:
\begin{align*}
&S_\tau:=\{u=u_\tau, v\geq v_\tau\},\\
&\Si_\tau:=\{t=\tau, r\leq R\}\cup S_\tau,
\end{align*}
where $u_\tau=\frac{\tau-R}{2}$, $v_\tau=\frac{\tau+R}{2}$. The
radius $R$ is a to-be-fixed constant. The corresponding energy flux
is
$$ E[\phi](\tau):=\int_{r\leq R}|\pa\phi|^2dx + \int_{S_\tau}|L\phi|^2+|\nabb\phi|^2\quad r^2dvd\om.
$$

We now give the assumptions on the metric $g$. Relative to the coordinates $(t, x)$, assume
\[
 g^{\mu\nu}=m_0^{\mu\nu}+h^{\mu\nu}, \quad m_0 \textnormal{ the Minkowski metric.}
\]
We assume $h^{\mu\nu}$ are given smooth functions satisfying the following conditions\begin{equation}
\label{HHqu}
\begin{split}
&|h^{\mu\nu}|+|\pa h^{\mu\nu}|\leq \delta_0r_+^{-1-2\a},\quad r=|x|\leq R,\\
 &|\pa h^{\mu\nu}|+|h^{\mu\nu}|+|Z^k h^\mu\nu|\leq \delta_0 (r_+^{-\f12-2\a}\tau_+^{-\f12-\f12\a}+r_+^{-1-2\a}),\quad (t, x)\in S_\tau,\\
&|\overline{\pa_v}h^{\mu\nu}|+|\pa h^{\Lb\Lb}|+|Z ^k h^{\Lb\Lb}|\leq \delta_0 r_+^{-1-2\a},\quad (t, x)\in S_\tau, \quad |k|\leq 6
\end{split}
\end{equation}
for some small positive constant $\a<\frac{1}{10}$ and some large constant $R>4$ (will be the radius of the foliation $\Si_\tau$. Hence in the sequel the foliation
$\Si_\tau$ is fixed). $\delta_0$ will be a small constant depending only on $\a$. Here we denote
\[
 r_+=1+r,\quad \tau_+=1+\tau
\]
and $\tau$ is the parameter of the foliation $\Si_\tau$ which can be defined as $t-\max\{r-R, 0\}$ for all $t\geq r-R$.

Without loss of generality (see the Remark \ref{remarkgendata}), we assume that the initial data
 $(\phi_0, \phi_1)$ are smooth and are supported on $\{|x|\leq R\}$. The initial energy is defined to be
\begin{equation}
\label{defofE}
 E_{0}=\sum\limits_{|k|\leq 5}\int_{\mathbb{R}^{3}}|\nabla Z^{k+1}\phi_0|^2+|Z^k\phi_1|^2+|\nabla\phi_0|^2dx.
\end{equation}
We see that $E_0$ is uniquely determined by the initial data $(\phi_0, \phi_1)$
together with the equation \eqref{QUASIEQ}.

We have the following small
data global existence result for quasilinear wave equations.
\begin{thm}
 \label{maintheorem}
Consider the quasilinear wave equation \eqref{QUASIEQ} satisfy the null condition. Assume that the background metric
$g$ satisfy condition \eqref{HHqu} for small positive constants $\delta_0$, $\a$. Assume the initial data $(\phi_0, \phi_1)$ are smooth and
are support on $\{|x|\leq R\}$. Then there exist $\delta_m>0$, depending only on $\a$, and
$\ep_0>0$, depending on $\a$, $R$, $g$, such that for all $\delta_0<\delta_m$, $E_0<\ep_0$, there exists a unique global smooth solution
  $\phi$ of equation \eqref{QUASIEQ} with the following properties
\begin{itemize}
\item[(1)] Energy decay
$$E[Z^k\phi](\tau)\leq C E_0 (1+\tau)^{-1-\a},\quad |k|\leq 5.$$
\item[(2)] Pointwise decay:
 \begin{align*}
\sum\limits_{|k|\leq 3}|Z^k\phi|&\leq C\sqrt{E_0} (1+r)^{-\f12}(1+|t-r+R|)^{-\f12-\f12\a};\\
\sum\limits_{|k|\leq 2}|\Lb Z^k\phi|+\sum\limits_{|k|\leq 1}|\pa \Lb Z^k\phi|&\leq C_\ep \sqrt{E_0}
(1+r)^{-1+\ep}(1+|t-r+R|)^{-\f12-\frac{\a}{2}};\\
\sum\limits_{|k|\leq 2}|\overline{\pa_v}Z^k\phi|+\sum\limits_{|k|\leq 1}|\pa \overline{\pa_v} Z^k\phi|&\leq C_\ep \sqrt{E_0}
(1+r)^{-\frac{3}{2}+\ep},\quad \ep>0
\end{align*}
\end{itemize}
where the constant $C$ depends only on $R$, $\a$, $g$ and $C_\ep$ also depends on $\ep$.
\end{thm}
We give several remarks
\begin{remark}
A similar result can be obtained in higher dimensions without null condition.
\end{remark}
\begin{remark}
 Inside the cylinder with radius $R$, the null condition on the nonlinearities is not necessary. The nonlinearities can be any quadratic terms of the solution
$\phi$ and its derivatives $\pa\phi$.
\end{remark}

\begin{remark}
As in \cite{yang1}, the smallness assumption ($\delta_0$ appeared in \eqref{HHqu}) on the metric $g$ inside the cylinder $\{(t, x)||x|\leq
R\}$ can be replaced by assuming two integrated local energy estimates. When $r\geq R$, the small constant $\delta_0$ in
the assumption \eqref{HHqu} can be removed as the
smallness can be obtained by choosing $R$ sufficiently large and
shrinking $\a$ to be $\f12\a$.
\end{remark}
\begin{remark}
For simplicity, we merely considered the scalar equations in this paper. However, minor modifications of our approach can also be applied to system of quasilinear wave equations
satisfying the null condition.
\end{remark}

\begin{remark}
\label{remarkgendata}
 It is not necessary to require that the initial data have compact support.
 The general assumption on the initial data can be that the following quantity
\[
 \sum\limits_{|k|\leq 6}\int_{\mathbb{R}^{3}}r^{1+\a}|\pa Z^k\phi(0, x)|^2dx
\]
is sufficiently small. In particular the constant $R$ in the
assumptions on the metric $g$ can be different from the radius of
the support of the initial data. A more general discussion on the initial data will appear in the author's forthcoming paper \cite{yang5}.
\end{remark}
\begin{remark}
 We remark here that the special case when the metric $g$ approaches the Minkowski metric in the spatial directions with a rate $(1+r)^{-1-\ep}$ has been discussed
in the recent work \cite{glsoChengbo}. But in that work there is an extra condition that the metric is static and is independent of time $t$.
\end{remark}

We now apply the above result to the problem of global stability of solutions to quasilinear wave equations initiated by S. Alinhac in \cite{alinhac-sls}. He
studied the quasilinear wave equations
\begin{equation}
\label{quasilinearstab}
\begin{cases}
\Box w+g^{\mu\nu\gamma}\pa_{\gamma}w\cdot \pa_{\mu\nu}w=0,\\
w(0,x)=\Phi_0(x)+\ep\phi_0, \quad \pa_t w(0,x)= \Phi_1(x)+\ep\phi_1
\end{cases}
\end{equation}
on Minkowski space, where $g^{\a\b\gamma}$ are constants satisfying
the null condition. Suppose $\Phi(t, x)$ is a smooth global solution of the above equation when $\ep=0$. He showed that if $\Phi$ satisfies the condition
\begin{equation}
\label{alinhaccond}
|g^{ij\gamma}\pa_{\gamma}\Phi\cdot\xi_i\xi_j|\leq \a_0
\sum\limits_{i=1}^{3}|\xi_i|^2,\quad \sum\limits_{|k|\leq
7}|\Gamma^k\pa \Phi|\leq C_0 (1+t)^{-1}(1+|r-t|)^{-\f12}
\end{equation}
for some positive constants $\a_0<1$ and $C_0$, then the solution of
the above quasilinear wave equation \eqref{quasilinearstab} exists
globally in time for all sufficiently small $\ep$. Here $\Gamma$
denotes the collection of vector fields given in line
\eqref{Lorenzinv} except the conformal killing vector field $K$.

We give weaker conditions than \eqref{alinhaccond} on the solution $\Phi$ to guarantee the global stability. We assume the initial data
$(\Phi_0, \Phi_1, \phi_0, \phi_1)$ are smooth and are supported on $\{r\leq R_0\}$ for some large constant $R_0$. Let $\Phi$ be a smooth solution of \eqref{quasilinearstab} when $\ep=0$.
Before some large time $t_0$, we assume the metric $m_0+g^{\mu\nu\ga}\pa_\ga \Phi$ is hyperbolic and
\begin{equation}
 \label{stacondw0}
|\pa^2\Phi|+ |Z^k \pa\Phi|\leq C_1,\quad t\leq t_0,\quad |k|\leq 6
\end{equation}
for some constant $C_1$. After time $t_0$, we assume $\Phi$ satisfies the following weak decay estimates
\begin{equation}
 \label{stabcondweak}
\begin{split}
&|\pa^2\Phi|+|Z^k\pa\Phi|\leq \delta_0 (r_+^{-\f12-2\a}\tau_+^{-\f12-2\a}+r_+^{-1-2\a}),\quad (t, x)\in S_\tau ,\quad t\geq t_0;\\
&|\pa\overline{\pa_v}\Phi|+|Z^k\overline{\pa_v}\Phi|\leq \delta_0r_+^{-1-2\a},\quad (t,x )\in S_\tau,\quad t\geq t_0;\\
&|\pa^2\Phi|+|\pa Z^k\phi|\leq \delta_0 (1+r)^{-1-2\a},\quad r\leq R,\quad t\geq t_0,\quad \forall k\leq 6,
\end{split}
\end{equation}
where $R=R_0+t_0$ is chosen to be radius of the foliation $\Si_\tau$. We let $E_0$ be the initial
energy for $(\ep\phi_0, \ep\phi_1)$ defined in \eqref{defofE}.

We have the following global stability of solutions to quasilinear wave equations.
\begin{thm}
 \label{stabquasi}
Assume the constants $g^{\mu\nu\ga}$ in \eqref{quasilinearstab}
satisfy the null condition. Assume the initial data $(\Phi_0,
\Phi_1, \phi_0, \phi_1)$ are smooth and are supported on $\{r\leq
R_0\}$ for some large constant $R_0$. Let $\Phi$ be a smooth solution
of \eqref{quasilinearstab} when $\ep=0$ satisfying the conditions \eqref{stacondw0},
\eqref{stabcondweak}. Then there exist two small positive constants $\delta_m>0$, depending on
$\a$, and $\ep_0>0$, depending on $\a$, $R_0$, $t_0$, $C_1$
such that for all $\delta_0<\delta_m$, $E_0<\ep_0$, there exists a
unique global smooth solution
  $w$ of equation \eqref{quasilinearstab} with the property that for the foliation $\Si_\tau$
with radius $R_0+t_0$, the difference $\phi=w-\Phi$ satisfies the same estimates as given in Theorem \ref{maintheorem} but with the constant $C$ depending
on $\a$, $R_0$, $t_0$, $C_1$.
\end{thm}

\begin{remark}
The problem of global stability of solutions to semilinear wave
equations has been discussed in \cite{yang2}.
\end{remark}
Compared to the condition \eqref{alinhaccond} imposed in
\cite{alinhac-sls}, we do not require the given solution $\Phi$ to
decay in time $t$ uniformly. In fact, we can even allow $\Phi$ to be
independent of $t$ in the cylinder $\{(t, x)|r\leq R\}$. Moreover,
since the initial data are supported on $\{r\leq R_0\}$, the finite
speed of propagation for wave equations shows that the solution
$\Phi$ vanishes when $r\geq t+R_0$. Hence condition
\eqref{alinhaccond} implies \eqref{stabcondweak}. Finally, the
collection of vector fields $Z$ used in the condition
\eqref{stabcondweak} is a subset of the collection $\Ga$ in
\eqref{alinhaccond}. In particular, we avoid the use of the scaling
vector field $S$ or the Lorentz rotations $L_i$ which grow in time $t$.

\bigskip

Our argument relies on a new method developed by Dafermos-Rodnianski in \cite{newapp}. Based on an integrated local energy
inequality, which is usually obtained by using the vector fields
$\pa_t$, $f\pa_r$, where $f$ is some appropriate function of $r$, as
multipliers, and a $p$-weighted energy inequality in a neighborhood of
the null infinity, the new approach leads to the decay, in
particular, of the energy flux $E[\phi](\tau)$ for solutions of
linear wave equations. The integrated local energy inequality has been
well studied on various backgrounds, including black hole spacetimes,
see e.g. \cite{newapp3} and references mentioned above. When the metric
is flat in a neighborhood of the null infinity, the $p$-weighted
energy inequality can be derived by multiplying the equation with
$r^p(\pa_t+\pa_r)(r\phi)$ and then integrating by parts. This is the
situation in \cite{yang1} as there the metric $g$ is identical to
the Minkowski metric when $r\geq R$. For the backgrounds considered
in this paper, the metric is merely asymptotically flat in the
spatial directions. As mentioned in the original work \cite{newapp}
of Dafermos-Rodnianski, a much more flexible and robust
way to derive the $p$-weighted energy inequality is to use the vector
fields $r^p(\pa_t+\pa_r)$ as multipliers. However, these vector
fields can not be applied directly to asymptotically flat backgrounds. We may need to modify the
vector fields as
\[
r^{p}(-2\pa^{\Lb}+ g^{\Lb\Lb}\Lb),
\]
see details in Section 3.2.

\bigskip

Nevertheless, for the general metrics $g$ in this paper, there is
another difficulty arising from the error terms on the boundary
$S_\tau$. Those error terms can not be controlled without losing any
derivatives. We hence are not able to show the decay of the energy
flux $E[\phi](\tau)$ as we did in \cite{yang1}. However, using the
boundedness of the integrated energy on the whole spacetime together
with a pigeon hole argument, we still can show the decay of the
integrated energy on the region bounded by $\Si_{\tau_1}$ and
$\Si_{\tau_{2}}$, see details in Section 4. The pointwise decay of
the solution then follows by commuting the equation with the vector
fields $\pa_t$, $\Om_{ij}$.

\bigskip

The plan of this paper is as follows: we will review the energy
method for wave equations and define the notations in Section 2. In
Section 3, we establish an integrated local energy inequality on the
region bounded by $\Si_{\tau_1}$ and $\Si_{\tau_{2}}$ and two
$p$-weighted energy
 inequalities. In Section 4, we show the decay of the integrated energy for solutions of linear wave
 equations. In the last section, we use bootstrap argument to prove
 the main theorems.

\textbf{Acknowledgements} The author is deeply indebted to his advisor
Igor Rodnianski for his continuous support on this problem. He
thanks Igor Rodnianski for sharing numerous valuable thoughts. The
author would also like to thank Beijing International Center for
Mathematical Research for the wonderful hospitality when he was
visiting there and part of this work was carried out there.

\section{Preliminaries and Energy Method }

 Given any Lorentzian metric
\[
g=g_{\mu\nu}dx^{\mu}dx^\nu,\quad x^0=t
\]
on $\mathbb{R}^{3+1}$, we let $g^{\mu\nu}$ denotes the components of
the inverse of the metric $g$. Throughout this paper, we let $A$,
$B$ be any vector fields in $\{L, \Lb, S_1, S_2\}$ and  $S$ be any
vector fields in $\{S_1, S_2\}$. Relative to the null frame, the
metric components are $g_{AB}$. The inverse is $g^{AB}$. We denote
\[
\pa^\mu=g^{\mu\nu}\pa_\nu,\quad \pa^A=g^{AB}B.
\]
At any fixed point $(t, x)$, we may choose $S_1$, $S_2$ such that
\begin{equation}
\label{com} [L, S]=-\frac{1}{r}S,\quad [\Lb, S]=\frac{1}{r}S,\quad
\left.[S_1, S_2]\right|_{(t, x)}=0,\quad S\in\{S_1, S_2\}
\end{equation}
This helps to compute those geometric quantities which are
independent of the choice of the local coordinates. Denote the incoming null hypersurface
\[
 \bar C(\tau_1, \tau_2, v):=\{(t, r, \om)| \tau_1\leq t-r+R\leq \tau_2, \quad t+r=2v\}.
\]
We simply use $\bar C(\tau_1, \tau_2)$ to denote the future null infinity (part of) where $v=\infty$. Define the energy flux through the null infinity as the
limit infimum of the energy flux through $\bar C(\tau_1, \tau_2, v)$ as $v\rightarrow \infty$, that is,
\[
E^N[\phi]_{\tau_1}^{\tau_2}:=\liminf\limits_{v\rightarrow\infty}\int_{\bar C(\tau_1, \tau_2, v)}(\pa_u\phi)^2+|\nabb\phi|^2\quad r^2dud\om.
\]
We define the modified energy flux
\[
 \tilde{E}[\phi](\tau):=E[\phi](\tau)+E^N[\phi]_{0}^{\tau}.
\]

\bigskip

We now review the energy method for wave equations. For a Lorentzian
space $(\mathbb{R}^{3+1}, g)$ with metric $g$, we denote $d\vol$ the
volume form. In the local coordinate system $(t, x)$, we have
\[
d\vol=\sqrt{-G}dxdt,\quad G=\det(g_{\mu\nu}).
\]
Here we have chosen $t$ to be the time orientation for the Lorentzian space $(\mathbb{R}^{3+1}, g)$.
We recall the energy-momentum tensor of the scalar field $\phi$ on the Lorentzian space $(\mathbb{R}^{3+1}, g)$ with metric $g$
\[
{\mathbb T}_{\mu\nu}[\phi]=\pa_\mu\phi\pa_\nu\phi-\frac12 g_{\mu\nu}\pa^{\gamma}\phi\pa_{\gamma}\phi.
\]
Throughout this paper, we raise and lower indices of any tensor relative to the given
metric $g$, e.g., $\pa^\ga=g^{\ga\mu}\pa_\mu$.
Given a vector field $X$, we define the currents
\[
J^X_\mu[\phi]= {\mathbb T}_{\mu\nu}[\phi]X^\nu, \qquad
K^X[\phi]= {\mathbb T}^{\mu\nu}[\phi]\pi^X_{\mu\nu},
\]
where $\pi^X_{\mu\nu}=\frac12 \mathcal{L}_Xg_{\mu\nu}$ is the deformation tensor of the vector field $X$. We denote
$J^X[\phi]$ as the vector field
\[
 J^X[\phi]=J_{\mu}^{X}[\phi]g^{\mu\nu}\pa_\nu.
\]
Recall that
\[
D^\mu J^X_\mu[\phi] = X(\phi)\Box_g\phi+K^X[\phi],
\]
where $\Box_g$ is the covariant wave operator and $D$ is the covariant derivative of the metric $g$.

Take any function $\chi$. We have the following identity
\begin{align*}
 D^{\mu}\left(-\f12\pa_{\mu}\chi\cdot \phi^2 +\f12 \chi\pa_{\mu}\phi^2\right)= \chi\phi\Box_g\phi+\chi\pa^{\gamma}\phi\pa_{\gamma}\phi -\f12\Box_g\chi\cdot\phi^2.
\end{align*}
Modify the vector field $J^X[\phi]$ to be
\begin{equation}
\label{mcurent} \tilde{J}^X[\phi]=\tilde{J}_{\mu}^X[\phi]g^{\mu\nu}\pa_\nu=\left(J_{\mu}^X[\phi] -
\f12\pa_{\mu}\chi \cdot\phi^2 + \f12 \chi\pa_{\mu}\phi^2\right)g^{\mu\nu}\pa_\nu.
\end{equation}
We then have the identity
\[
D^\mu \tilde{J}^X_\mu[\phi] = (X(\phi)+\chi\phi)\Box_g\phi+K^X[\phi]+\chi\pa^{\gamma}\phi\pa_{\gamma}\phi -\f12\Box_g\chi\cdot\phi^2.
\]
For any bounded region $\mathcal{D}$ in $\mathbb{R}^{3+1}$, using Stoke's formula, we have the following energy identity
\begin{align}
\notag \iint_{\mathcal{D}}D^\mu \tilde{J}^X_\mu[\phi]d\vol&=\iint_{\mathcal{D}}\Box_g\phi
(\chi\phi+X(\phi))
 + K^X[\phi] +\chi\pa^\ga\phi\pa_\ga\phi-\f12 \Box_g\chi \cdot \phi^2d\vol\\
&=\int_{\pa \mathcal{D}}i_{\tilde{J}^X[\phi]}d\vol,
\label{energyeq}
\end{align}
$\pa\mathcal{D}$ denotes the boundary of the domain $\mathcal{D}$ and $i_Y d\vol$ denotes the contraction of the volume form $d\vol$
with the vector field $Y$ which gives the surface measure of the
boundary. For example, for any basis $\{e_1, e_2,
\ldots, e_n\}$, we have $i_{e_1}( de_1\wedge de_2\wedge\ldots
de_k)=de_2\wedge de_3\wedge\ldots\wedge de_k$. Here we have chosen $t$ to be the time orientation. For more details on this formula, we refer to the appendix of \cite{stefanos}.

Throughout this paper, the domain $\mathcal{D}$ will be regular regions bounded by the $t$-constant slices, the outgoing null hypersurfaces $S_\tau$ or the incoming null
hypersurfaces $\bar C(\tau_1, \tau_2, v)$. We now compute $i_{\tilde{J}^{X}[\phi]}d\vol$ on these three kinds of hypersurfaces.
We now compute $i_{\tilde{J}^{X}[\phi]}d\vol$ on $\Si_\tau$ or on the $v$-constant incoming null hypersurfaces (with respect to the Minkowski metric).
We have the following three cases.

On
$\Si_\tau\cap \{r\leq R\}$, the surface measure is a function times $dx$. Recall the volume form
\[
d\vol=\sqrt{-G}dxdt=-\sqrt{-G}dtdx.
\]
Here note that $dx$ is a $3$-form. We thus can show that
\begin{equation}
\label{curlessR}
\begin{split}
 i_{\tilde{J}^{X}[\phi]}d\vol&=i_{(\tilde{J}^{X}[\phi])^\mu\pa_\mu}d\vol=-(\tilde{J}^{X}[\phi])^0\sqrt{-G}dx\\
 &=-(\pa^t\phi
X(\phi)-\f12 X^0\pa^\ga\phi\pa_\ga\phi-\f12 \pa^t\chi \cdot
\phi^2+\chi\pa^t\phi\cdot \phi)\sqrt{-G}dx.
\end{split}
\end{equation}
On the null hypersurface $S_\tau$ with respect to the Minkowski metric, we can write the volume form
\[
d\vol=\sqrt{-G}dxdt=\sqrt{-G}r^2drdt d\om=2\sqrt{-G}r^2dvdud\om=-2\sqrt{-G}dudvd\om.
\]
Here $u=\frac{t-r}{2}$, $v=\frac{t+r}{2}$ are the null coordinates.
Notice that $\Lb=\pa_u$. We can compute
\begin{equation}
\label{curStau}
i_{\tilde{J}^{X}[\phi]}d\vol=-2\sqrt{-G}r^2(\pa^{\Lb}\phi
X(\phi)-\f12 X^{\Lb}\pa^\ga\phi\pa_\ga\phi-\f12
\pa^{\Lb}\chi\phi^2+\chi \pa^{\Lb}\phi\cdot \phi)dvd\om.
\end{equation}
Similarly, on the $v$-constant incoming null hypersurfaces $\{v=constant\}$, we
have
\begin{equation}
\label{curnullinfy} i_{\tilde{J}^{X}[\phi]}d\vol=2\sqrt{-G}r^2(\pa^L\phi X(\phi)-\f12 X^L\pa^\ga\phi\pa_\ga\phi-\f12
\pa^L\chi \phi^2+\chi \pa^L\phi\cdot\phi)dud\om.
\end{equation}
We remark here that the above three formulae hold for any vector field $X$ and any function $\chi$.

\bigskip

The following several lemmas, which have been proven in
\cite{yang1}, will be used later on.
\begin{lem}
\label{lem1} Let $\phi$ be a smooth function on $\mathbb{R}^{3+1}$. Assume
\[
\lim\limits_{v\rightarrow \infty}\phi(u_0, v, \om)=0,\quad u_0=-\f12
R.
\]
Then in the polar coordinates $(t, r, \om)$, we have
$$r\int_{\om}|\phi(r+\tau-R, r, \om)|^2 d\om\leq 4\tilde{E}[\phi](\tau).$$
Moreover, if $E^N[\phi]_0^\tau<\infty$, then
\[
r\int_{\om}|\phi(r+\tau-R, r, \om)|^2 d\om\leq 4E[\phi](\tau).
\]
\end{lem}

For solutions of linear wave equations, the good derivative $\pa_v$ of the solution decays better. In that case, we have
\begin{lem}
\label{lempphi2} Let $\a_1, \a_2>0$. Assume $\phi$ satisfies the condition in Lemma \ref{lem1}.
Then we have
\begin{equation*}
\int_{S_\tau}r^{1-\a_1}\phi^2dvd\om\leq C_0 R^{1-\a_1} \tilde{E}[\phi](\tau)+
C_0\int_{S_\tau}r^{1+\a_2} |\pa_v(r\phi)|^2dvd\om ,
\end{equation*}
where $C_0$ is a constant depending only on $\a_1$, $\a_2$.
\end{lem}
\begin{proof} Let $\psi=r\phi$. By Lemma \ref{lem1}, we have
\begin{equation*}
\label{phi2bd}
\begin{split}
\int_{\om}|\psi|^2(\tau,v,\om)d\om&\leq C_0\int_{\om}|\psi|^2(\tau, v_\tau, \om)d\om +C_0\left(\int_{v_\tau}^v\int_{\om}|\pa_v\psi|d\om dv\right)^2\\
 &\leq C_0 R \tilde{E}[\phi](\tau) +  C_0\int_{v_\tau}^v\int_{\om}r^{1+\a_2}|\pa_v\psi|^2d\om dv\int_{v_\tau}^v r^{-1-\a_2}dv\\
&\leq C_0 R \tilde{E}[\phi](\tau)+ C_0\frac{R^{-\a_2}-r^{-\a_2}}{\a_2}\int_{S_\tau}r^{1+\a_2}|\pa_v\psi|^2d\om dv
\end{split}
\end{equation*}
Multiply the above inequality by $r^{-1-\a_1}$ and then integrate
from $v_\tau=\frac{\tau+R}{2}$ to infinity. We obtain
\begin{align*}
\int_{S_\tau}r^{1-\a_1}\phi^2dvd\om
&=\int_{v_\tau}^{\infty}r^{-1-\a_1}\int_{\om}|\psi|^2d\om dv \\
&\leq C_0 R^{1-\a_1} \tilde{E}[\phi](\tau)+ C_0
\frac{R^{-\a_1-\a_2}}{\a_1(\a_1+\a_2)}\int_{S_\tau}r^{1+\a_2}|\pa_v\psi|^2d\om dv\\
&\leq C_0 R^{1-\a_1} \tilde{E}[\phi](\tau)+ C_0 \int_{S_\tau}r^{1+\a_2}|\pa_v\psi|^2d\om dv.
\end{align*}
\end{proof}
We will also frequently use the following simple lemma.
\begin{lem}
 \label{lweightILE}
Suppose $f(\tau)$ is smooth. Then we have the identity
\begin{equation*}
\int_{\tau_1}^{\tau_2}(1+s)^\b f(s)ds=\b\int_{\tau_1}^{\tau_2}(1+\tau)^{\b-1}\int_{\tau}^{\tau_2}f(s)ds
d\tau+(1+\tau_1)^{\b}\int_{\tau_1}^{\tau_2}f(s)ds
\end{equation*}
for $\forall \b\in\mathbb{R}$.
\end{lem}
\begin{proof}
 Let
\[
 F(\tau)=\int_{\tau}^{\tau_2}f(s)ds.
\]
Integration by parts gives the lemma.
\end{proof}

We also need the following analogue of Hardy's inequality to control $\phi$ by the energy.
\begin{lem}
\label{lem2} Let $\phi$ satisfy the same conditions as in the
previous lemma. Then
\begin{equation*}
\int_{\{r\leq
R\}\cap\Si_\tau}\left(\frac{\phi}{1+r}\right)^2dx+\int_{S_{\tau}}\left(\frac{\phi}{1+r}\right)^2r^2dvd\om
\leq 12 \tilde{E}[\phi](\tau).
\end{equation*}
In particular
\begin{equation*}
\int_{r\leq R}\phi^2dx\leq 12(1+R)^2 \tilde{E}[\phi](\tau).
\end{equation*}
Here we simply use $r\leq R$ to denote the integral region $\{r\leq
R\}\cap \Si_\tau$.
\end{lem}

\begin{remark}
\label{remark1}
 By Lemma \ref{lem1}, if $E^N[\phi]_0^\tau$ is finite, then all the above statements hold if we replace $\tilde{E}[\phi](\tau)$ with $E[\phi](\tau)$.
\end{remark}

Finally, for $\forall \a,\quad  p\geq0,\quad \b\in \mathbb{R}$, we define several notations:
\begin{align*}
 &I^\a[\phi]_{\tau_1}^{\tau_2}:=\int_{\tau_1}^{\tau_2}\int_{\Si_\tau}\frac{|\bar\pa\phi|^2}{(1+r)^{1+\a}}dxd\tau,
\quad S^\a[\phi](\tau):=\int_{S_\tau}\frac{|\pa\phi|^2}{(1+r)^{1+\a}}r^2 dvd\om,\\
& D^\a[F]_{\tau_1}^{\tau_2}:=\int_{\tau_1}^{\tau_2}\int_{\Si_\tau}(1+r)^{1+\a}|F|^2 dxd\tau, \quad g^p[\phi](\tau):=\int_{S_\tau}r^p|\pa_v\psi|^2dvd\om,\\
&G^{p, \b}[\phi]_{\tau_2}^{\tau_2}:=\int_{\tau_1}^{\tau_2}\tau_+^{-\b}\int_{S_\tau}r^p|\pa_v\psi|^2dvd\om d\tau,\quad
E^\b[\phi]_{\tau_1}^{\tau_2}:=\int_{\tau_1}^{\tau_2}\tau_+^{-\b}E[\phi](\tau)d\tau\\
& \bar G^{p, \b}[\phi]_{\tau_2}^{\tau_2}:=\int_{\tau_1}^{\tau_2}\tau_+^{-\b}\int_{S_\tau}r^p|\overline{\pa_v}\psi|^2dvd\om d\tau, \quad \bar g^{p}[\phi](\tau):=\int_{S_\tau}r^p|\overline{\pa_v}\psi|^2dvd\om,
\end{align*}
where $\psi=r\phi$. Similarly, we have the notation for $\tilde{E}^\b[\phi]_{\tau_1}^{\tau_2}$. We remark here that this notation is different from $E^N[\phi]_{\tau_1}^{\tau_2}$ which
is the energy flux through the null infinity.

\section{Weighted Energy Estimates}
Our approach relies on two estimates: integrated local energy
inequality and $p$-weighted energy inequality. In this section, we use
the multiplier method to establish an integrated energy
inequality and two $p$-weighted energy inequalities for quasilinear
wave equations. The integrated energy inequality was first
proven by C. Morawetz in \cite{mora2}. We follow the method developed in
\cite{dr3} to obtain the integrated energy inequality here. In
\cite{newapp}, Dafermos-Rodnianski introduced the
$p$-weighted energy inequalities in a neighborhood of null infinity.
These estimates have been established in \cite{yang2} for semilinear
wave equation. As mentioned in the original work of Dafermos-Rodnianski,
we can use the robust multiplier method to show the $p$-weighted energy inequality on general backgrounds.

\bigskip

In this section, we prove a general integrated energy inequality for solutions of the linear wave equations
\begin{equation}
 \label{LWAVEEQ}
 \Box_g \phi+N(\phi)=F
\end{equation}
on the Lorentzian manifold $(\mathbb{R}^{3+1}, g)$. Here $N(\phi)=N^\mu\pa_\mu\phi$ is a linear term and $N$ is a vector field on $\mathbb{R}^{3+1}$ with components $N^\mu$.

Fix a large constant $R>8$ so that we can determine the foliation $\Si_\tau$ with radius $R$. Recall that $g=h+m_0$. We assume $h^{\mu\nu}$, $N^\mu$ satisfy the following weak decay estimates
\begin{equation}
\label{HH}
\begin{split}
&|h^{\mu\nu}|+|\pa h^{\mu\nu}|+|N^\mu|\leq \delta_0r_+^{-1-2\a},\quad r=|x|\leq R,\\
 &|\pa h^{\mu\nu}|+|h^{\mu\nu}|+|N^\mu|\leq \delta_0 (r_+^{-\f12-2\a}\tau_+^{-\f12-\f12\a}+r_+^{-1-2\a}),\quad (t, x)\in S_\tau,\\
&|\overline{\pa_v}h^{\mu\nu}|+|\pa h^{\Lb\Lb}|+|h^{\Lb\Lb}|+|N^{\Lb}|\leq \delta_0 r_+^{-1-2\a},\quad (t, x)\in S_\tau,
\end{split}
\end{equation}
where $\delta_0$, $\a$ are positive constants and $\a<\frac{1}{10}$. Here we recall that $r_+=1+r$, $\tau_+=1+\tau$; $h^{\Lb\Lb}$ is the component of $h$ with
respect to the null frame $\{L, \Lb, S_1, S_2\}$.

\subsection{The integrated local energy estimates}

We establish the following key estimates for solutions of linear wave
equations.
\begin{prop}
\label{ILEthm} Assume that the given metric $g$ satisfy the above estimates \eqref{HH} for some positive constant $\a$, $\delta_0$.
Let $\phi$ be a smooth solution of the linear wave equation \eqref{LWAVEEQ} and satisfy the conditions in Lemma \ref{lem1}.
If $\delta_0$ is sufficiently small depending only on $\a$ then for $\forall
\tau_1\leq \tau_2$ we have the boundedness of the integrated energy
\begin{equation}
\label{boundILE} I^\a[\phi]_0^\infty\les \tilde{E}[\phi](0)+\delta_0 S^\a[\phi](0)+D^\a[F]_0^\infty.
\end{equation}
If in addition we have
\[
I^\a[\phi]_0^{\infty}=\int_{0}^{\infty}\int_{\Si_\tau}\frac{|\bar\pa\phi|^2}{(1+r)^{1+\a}}dxd\tau<\infty,
\]
then
\begin{itemize}
\item[(1)] Integrated energy estimate
\begin{equation}
 \label{ILE0}
I^\a[\phi]_{\tau_1}^{\tau_2}+\int_{\tau_1}^{\tau_2}\int_{S_\tau}
\frac{|\nabb\phi|^2}{1+r}dxd\tau \leq C_{\a}(
\tilde{E}[\phi](\tau_1)+\delta_0 S^\a[\phi](\tau_i)+D^{\a}[F]_{\tau_1}^{\tau_2});
\end{equation}
\item[(2)] Energy bound
\begin{equation}
 \label{eb}
\tilde{E}[\phi](\tau_2)+E^N[\phi]_{\tau_1}^{\tau_2}\leq C_{\a}(
\tilde{E}[\phi](\tau_1)+
\delta_0 S^\a[\phi](\tau_i)+D^{\a}[F]_{\tau_1}^{\tau_2}),
\end{equation}
\end{itemize}
where $S^\a[\phi](\tau_i)=S^\a[\phi](\tau_1)+S^\a[\phi](\tau_2)$. The constant $C_{\a}$ depends only on $\a$.
The definitions for $I^\a[\phi]_{\tau_1}^{\tau_2}$, $S^\a[\phi](\tau)$,
$D^\a[F]_{\tau_1}^{\tau_2}$ can be found in the end of the previous section.
\end{prop}

To prove \eqref{ILE0} and \eqref{eb}, we need a priori asymptotical estimate for the solution, i.e., in this proposition, we assume the
integrated energy $I^\a[\phi]_0^{\infty}$ is finite. The inequality \eqref{boundILE} will be used to verify this condition with appropriate
initial condition and some boundedness of the inhomogeneous term $F$.
\begin{remark}
We mention here that variants and generalizations of estimate
\eqref{ILE0} can also be found in \cite{sogge-metcalfe2},
\cite{sogge-metcalfe} and reference therein. However, the conditions on the given metric $g$ here is more general and the foliations used here are different.
\end{remark}
We use the vector field method to prove the above proposition. More precisely, we construct vector fields $X=f\pa_r$, $f$ is a function of $r$. Using the energy identities
\eqref{energyeq} applied to the region bounded by $\Si_{\tau_1}$, $\Si_{\tau_2}$, we are able to derive the integrated energy estimates as well as the energy estimates.

The proof for the integrated energy estimate \eqref{ILE0} and the energy estimate \eqref{eb} is modification of that in \cite{yang1} for wave equation on curved background. The only
difference is to control the boundary terms on $S_\tau$. For
completeness, we roughly repeat the proof here.

To avoid to many constant, in this section, we make the convention that $A\les B$ means $A\leq CB$ for some constant $C$ depending only on $\a$.

\subsubsection{The vector field $f\pa_r$}
\label{section311}
Let $v>\frac{\tau_1+R}{2}$. Consider the region $\mathcal{D}$ bounded by $\Si_{\tau_1}$, $\Si_{\tau_2}$ and the incoming null hypersurface $\bar C(\tau_1, \tau_2, v)$. Let
\[
\Si_\tau^v:=\Si_\tau\cap\{t+r\leq 2v\}.
\]
Take the vector field
$X$ as follows
$$X=f\pa_r$$ for some function $f$ of $r$ such that $f(0)=0$. Hence $X$ is a well defined vector filed on $\mathbb{R}^{3+1}$.
Thus in the energy inequality \eqref{energyeq}, we can compute the current
$K^X[\phi]$
\begin{align*}
K^X[\phi]=\mathbb{T}^{\mu\nu}[\phi]\pi^X_{\mu\nu}&=\pa_j(f\frac{x_i}{r})\pa^{j}\phi\cdot \pa_i \phi-(\f12 f'+r^{-1}f)\pa^{\gamma}\phi \pa_{\gamma}\phi\\
& \quad-\f12 f\pa_r g^{\mu\nu}\cdot \pa_{\mu}\phi
\pa_{\nu}\phi+\frac{1}{4}f \pa_{r}g^{\mu\nu}\cdot
g_{\mu\nu}\pa^{\gamma}\phi\pa_{\gamma}\phi,
\end{align*}
where we denote $f'$ as $\pa_r f$.

Next we choose the function $\chi$ in the modified vector field
\eqref{mcurent} to be
\[
 \chi=r^{-1}f.
\]
Then from the energy identity \eqref{energyeq}, we obtain
\begin{align}
\label{menergyeq}
&\int_{{\Sigma}_{\tau_1}^v}i_{\tilde{J}^X[\phi]}d\vol -
\int_{{\Sigma}_{\tau_2}^v}i_{\tilde{J}^X[\phi]}d\vol+
\int_{\bar C(\tau_1, \tau_2, v)}i_{\tilde{J}^X[\phi]}d\vol\\
\notag&=\int_{\tau_1}^{\tau_2}\int_{\Sigma_\tau^v} \Box_g
\phi(X(\phi)+\phi\chi)+\f12 f'(|\pa_r\phi|^2+|\pa_t\phi|^2)
+(\chi-\f12 f')|\nabb\phi|^2- \f12\Box_g\chi\cdot\phi^2 +E(X)d\vol,
\end{align}
where the error term $E(X)$ is given as follows
\begin{equation}
 \label{error}
\begin{split}
E(X) = &\pa_j(f\frac{x_i}{r})h^{j\mu}\pa_{\mu}\phi\cdot \pa_i
\phi-\f12 f'h^{\mu\nu}\pa_{\mu}\phi \pa_{\nu}\phi-\f12 f\pa_r
g^{\mu\nu}\cdot \pa_{\mu}\phi \pa_{\nu}\phi+\frac{1}{4}f\pa_{r}
g^{\mu\nu}\cdot g_{\mu\nu}\pa^{\gamma}\phi\pa_{\gamma}\phi.
\end{split}
\end{equation}
We now explicitly
construct the function $f$ as follows
$$f=2\a^{-1}-\frac{2\a^{-1}}{(1+r)^{\a}},\quad \chi=r^{-1}f.$$ We have
\begin{align*}
 \chi-r^{-1}f + \f12 f'=r^{-1}f + \f12 f' - \chi=\frac{1}{(1+r)^{\alpha +1}}
\end{align*}
In particular, when $r\geq R>8$, we have the following improved estimate for
$\chi-\f12 f'$
\begin{equation}\label{improvnabb}\chi- \f12
f'\geq \frac{\b}{r} - \frac{1+\beta}{r(1+r)^\a}\geq \frac{1}{r},
\quad r\geq R.
\end{equation}
This improved estimate will be used to show the improved integrated
energy estimate \eqref{ILE0} for the angular derivative of the
solution. We can estimate
\begin{align*}
 |(\Box_g-\Delta) \chi|&=|h^{ij}\pa_{ij}\chi+(\pa_\mu g^{\mu i}+\f12 g^{\mu i}\pa_{\mu}g_{\nu\gamma}\cdot g^{\nu\gamma})\pa_i\chi|\les \frac{|h|+|\pa h|}{r(1+r)},
\end{align*}
where $\Delta$ is the Laplacian operator on $\mathbb{R}^3$. From the assumption \eqref{HH}, we have
\begin{equation*}
|(\Box_g-\Delta) \chi|\les
 \begin{cases}
 \frac{\delta_0}{r(1+r)^{2+\a}},\quad r\leq R;\\
\frac{\delta_0 r_+^{-\f12-\a}\tau_+^{-\f12-\a}}{(1+r)^2}\leq \frac{\delta_0}{(1+r)^2}(r_+^{-1-\a}+\tau_+^{-1-\a}),\quad (t, x)\in S_\tau.
 \end{cases}
\end{equation*}
Using Lemma \ref{lem2} to control the integral of $\frac{\phi^2}{(1+r)^2}$, we can show that
\begin{align*}
 \int_{\tau_1}^{\tau_2}\int_{\Sigma_\tau^v}|&(\Box_g-\Delta)\chi|\phi^2 d\vol\les \delta_0\int_{\tau_1}^{\tau_2}\int_{r\leq R}
\frac{|\phi|^2}{r(1+r)^{2+\a}}dxd\tau\\&+\delta_0\int_{\tau_1}^{\tau_2}\int_{\Sigma_\tau}\frac{|\pa\phi|^2}{(1+r)^{1+\a}}dxd\tau+\delta_0\int_{\tau_1}^{\tau_2}\tau_+^{-1-\a}\tilde{E}[\phi](\tau)d\tau.
\end{align*}
Recall that $-\Delta\chi=\frac{2(1+\a)}{r(1+r)^{2+\a}}$. Then from the above energy inequality \eqref{menergyeq}, we obtain
\begin{align}
\label{menergyeq1} &\int_{\tau_1}^{\tau_2}\int_{\Sigma_\tau^v}
\frac{|\bar\pa\phi|^2}{(1+r)^{1+\a}} +E(X)dx d\tau\les|\int_{{\Sigma}_{\tau_1}^v}i_{\tilde{J}^X[\phi]}d\vol-\int_{\bar C(0,\tau_1, v)}i_{\tilde{J}^X[\phi]}d\vol|\\
\notag
&\quad+|\int_{{\Sigma}_{\tau_2}^v}i_{\tilde{J}^X[\phi]}d\vol-\int_{\bar C(0, \tau_2, v)}i_{\tilde{J}^X[\phi]}d\vol|+
\int_{\tau_1}^{\tau_2}\int_{\Sigma_\tau^v}
|\Box_g\phi||X(\phi)+\chi\phi|dx d\tau\\
\notag
&\quad+\delta_0\int_{\tau_1}^{\tau_2}\tau_+^{-1-\a}\tilde{E}[\phi](\tau)d\tau+\int_{\tau_1}^{\tau_2}\int_{r\leq R}
\frac{\delta_0|\bar\pa\phi|^2}{(1+r)^{1+\a}}dxd\tau+\int_{\tau_1}^{\tau_2}\int_{\Sigma_\tau}\frac{\delta_0|\pa\phi|^2}{(1+r)^{1+\a}}dxd\tau.
\end{align}
After taking the limit $v\rightarrow\infty$, for sufficiently small $\delta_0$ the last two terms in the last line will be absorbed.
We will later show that integral of the error term $E(X)$ can be absorbed. We first demonstrate that the integral on
the boundary can be bounded by the energy $E[\phi](\tau)$.
\begin{lem}
\label{prop1}
We have
\begin{equation*}
\liminf\limits_{v\rightarrow \infty}\left|\int_{{\Sigma}_{\tau}^v}i_{\tilde{J}^X[\phi]}d\vol-\int_{\bar C(0,\tau, v)}
i_{\tilde{J}^X[\phi]}d\vol\right|\les \tilde{E}[\phi](\tau)+\delta_0 S^\a[\phi](\tau).
\end{equation*}
\end{lem}
\begin{proof}
The boundary $\Si_\tau^v\cup \bar C(0, \tau, v)$ consists of three parts: the spacelike $t$-constant slice $\{r\leq R\}$, the outgoing null hypersurface $S_\tau$ and
the incoming null hypersurface $\bar C(0, \tau, v)$. On the $t$-constant slice restricted to the region $\{r\leq R\}$, we use the formula
\eqref{curlessR}. Recall that
\[
 |\chi|\les\frac{1}{1+r}, \quad |f|\les 1, \quad |\chi'|\les \frac{1}{(1+r)^2}.
\]
We can show that
\[
\left|\int_{\Si_\tau\cap \{r\leq R\}}i_{\tilde{J}^X[\phi]}d\vol\right|\les \int_{\Si_\tau\cap\{r\leq R\}}|\bar\pa\phi|^2dx\les \tilde{E}[\phi](\tau).
\]
On $S_\tau$, using the formula \eqref{curStau}, we have
\begin{align*}
i_{\tilde{J}^{X}[\phi]}d\vol=-2\sqrt{-G}r^2(\pa^{\Lb}\phi
X(\phi)-\f12 X^{\Lb}\pa^\ga\phi\pa_\ga\phi-\f12
\pa^{\Lb}\chi\cdot\phi^2+\chi \pa^{\Lb}\phi\cdot\phi)dvd\om.
\end{align*}
We first estimate the last two terms in the above expression.
Recall that
 \[
|\chi|\les (1+r)^{-1},\quad |\pa^{\Lb}\chi|\les (1+r)^{-2}
\]
and note that
\[
|\pa^{\Lb}\phi|\leq |h||\pa\phi|+|\overline{\pa_v}\phi|,
\]
where $\overline{\pa_v}=(L, S_1, S_2)$. By the assumption \eqref{HH}, we can bound
\[
|-\f12 \pa^{\Lb}\chi \cdot \phi^2+\chi
\pa^{\Lb}\phi\cdot\phi|\les
\frac{\phi^2}{(1+r)^2}+|\overline{\pa_v}\phi|^2+\delta_0
\frac{|\pa\phi|^2}{(1+r)^{1+\a}}.
\]
The integral of the first two terms on the right hand side of the above inequality can be bounded by
the energy flux through $S_\tau$ by using Lemma \ref{lem2}. The integral of the
third term, by the definition, is exactly $S^\a[\phi](\tau)$, which will be absorbed for sufficiently small $\delta_0$.

Now to estimate
$\int_{S_\tau\cap\{t+r\leq 2v\}}i_{\tilde{J}^{X}[\phi]}d\vol$, it remains to estimate the integral of the first two
terms on the right hand side of the expression for $i_{\tilde{J}^{X}[\phi]}d\vol$. Recall
that $X=f\pa_r$, $\pa_r=\f12(L-\Lb)$, $|f|\les 1$. In particular,
we have $X=\f12 f(L-\Lb)$, $X^{\Lb}=-\f12 f$. Hence we have
\begin{align*}
|\pa^{\Lb}\phi X(\phi)-\f12
X^{\Lb}\pa^\ga\phi\pa_\ga\phi|&\les |\pa^{\Lb}\phi
L(\phi)|+|\f12 \pa^\ga\phi\pa_\ga\phi-\pa^{\Lb}\phi \Lb(\phi)|\\
&\les |\pa^{\Lb}\phi L(\phi)|+\f12|g^{\bar A\bar B}\bar
A(\phi)\bar B(\phi)-g^{\Lb\Lb}\Lb(\phi)\Lb(\phi)|\\
&\les \delta_0(1+r)^{-1-\a} |\pa\phi|^2+|\overline{\pa_v}\phi|^2,
\end{align*}
where $\bar A, \bar B\in\{L, S_1, S_2\}$. Therefore we can estimate
\begin{align*}
\left|\int_{S_\tau\cap\{t+r\leq 2v\}}i_{\tilde{J}^X[\phi]}d\vol\right|&\les
\tilde{E}[\phi](\tau)+ \int_{S_\tau}\frac{\delta_0|\pa\phi|^2}{(1+r)^{1+\a}}
r^2dvd\om\les \tilde{E}[\phi](\tau)+\delta_0 S^\a[\phi](\tau).
\end{align*}
Finally, we estimate the integral on the incoming null hypersurface $\bar C(0, \tau, v)$. Since the metric is asymptotically flat, using the
formula \eqref{curnullinfy}, we can split the expression of $i_{\tilde{J}^{X}[\phi]}d\vol$ according to the decomposition of the metric
$g=h+m_0$
\begin{align}
\notag &\pa^L\phi X(\phi)-\f12X^L\pa^\ga\phi\pa_\ga\phi- \f12\pa^L\chi \phi^2+\chi
\pa^L\phi\cdot\phi\\
\label{formu0}
&=h^{LA}A(\phi)X(\phi)-\f12X^L h^{\mu\nu}\pa_\nu\phi\pa_\mu\phi-\f12 h^{LA}A(\chi)\phi^2+\chi h^{L A}A(\phi)\phi\\
\notag &\quad+\frac{f}{4}(|\pa_u\phi|^2-|\nabb\phi|^2)-\frac{1}{4}\chi' \phi^2-\f12 \chi \pa_u\phi \phi.
\end{align}
Recall that
\[
|f|\les 1,\quad |\chi'|\les \frac{1}{(1+r)^2},\quad 2|\chi \pa_u\phi\cdot \phi|\leq \chi^2\phi^2+|\pa_u\phi|^2.
\]
On $\bar C(0, \tau, v)$, we can use a similar version of Lemma \ref{lem2} to control the integral of $|\chi'|\phi^2$, $|\chi\phi|^2$ by the
energy flux through $\bar C(0, \tau, v)$. That is we can estimate
\begin{align*}
&\liminf\limits_{v\rightarrow \infty} |\int_{\bar C(0,\tau, v)}2(\frac{f}{4}(|\pa_u\phi|^2-|\nabb\phi|^2)-\frac{1}{4}\chi' \phi^2-\f12 \chi
\pa_u\phi \phi)r^2\sqrt{-G}dud\om|\\
&\les \liminf\limits_{v\rightarrow \infty} \int_{\bar C(0, \tau, v)}(|\pa_u\phi|^2+|\nabb\phi|^2) r^2 dud\om=E^N[\phi]_0^\tau.
\end{align*}
Next we need to control the error terms which consist of the second line of the decomposition \eqref{formu0}. Since we assumed that
\[
I^\a[\phi]_0^\infty=\int_0^{\infty}\int_{\Si_\tau}\frac{|\bar\pa\phi|^2}{(1+r)^{1+\a}}dxd\tau
\]
is finite. In particular, we can choose a sequence $v_n\rightarrow\infty$ so that
\[
\int_{\bar C(0, \tau, v_n)}\frac{|\bar\pa\phi|^2}{(1+r)^{1+\a}}r^2dud\om\leq M v_n^{-1}
\]
for some constant $M$. Therefore by the assumption on the metric \eqref{HH}, we have
\begin{align*}
&|\int_{\bar C(0, \tau, v_n)}(h^{LA}A(\phi)X(\phi)-\f12X^L h^{\mu\nu}\pa_\nu\phi\pa_\mu\phi-\f12 h^{LA}A(\chi)\phi^2+\chi h^{L
A}A(\phi)\phi)r^2dud\om|\\
&\les \int_{\bar C(0, \tau, v_n)}|\bar\pa\phi|^2 r_+^{-\f12-\a} r^2 dud\om\les v_n^{\f12}\int_{\bar C(0, \tau, v_n)}|\bar\pa\phi|^2 r_+^{-1-\a}
r^2 dud\om\les M v_n^{-\f12}.
\end{align*}
Hence from the formula \eqref{curnullinfy}, we have shown that
\begin{align*}
&\liminf\limits_{v\rightarrow \infty} |\int_{\bar C(0,\tau, v)}i_{\tilde{J}^X[\phi]}d\vol|\les E^N[\phi]_0^\tau+\lim\limits_{n\rightarrow}M
v_n^{-\f12}=E^N[\phi]_0^\tau.
\end{align*}
The Lemma then follows as $E^N[\phi]_0^\tau\leq \tilde{E}[\phi](\tau)$.
\end{proof}

\bigskip

This lemma implies that the boundary terms on the right hand
side of the integrated energy estimate \eqref{menergyeq1} can be bounded by the energy flux plus an
error. Next we estimate the main error term $E(X)$ defined in line
\eqref{error}. We can compute
\[
|\pa_j(r^{-1}fx_i)|=|\pa_j(\chi x_i)|\les \frac{1}{1+r}.
\]
Thus the first term in line \eqref{error} can be controlled by
\[
|\pa_j(r^{-1}f x_i)h^{j\mu}\pa_\mu\phi\pa_i\phi|\les \delta_0 (1+r)^{-1-\a}|\pa\phi|^2.
\]
For the other terms, the idea is that if the good derivative
$\overline{\pa_v}$ hits on $\phi$ we can use Cauchy-Schwartz
inequality to bound it by $\tau_+^{-1-\a}|\overline{\pa_v}\phi|^2$ plus $r_+^{-1-\a}|\pa\phi|^2$. For the bad
term $\Lb(\phi) \Lb(\phi)$, we rely on the better decay of the
metric component $g^{\Lb\Lb}$. First for any vector field $Y$ such
that $Y( \Lb_\mu)=0$, $\|Y\|\leq 2$, we can write
\begin{align*}
 Y(g^{\mu\nu})\pa_\mu\phi\pa_\nu\phi&=Y(g^{\mu\nu})(\paL_\mu+\f12
\Lb_\mu\Lb)\phi(\paL_\nu+\f12\Lb_\nu\Lb)\phi\\
&=Y(g^{\mu\nu})\paL_\mu\phi\paL_\nu\phi+Y( g^{\mu\nu})
\Lb_\mu\paL_\mu\phi \Lb(\phi)+\frac{1}{4}Y
(g^{\mu\nu}\Lb_\mu\Lb_\nu)\Lb(\phi)\Lb(\phi).
\end{align*}
Here $\paL_\nu=\pa_\nu-\f12 \Lb_\nu \Lb$. From the assumption on the metric
\eqref{HH}, we can estimate
\begin{equation*}
|Y(g^{\mu\nu})\pa_\mu\phi\pa_\nu\phi|\les
\begin{cases}
\delta_0(1+r)^{-1-\a}
|\pa\phi|^2,\quad |x|\leq R,\\
\delta_0(1+r)^{-1-\a} |\pa\phi|^2+\delta_0 \tau_+^{-1-\a}|\overline{\pa_v}\phi|^2,\quad (t, x)\in S_\tau.
\end{cases}
\end{equation*}
Similarly, we have
\begin{equation}
\label{nullforest}
|h^{\mu\nu}\pa_\mu\phi\pa_\nu\phi|\les
\begin{cases}
\delta_0(1+r)^{-1-\a}
|\bar\pa\phi|^2,\quad |x|\leq R,\\
\delta_0(1+r)^{-1-\a} |\bar\pa\phi|^2+\delta_0 \tau_+^{-1-\a}|\overline{\pa_v}\phi|^2,\quad (t, x)\in S_\tau.
\end{cases}
\end{equation}
Here $\bar\phi=(\pa\phi, r_+^{-1}\phi)$. In particular, we conclude that
\begin{align*}
\f12|f'
h^{\mu\nu}\pa_\mu\phi\pa_\nu\phi|,\quad |\pa_r g^{\mu\nu}\cdot g_{\mu\nu}\pa^\ga\phi\pa_\ga\phi|
\end{align*}
verify the same estimates as in \eqref{nullforest}. Since $|f|\les 1$, $\pa_r(\Lb_\mu)=0$, the integral of the error term $E(X)$ obeys the
following estimates
\begin{align}
\notag
\int_{\tau_1}^{\tau_2}\int_{\Si_\tau}|E(X)|dx
d\tau\les& \delta_0 \int_{\tau_1}^{\tau_2}\int_{r\leq R}\frac{|\bar\pa\phi|^2}{(1+r)^{1+\a}}dx d\tau+\delta_0
\int_{\tau_1}^{\tau_2}\int_{\Si_\tau}\frac{|\bar\pa\phi|^2}{(1+r)^{1+\a}}dx d\tau\\&
\label{errExint}
+\delta_0\int_{\tau_1}^{\tau_2}\tau_+^{-1-\a}\tilde{E}[\phi](\tau)d\tau.
\end{align}
Finally, we discuss the inhomogeneous term $F$ and the linear term $N(\phi)$. For the linear term $N(\phi)$, note that
\[
|X(\phi)+\chi\phi|\les |\bar\pa\phi|.
\]
Therefore $|N(\phi)(X(\phi)+\chi\phi)|$ also satisfy the estimate \eqref{nullforest}. The inhomogeneous term $F$ can be bounded as follows
\[
|F(X(\phi)+\chi\phi)|\les \ep_0^{-1}(1+r)^{1+\a}|F|^2+\ep_0(1+r)^{-1-\a}|\bar \pa\phi|^2,\quad \ep_0>0.
\]
Now using Lemma \ref{prop1} to control the boundary terms on $\Si_\tau$, for sufficiently small $\delta_0$, $\ep_0$, depending only on $\a$, the
integrated energy inequality \eqref{menergyeq1} leads to
\begin{align*}
I^\a[\phi]_{\tau_1}^{\tau_2}&=\int_{\tau_1}^{\tau_1}\int_{\Si_\tau}\frac{|\bar\pa\phi|^2}{(1+r)^{1+\a}}dxd\tau
\les \delta_0 \int_{\tau_1}^{\tau_2}\int_{r\leq R}\frac{|\bar\pa\phi|^2}{(1+r)^{1+\a}}dx d\tau\\&+\tilde{E}[\phi](\tau_i)+\delta_0
S^\a[\phi](\tau_i)+\delta_0\int_{\tau_1}^{\tau_2}\tau_+^{-1-\a}\tilde{E}[\phi](\tau)d\tau+D^{\a}[F]_{\tau_1}^{\tau_2}.
\end{align*}
If $\delta_0$ is also small depending only on $\a$, then the first term on the RHS of the above estimate could be absorbed. Thus we have
\begin{equation} \label{ILEE}
I^\a[\phi]_{\tau_1}^{\tau_2}
\les \tilde{E}[\phi](\tau_i)+\delta_0
S^\a[\phi](\tau_i)+\delta_0\int_{\tau_1}^{\tau_2}\tau_+^{-1-\a}\tilde{E}[\phi](\tau)d\tau+D^{\a}[F]_{\tau_1}^{\tau_2}.
\end{equation}
Here for simplicity, $\tilde{E}[\phi](\tau_i)$ denotes $\tilde{E}[\phi](\tau_1)+\tilde{E}[\phi](\tau_2)$, similarly for $S^\a[\phi](\tau_i)$.

To prove \eqref{boundILE}, we choose the domain to be the finite region bounded by $\Si_0$ and $\{t=t_1\}$ and we do the same estimates as
above. Denote
\begin{align*}
&v(t_1)=2t_1-\tau, \quad C_0=\tilde{E}[\phi](0)+S^\a[\phi](0).\\
& E(t)=\left.\int_{r\leq t+R}|\pa\phi|^2dx \right|_{t=t},\quad D(t)=\int_{0}^{t}\int_{S_\tau\cap\{v\leq
v(\tau)\}}\tau_+^{-1-\a}|\overline{\pa_v}\phi|^2r^2dvd\om d\tau.
\end{align*}
We can show that
\begin{equation}
\label{boundILE0}
\begin{split}
\int_{0}^{t_1}\int_{\Si_\tau^{v(t_1)}}\frac{|\bar\pa\phi|^2}{(1+r)^{1+\a}}dxd\tau\les C_0+E(t_1)+\delta_0 D(t_1)+D^\a[F]_0^{t_1}.
\end{split}
\end{equation}
The energy flux $E(t_1)$ plays the same role as $\tilde{E}[\phi](\tau_2)+\delta_0 S^\a[\phi](\tau_2)$. The only difference is that $D(t)$ does
not contain the part from the cylinder $\{r\leq R\}$. The reason is that in fact the error term from this part could be absorbed for
sufficiently small $\delta_0$. Thus \eqref{boundILE0} holds for
all $t_1\geq 0$.

\subsubsection{The vector field $\pa_t$} We have taken the vector field $f\pa_r$ as multipliers to obtain the above integrated energy inequality
\eqref{ILEE} which will imply \eqref{ILE0} in Proposition \ref{ILEthm} if we can further control the energy flux $\tilde{E}[\phi](\tau)$. Next,
we take $\pa_t$ as multipliers to obtain the classical energy estimates.

Similar to the case when $X=f\pa_r$ discussed above, we apply the energy identity \eqref{energyeq} to the region bounded by
$\Si_{\tau_1}$, $\Si_{\tau_2}$ and the incoming null hypersurface $\bar C(\tau_1, \tau_2, v)$ and the vector field $X=\pa_t$, the function $\chi=0$. We obtain
\begin{align}
\notag &\int_{{\Sigma}_{\tau_1}^v}i_{J^X[\phi]}d\vol - \int_{{\Sigma}_{\tau_2}^v}i_{J^X[\phi]}d\vol+
\int_{\bar C(\tau_1, \tau_2, v)}i_{J^X[\phi]}d\vol\\
\label{iden0} &=\int_{\tau_1}^{\tau_2}\int_{\Sigma_\tau^v} \Box_g \phi X(\phi)+K^{X}[\phi]d\vol,
\end{align}
The estimates of the current $K^X[\phi]$ and the inhomogeneous term $\Box_g\phi$
are quite similar to the case when $X=f\pa_r$. And we can show that
\begin{align*}
&|K^{\pa_t}[\phi]|+|\Box_g \phi X(\phi)|=|-\pa_t
g^{\mu\nu}\pa_\mu\phi\pa_\nu\phi+\f12 \pa_t g^{\mu\nu}\cdot
g_{\mu\nu}\pa^\ga\phi\pa_\ga\phi|+|(F-N(\phi))\pa_t\phi|\\
&\les \delta_0\chi_{\{|x|\leq R\}}\frac{|\pa\phi|^2}{(1+r)^{1+\a}}+\frac{\delta_0|\pa\phi|^2}{(1+r)^{1+\a}}+\delta_0\chi_{\{|x|>R\}}\tau_+^{-1-\a}|\overline{\pa_v}\phi|^2
+\delta_0^{-1}(1+r)^{1+\a}|F|^2.
\end{align*}
Here $\chi_A$ is the characteristic function on the set $A$. Next we estimate the boundary terms. On $\Si_\tau^v\cap\{r\leq R\}$, using the formula \eqref{curlessR}, we have
\[
i_{{J}^X[\phi]}d\vol=(-\pa^t\phi\pa_t\phi+\f12
\pa^\ga\phi\pa_\ga\phi)\sqrt{-G}dx.
\]
Since $\delta_0$ is small, we can conclude that there is a positive constant $\la$ depending only $\delta_0$ such that
\[
 \la\int_{r\leq R}|\pa\phi|^2 dx\leq \int_{r\leq R}i_{{J}^X[\phi]}d\vol\leq \la^{-1}\int_{r\leq R}|\pa\phi|^2dx.
\]
On $S_\tau$, the formula \eqref{curStau} implies that
\begin{align*}
 i_{{J}^X[\phi]}d\vol&=-2(g^{\Lb A}A(\phi)\pa_t\phi-\frac{1}{4}\pa^\ga\phi\pa_\ga\phi)\sqrt{-G}r^2dvd\om=\f12(|\overline{\pa_v}\phi|^2+Err_1)\sqrt{-G}r^2dvd\om,
\end{align*}
where the error $Err_1$ obeys
\begin{align*}
|Err_1|=|4h^{\Lb
A}A(\phi)\pa_t\phi-h^{\ga\nu}\pa_\ga\phi\pa_\nu\phi| \les\delta_0\frac{|\pa\phi|^2}{(1+r)^{1+\a}}+\delta_0|\overline{\pa_v}\phi|^2.
\end{align*}
For sufficiently small $\delta_0$, depending only on $\a$, we can conclude that for some positive constants $C_1>4$, $C_2$, depending on
$\a$, $\la$, we have
\[
C_1^{-1}(\int_{\Si_\tau^v}|\overline{\pa_v}\phi|^2d\si-C_2\delta_0 S[\phi](\tau))\leq\int_{\Si_\tau^v}i_{J_{\mu}^{\pa_t}[\phi]}d\vol\leq C_1(
E[\phi](\tau)+\delta_0 S^\a[\phi](\tau)) ,
\]
where $d\si=dx$, $r\leq R$; $d\si=r^2dvd\om$, $r>R$.

Next we estimate the boundary term on $\bar C(\tau_1, \tau_2, v)$. On such incoming null hypersurfaces, we
have
\[
i_{{J}^{\pa_t}[\phi]}d\vol=\left(-\f12(|\pa_u\phi|^2+|\nabb\phi|^2)+2h^{L A}A(\phi)\pa_t\phi-\f12 h^{\mu\nu}\pa_\mu\phi\pa_\nu\phi\right)r^2dud\om.
\]
Since $I^\a[\phi]_0^\infty$ is finite, we can choose a sequence $v_n\rightarrow \infty$ such that
\[
\int_{\bar C(\tau_1, \tau_2, v_n)}\frac{|\bar\pa\phi|^2}{(1+r)^{1+\a}}r^2dud\om\leq M v_n^{-1}
\]
for some constant $M$. In particular, we have
\begin{align*}
\liminf\limits_{n\rightarrow \infty}(-\int_{\bar C(\tau_1, \tau_2, v_n)}i_{{J}^{\pa_t}[\phi]}d\vol)&=\liminf\limits_{n\rightarrow
\infty}\int_{\bar C(\tau_1, \tau_2, v_n)}\f12(|\pa_u\phi|^2+|\nabb\phi|^2)r^2dud\om=\f12 E^N[\phi]_{\tau_1}^{\tau_2}.
\end{align*}
Now from the energy identity \eqref{iden0} and all the above estimates, we can derive
\begin{align*}
\tilde{E}[\phi](\tau_2)+E^N[\phi]_{\tau_1}^{\tau_2}\les& E[\phi](\tau_1)+\delta_0
S^\a[\phi](\tau_i)+\delta_0I^\a[\phi]_{\tau_1}^{\tau_2}+\delta_0\int_{\tau_1}^{\tau_2}\tau_+^{-1-\a}E[\phi](\tau)d\tau\\
&+\delta_0\int_{\tau_1}^{\tau_2}\int_{r\leq R}\frac{|\pa\phi|^2}{(1+r)^{1+\a}}dxd\tau+\delta_0^{-1}D^{\a}[F]_{\tau_1}^{\tau_2}.
\end{align*}
Again $S^\a[\phi](\tau_i)=S^\a[\phi](\tau_1)+S^\a[\phi](\tau_2)$. Since either $\delta_0$ is small, using the integrated energy estimates \eqref{ILEE}, for sufficiently
small $\delta_0$, depending only on $\a$, we have
\begin{equation}
\label{energ1}
\begin{split}
\tilde{E}[\phi](\tau_2)\les& \tilde{E}[\phi](\tau_1)+\delta_0
S^\a[\phi](\tau_i)+\delta_0\int_{\tau_1}^{\tau_2}\tau_+^{-1-\a}E[\phi](\tau)d\tau+D^{\a}[F]_{\tau_1}^{\tau_2}.
\end{split}
\end{equation}
Now the problem is how to estimate the integral of the energy with negative weights in $\tau_+$, which can usually be bounded by using
Gronwall's inequality. However, due to the presence of $S^\a[\phi](\tau_2)$ on the right hand side, we are not able to use Gronwall's inequality
directly. Instead, let $\tau_2=\tau$ and then integrate the above energy inequality with respect to $\tau$ from $\tau_1$ to $\tau_2$. Use the
integrated energy inequality \eqref{ILEE} to bound the integral of $S^\a[\phi](\tau)$. We can show that
\begin{align*}
 \int_{\tau_1}^{\tau_2}\tau_+^{-1-\a}\tilde{E}[\phi](\tau)d\tau\les &\tilde{E}[\phi](\tau_1)+\delta_0 \tilde{E}[\phi](\tau_i)
 +\delta_0
 S^\a[\phi](\tau_i)+\delta_0\int_{\tau_1}^{\tau_2}\tau_+^{-1-\a}E[\phi](\tau)d\tau+D^{\a}[F]_{\tau_1}^{\tau_2}.
\end{align*}
For small $\delta_0$, the above estimate leads to
\[
\int_{\tau_1}^{\tau_2}\tau_+^{-1-a}\tilde{E}[\phi](\tau)d\tau\les \tilde{E}[\phi](\tau_1)+\delta_0 \tilde{E}[\phi](\tau_i)
 +\delta_0
 S^\a[\phi](\tau_i)+D^{\a}[F]_{\tau_1}^{\tau_2}.
\]
Then the energy estimate \eqref{eb} of Proposition \ref{ILEthm} follows from \eqref{energ1} if $\delta_0$ is sufficiently small, depending only
on $\a$, $\delta_0$. This energy estimate together with \eqref{ILEE} implies the integrated energy estimate \eqref{ILE0} of
Proposition \ref{ILEthm}. For the improved integrated energy estimate for the angular derivative of $\phi$, we note that we in fact have the
improved decay estimate \eqref{improvnabb}. We thus have shown the integrated energy estimate \eqref{ILE0} and the energy estimate \eqref{eb}.

Next we prove the boundedness of the
integrated energy \eqref{boundILE}. We do the energy estimate on the region bounded by $\Si_0$ and $\{t=t_1\}$, that is, as above we apply the
energy identity \eqref{energyeq} to such compact region with $X=\pa_t$, $\chi=0$. Similar to the above discussion, we can obtain
\[
E(t_1)\les C_0+\delta_0 D(t_1)+D^\a[F]_0^{t_1}.
\]
Here $E(t)$, $D(t)$, $C_0$ have been defined after line \eqref{boundILE0}. Thus by the estimate \eqref{boundILE0} there, we have
\[
\int_{0}^{t_1}\int_{\Si_\tau^{v(t_1)}}\frac{|\bar\pa\phi|^2}{(1+r)^{1+\a}}dxd\tau\les C_0+\delta_0 D(t_1)+D^\a[F]_0^{t_1}.
\]
Now to estimate $D(t_1)$, we do the energy estimate on the region bounded by $\Si_{\tau_1}$ and $\{t=t_1\}$ for $0\leq \tau_1\leq t_1$. We can
show that
\begin{align*}
\int_{S_{\tau_1}\cap \{t\leq t_1\}}|\overline{\pa_v}\phi|^2r^2dvd\om\les &E(t_1)+\delta_0\int_{\tau_1}^{t_1}\int_{S_\tau\cap\{t\leq t_1\}}
\tau_+^{-1-\a}|\overline{\pa_v}\phi|^2r^2dvd\om+\delta_0 D(t_1)\\
&+\delta_0\int_{S_{\tau_1}}\frac{|\pa\phi|^2}{(1+r)^{1+\a}}r^2dvd\om+C_0+D^\a[F]_0^{t_1}\\
&\les C_0+\delta_0 D(t_1)+\delta_0\int_{S_{\tau_1}}\frac{|\pa\phi|^2}{(1+r)^{1+\a}}r^2dvd\om+D^\a[F]_0^{t_1}.
\end{align*}
Multiply the above inequality by $(1+\tau_1)^{-1- a}$ and then integrated with respect to $\tau_1$ from $0$ to $t_1$. We can show that
\[
D(t_1)\les C_0+D^a[F]_0^{t_1}+\delta_0 D(t_1)+\delta_0(C_0+\delta_0 D(t_1)+D^\a[F]_0^{t_1}).
\]
Let $\delta_0$ to be sufficiently small. We conclude that
\[
D(t_1)\les C_0+D^\a[F]_0^{t_1}.
\]
This implies that
\[
\int_{0}^{t_1}\int_{\Si_\tau^{v(t_1)}}\frac{|\bar\pa\phi|^2}{(1+r)^{1+\a}}dxd\tau\les C_0+\delta_0 D(t_1)+D^\a[F]_0^{t_1}\les
C_0+D^\a[F]_0^{t_1}.
\]
Since the implicit constant is independent of $t_1$ and $C_0=\tilde{E}[\phi](0)+S^\a[\phi](0)$, we obtain \eqref{boundILE} by letting $t_1\rightarrow \infty$.
We thus finished the proof of Proposition \ref{ILEthm}.

\subsection{$p$-weighted Energy Inequality on Asymptotically Flat Spacetime}

In this section we establish the $p$-weighted energy inequalities on asymptotically flat spacetimes in a neighborhood of null infinity. We still consider solutions of the linear wave
equations \eqref{LWAVEEQ} on $(\mathbb{R}^{3+1}, g)$ with the metric $g$, the vector field $N$ satisfying the estimates \eqref{HH} and the following additional estimate
\begin{equation}
\label{HHimp} |\nabb h^{\Lb\Lb}|\leq\delta_0 (r_+^{-\frac{3}{2}-2\a}+r_+^{-1-\a}\tau_+^{-\f12-\f12\a}),\quad
r\geq R.
\end{equation}
We assume $\delta_0$ is sufficiently small, depending
only on $\a$ and $\a$ is a small positive constant. We have
\begin{prop}
\label{ILEthmpw} Let $\ep$, $\a_1$, $\a_2$ be positive constant such that
\[
 0<\ep< \frac{\a^2}{4}<\a<\frac{2\a+\a\ep}{2-\a}\leq\a_1<\a_2\leq \frac{7}{3}\a-\a_1-\ep.
\]
Let $\phi$ be the solution
of the linear wave equation \eqref{LWAVEEQ}. Assume $\phi$ satisfies the conditions in Lemma \ref{lem1} and $I^\ep[\phi]_0^\infty$ is finite. Then
\begin{itemize}
\item[(1)]: $p$-weighted energy inequality with weights $r^{1+\a_1}$
\begin{align}
\notag
&g^{1+\a_1}[\phi](\tau_2)+\int_{\tau_1}^{\tau_2}\int_{S_\tau}r^{\a_1}|\overline{\pa_v}\psi|^2 dvd\om d\tau \les g^{1+\a_1}[\phi](\tau_1)+\int_{\tau_1}^{\tau_2}\tau_+^{\ep}D^{\a_1}[F]_{\tau}^{\tau_2}d\tau\\
\label{pwe1a}
&+R^{1+\a_1}((\tau_1)_+^{1-\a}\tilde{E}[\phi](\tau_1)+\delta_0(\tau_i)_+^{1-\a}S^\ep[\phi](\tau_i)+(\tau_1)_+^{1+\ep} D^{\a_1}[F]_{\tau_1}^{\tau_2}),
\end{align}
\item[(2)]: $p$-weighted energy inequality with weights $r$
\begin{align}
\notag
&g^{1}[\phi](\tau_2)+\int_{\tau_1}^{\tau_2}\tilde{E}[\phi](\tau)d\tau\les g^1[\phi](\tau_1)+\int_{\tau_1}^{\tau_2}D^{\a_1}[F]_{\tau}^{\tau_2}d\tau\\
\label{pwe1}
&+R^{1+\ep}((\tau_1)_+^{1-\a}\tilde{E}[\phi](\tau_1)+(\tau_1)_+D^{\a_1}[F]_{\tau_1}^{\tau_2}+\delta_0(\tau_i)_+^{1-\a}S^\ep[\phi](\tau_i)),
\end{align}
\end{itemize}
where $S^\ep[\phi](\tau_i)=S^\ep[\phi](\tau_1)+S^\ep[\phi](\tau_2)$. The implicit constants also depend on $\ep$, $\a_1$, $\a_2$. The notations are defined in the end of Section 2.
\end{prop}

\begin{remark}
$\ep$ is much smaller than $\a$ and can be taken to be, for example
$\ep=\frac{\a}{10000}$. This small constant will appear in the integrated local energy estimate \eqref{ILE0} (since $\ep$ is smaller than $\a$, the estimate
\eqref{ILE0} also holds for $\ep$). $p=1+\a_1$ is the maximal $p$ we can take in the $p$-weighted energy inequality.

\end{remark}

In \cite{yang1}, \cite{yang2}, the $p$-weighted energy inequalities on flat (in a neighborhood of null infinity) spacetimes  are established by multiplying the
equation in null coordinates with $r^p\pa_v(r\phi)$ and then
integrating by parts. On asymptotically flat spacetimes, a more robust way, as also mentioned in \cite{newapp}, to prove the $p$-weighted energy
inequalities is to use the vector field method. When the metric is flat, we can alternatively derive the $p$-weighted energy inequalities
by using the vector field $r^p\pa_v$ as multipliers. For the general metrics with very weak ($\a$ is small) decay properties, we need to construct the corresponding vector fields as multipliers.

\bigskip

We use the vector field method to establish the above two $p$-weighted energy inequalities. Let $f$ be a smooth compactly supported nonnegative
function of $r$ defined on $[R, \infty)$. Choose the corresponding vector field as follows
\begin{equation}
\label{defofX}
 X=fY=f(-2g^{\Lb A}A+ g^{\Lb\Lb}\Lb)=f(-2\pa^{\Lb}+ g^{\Lb\Lb}\Lb),
\end{equation}
where $A$ runs over the null frame $\{\Lb, L, S_1, S_2\}$.

Although the null frame is merely defined locally and depends on the choice of $S_1$ and $S_2$, the above $X$ is in fact a well defined vector
field when $r\geq R$. Notice that the vector $g^{\Lb A}A$ can be viewed as the unique vector field which is orthogonal to the hypersurface
$S_\tau$ such that the inner product with $\Lb$ relative to the metric $g$ is 1. Since $\Lb$ is a global well defined vector field when $r\geq
R$, we conclude that $X$ is also a well defined vector field on $\{r\geq R\}$.

We rely on the energy identity \eqref{energyeq}. We choose the vector field $X$ as above \eqref{defofX}. Let the function $\chi=r^{-1}f$. The integral region $\mathcal{D}$ is
bounded by $S_{\tau_1}$, $S_{\tau_2}$ and
\[
 C_R=\{r=R, \tau_1\leq t\leq \tau_2\}.
\]
Since $f$ has compact support, the energy identity \eqref{energyeq} implies that
\begin{align}
\notag&\int_{{S}_{\tau_1}}i_{\tilde{J}^X[\phi]}d\vol -
\int_{{S}_{\tau_2}}i_{\tilde{J}^X[\phi]}d\vol-\int_{C_R}i_{\tilde{J}^X[\phi]}d\vol\\
\label{pweeq0}
&=\int_{\tau_1}^{\tau_2}\int_{S_\tau}(F-N(\phi))( X(\phi)+\chi\phi)+ K^X[\phi] +\chi\pa^\ga\phi\pa_\ga\phi -\f12 \Box_g\chi \phi^2 d\vol.
\end{align}
In this subsection, we define two functions on $S_\tau$
\[
 H=\delta_0\tau_+^{-\f12-\f12\a}r_+^{-\f12-2\a}, \quad \bar H=\delta_0 r_+^{-1-2\a}.
\]
We now estimate term by term in the above energy identity \eqref{pweeq0}. First for the linear term $N(\phi)(X(\phi)+\chi\phi)$, note that
\[
 |N(\phi)|\les \bar H |\pa\phi|+H|\overline{\pa_v}\phi|.
\]
We can write
\[
 X(\phi)+\chi\phi=\chi L(\psi)-2fh^{\Lb A}A(\phi)+fg^{\Lb\Lb}\Lb(\phi),\quad \psi=r\phi.
\]
Therefore we have
\begin{equation}
 \label{lineartest}
|N(\phi)(X\phi+\chi\phi)|\les (\bar H|\pa\phi|+H|\overline{\pa_v}\phi|)(\chi |L(\psi)|+f\bar H|\pa\phi|+f H|\overline{\pa_v}\phi|).
\end{equation}
Next we estimate the main term $K^X[\phi]$. Relative to the null frame $\{L,
\Lb, S_1, S_2\}$, we can calculate the deformation tensor of the
vector field $X$
\[
 \pi^X_{AB}=\f12(X(g_{AB})+g([A, X], B)+g([B, X], A)),
\]
where $[A, X]$ denotes the commutator of the two vector fields $A$, $X$. We remark here that $\pi^X_{AB}$ is defined
locally. However, the current $K^X[\phi]$ is independent of the choice of local coordinates. We thus can compute it relative to the null frame $\{L, \Lb, S_1, S_2\}$. We can compute
\begin{equation}
\label{KXdecomp}
\begin{split}
&K^X[\phi]+\chi\pa^\ga\phi\pa_\ga\phi=-\f12 X(g^{AB})A(\phi)B(\phi)+X^C[A, C](\phi)\pa^A\phi\\
&\qquad+f\pa^AY^C C(\phi)A(\phi)
+\pa^A f A(\phi)Y(\phi)-(\f12 div(X)-\chi)\pa^{\ga}\phi\pa_{\ga}\phi,
\end{split}
\end{equation}
where $div (X)$ is the divergence of the vector field $X$ with respect to the metric $g$ (also see the definition below). Since the metric
is asymptotically flat, we decompose the above expression according to the metric decomposition $g=h+m_0$.

We first consider $div(X)$. Recall that $X=fY$. We have
\[
div(X)=Y(f)+f div(Y).
\]
Recall that (see Section 2.3) at a fixed
point we can require that $[S_1, S_2]=0$. We thus can compute
\begin{align*}
\f12 div(Y)&=\f12 A(Y^A)+\frac{1}{4} Y(g_{AB})g^{AB}+\f12 Y^C g([A, C], B)g^{AB}\\
&= \f12 A(Y^A)-\frac{1}{4} Y(g^{AB})g_{AB}+r^{-1}(Y^L-Y^{\Lb})\\
&=r^{-1}-2r^{-1}h^{L\Lb}+r^{-1}g^{\Lb\Lb}-L(g^{L\Lb})-S(g^{S\Lb})-\f12 \Lb(g^{\Lb\Lb}),
\end{align*}
where $h^{AB}=g^{AB}-m_{0}^{AB}$, $S\in\{S_1, S_2\}$. Note that
\[
 Y(f)=-2m_0^{L\Lb} L(f)-2h^{\Lb A}A(f)+g^{\Lb\Lb}\Lb(f)=f'-2h^{\Lb A}A(f)+g^{\Lb\Lb}\Lb(f).
\]
Using the estimates \eqref{nullforest} to control $\pa^\ga\phi\pa_\ga\phi$, we then can
write the last term in the expression \eqref{KXdecomp} as
\begin{equation}
\label{divXde}
\begin{split}
 &(\f12 div(X)-\chi)\pa^\ga\phi\pa_\ga\phi=\f12 f'(-L(\phi)\Lb(\phi)+|\nabb\phi|^2)+Er_1,\\
 &|Er_1|\les (\chi+|f'|)\left(H|\overline{\pa_v\phi}||\pa\phi|+\bar H |\pa\phi|^2\right)
+\bar H f(|L(\phi)||\pa\phi|+|\nabb\phi|^2),
\end{split}
\end{equation}
where recall that $\chi=r^{-1}f$.

Similarly, we can write
\begin{equation}
\label{ACcom}
\begin{split}
 X^C [A, C](\phi)\pa^A\phi&=\chi(1-2h^{L\Lb}+h^{\Lb\Lb})S(\phi)\pa^S\phi+\chi h^{\Lb S} S(\phi)(\pa^{\Lb}\phi-\pa^L\phi)\\
&=\chi |\nabb\phi|^2+Er_2,\quad |Er_2|\les
\chi H|\nabb\phi||\pa\phi|.
\end{split}
\end{equation}
Next for $\pa^A f A(f)Y(\phi)$, recall that $f$ is a function of
$r$. We have
\begin{align*}
 &\pa^A f A(\phi)=f'(-\f12 \Lb (\phi)+\f12 L(\phi)+h^{LB}B(\phi)-h^{\Lb B}B(\phi)),\\
&Y(\phi)=-2\pa^{\Lb}\phi+g^{\Lb\Lb}\Lb(\phi)=L(\phi)-2h^{\Lb
B}B(\phi)+h^{\Lb\Lb}\Lb(\phi).
\end{align*}
Since
\[
 |-2h^{\Lb B}B(\phi)+h^{\Lb\Lb}\Lb(\phi)|\les \bar H|\pa\phi|+H|\overline{\pa_v}\phi|,
\]
we can write
\begin{equation}
 \label{AfY}
\begin{split}
\pa^A f A(\phi)Y(\phi)&=\f12 f'(L(\phi)-\Lb(\phi))L(\phi)+Er_3,\\
|Er_3|&\les H |f'||\pa\phi||\overline{\pa_v}\phi|+|f'|\bar H |\pa\phi|^2.
\end{split}
\end{equation}
We finally estimate the main error term in \eqref{pweeq0}
\[
-\f12 X(g^{AB})A(\phi)B(\phi)+f\pa^A Y^C C(\phi)A(\phi).
\]
Since this term is linear in $f$, it suffices to
consider the quadratic form
\[(-\f12 Y(g^{AB})+\pa^A Y^B)A(\phi)B(\phi),\quad Y=-2\pa^{\Lb}+g^{\Lb\Lb}\Lb.\]
of $\Lb(\phi)$,
$L(\phi)$, $S(\phi)$. The
coefficient of $\Lb(\phi)\Lb(\phi)$ satisfies
\[
\left|-\f12
(-2\pa^{\Lb}+g^{\Lb\Lb}\Lb)(g^{\Lb\Lb})-\pa^{\Lb}g^{\Lb\Lb}\right|=\left|-\f12
g^{\Lb\Lb}\Lb(g^{\Lb\Lb})\right|\les \bar H^2.
\]
Using the improved decay assumption
\eqref{HHimp} on $\nabb h^{\Lb\Lb}$, we can bound the coefficients for $\Lb(\phi)S(\phi)$, $S\in\{S_1, S_2\}$ as follows
\[
|(2\pa^{\Lb}-g^{\Lb\Lb}\Lb)(g^{\Lb S})-2\pa^{\Lb}g^{\Lb S}-\pa^S
g^{\Lb\Lb}|=|g^{\Lb\Lb}\Lb(g^{\Lb S})+\pa^S(g^{\Lb\Lb})|\les r^{-\f12}(H+\bar H).
\]
Similarly, the coefficient for $L(\phi)\Lb(\phi)$ can be bounded by $C_0\bar H$.
For $A(\phi)B(\phi)$, $A$, $B\in \{S_1, S_2\}$, we rely on the better decay in $r$ of $\overline{\pa_v} g$ to control the coefficient
\[
 |\pa^{\Lb}g^{AB}-\f12 g^{\Lb\Lb}\Lb(g^{AB})-\pa^A g^{\Lb B}-\pa^B g^{\Lb A}|\les \bar H+H^2\les \bar H.
\]
Finally for $L(\phi)\overline{\pa_v}\phi$, the coefficient can simply be bounded by $C_0 H$ for some constant $C_0$ depending only on $\a$. Summarizing, we have shown
\begin{align*}
&\left|-\f12 X(g^{AB})A(\phi)B(\phi)+f\pa^A Y^C C(\phi)A(\phi)\right|\les f\left(\bar H^2|\Lb(\phi)|^2+\bar
H|\Lb(\phi)L(\phi)|\right.\\&\qquad\qquad\left.+ (\bar H+H) r^{-\f12}|\Lb(\phi)||\nabb\phi|+\bar
H|\nabb\phi|^2+H|L(\phi)||\overline{\pa_v}\phi|\right).
\end{align*}
Combine this estimate with estimates \eqref{divXde},
\eqref{ACcom}, \eqref{AfY}. The equation \eqref{KXdecomp} gives the estimate for $K^X[\phi]+\chi \pa^\ga\phi\pa_\ga\phi$. Then from \eqref{pweeq0} and \eqref{lineartest},
we derive
\begin{align}
\label{pMenergy} &\int_{{S}_{\tau_1}}i_{\tilde{J}^X[\phi]}d\vol -
\int_{{S}_{\tau_2}}i_{\tilde{J}^X[\phi]}d\vol-\int_{C_R}i_{\tilde{J}^X[\phi]}d\vol\\
\notag &=\int_{\tau_1}^{\tau_2}\int_{S_\tau} F(
X(\phi)+\chi\phi)
 + \f12 f'|L(\phi)|^2+(\chi-\f12 f')|\nabb\phi|^2-\f12 \Delta\chi \cdot\phi^2 +E(X)d\vol,
\end{align}
where the error term $E(X)$ can be bounded as follows
\begin{align*}
 |E(X)| \les & (\chi+|f'|)\bar H |\pa\phi|^2+(|f'|H+fr^{-\f12}(\bar H+H) )|\pa\phi||\overline{\pa_v}\phi|+\chi\bar H|L\psi||\pa\phi|\\
&+fH^2|\overline{\pa_v}\phi|^2+f\bar H(|L(\phi)||\pa\phi|+|\nabb\phi|^2)
+ fH |L(\phi)||\overline{\pa_v}\phi|\\&+|(\Box_g-\Delta)\chi|
\phi^2+\chi H |L\psi||\overline{\pa_v}\phi|,
\end{align*}
where the Laplacian $\Delta$ is with respect to the flat metric on
$\mathbb{R}^3$. We further estimate $E(X)$ by choosing the function
$f$ explicitly. Let $\kappa$ be a smooth positive cutoff function on
$[0, \infty)$ such that
\[
 \kappa(x)=1,\quad  x\leq 1; \quad \kappa(x)=0, \quad x\geq 2;\quad |\kappa'(x)|\leq 2.
\]
Let $M$ be a large constant and then let $f=r^p\kappa_M=r^p
\kappa(\frac{|x|-R}{M})$. Recall that $\chi=r^{-1}f$. We can show
that
\[
|(\Box_g-\Delta)\chi|\les H r^{p-2},\quad (t, x)\in S_\tau.
\]
To estimate the boundary term on $C_R$, we can take $p=0$ in the
$p$-weighted energy inequality. Hence we also need to estimate the error term $E(X)$ when $p=0$. For this case we have
\[
|f|\leq 1, \quad |f'|=|\frac{\kappa'(r-R)}{M}|\les r^{-1}, \quad |\chi|\leq r^{-1} .
 \]
 And we can estimate the error term $E(X)$ as follows
\begin{equation}
 \label{errEx0}
\begin{split}
|E(X)|\les \delta_0(r^{-1-\a}|\bar\pa\phi|^2+\tau_+^{-1-\a}|\overline{\pa_v}\phi|^2+\tau_+^{-1-\a}r^{-2}\phi^2).
\end{split}
\end{equation}

We now estimate the error term $E(X)$ for general $p\in[0, 1+\a_1]$. Relative to $\psi=r\phi$, from the above estimate for $E(X)$, we can estimate
$r^2 E(X)$ as follows
\begin{equation*}
\begin{split}
r^2|E(X)|&\les \delta_0 r^{-1-\ep}|\bar\pa\psi|^2+r^{p-\f12}(
\bar H+H)(|\bar\pa\psi||\overline{\pa_v}\psi|+|\bar\pa\psi||\phi|)+fH^2|\overline{\pa_v}\psi|^2+fH\phi^2\\
&\quad+f\bar H(|L\psi||\bar\pa
\psi|+|\nabb\psi|^2+|\bar\pa\psi||\phi|)+fH(|L\psi||\overline{\pa_v}\psi|+|\phi||\overline{\pa_v}\psi|)\\
&\les \delta_0(r^{-1-\ep}|\bar\pa\psi|^2+r^{p-1}|\overline{\pa_v}\psi|^2+r^{p-1-2\a}|L(\psi)||\overline\pa\psi|+r^{p}\tau_+^{-1-\a}|L\psi|^2\\
&\quad +r^{p-1-2\a}|\phi||\bar\pa\psi| +r^{p-\a_2}\tau_+^{-1-\a}|\phi|^2),
\end{split}
\end{equation*}
where $\bar \pa\psi=(\pa\psi, \frac{\psi}{1+r})$. Here we have used the assumption \eqref{HH} and Cauchy-Schwartz inequality to obtain the above
estimates.

We need to further control
$r^{p-1-2\a}|\bar\pa\psi|(|\phi|+|L\psi|)$. Note that
\[
p\leq 1+\a_1, \quad \a_1+\a_2+\ep<2\a+\frac{1}{3}\a.
\]
Using Jensen's inequality, we can show that for all $p\in[0, 1+\a_1]$
\begin{align*}
 r^{p-1-2\a}|\bar\pa\psi|(|\phi|+|L\psi|)&\les r_+^{-1-\ep}\tau_+^{1-\a}|\bar\pa\psi|^2+r^{2p-1-4\a+\ep}\tau_+^{-1+\a}(|L\psi|^2+|\phi|^2)\\
&\les r_+^{-1-\ep}\tau_+^{1-\a}|\bar\pa\psi|^2+\left(1+\tau_+^{-1-\f12\a}r^{p-\a_2} \right)(|L\psi|^2+|\phi|^2).
\end{align*}
In particular, we can estimate $r^2E(X)$ as follows
\begin{equation*}
\begin{split}
r^2|E(X)|&\les \delta_0 (\tau_+^{1-\a}
r^{-1-\ep}|\bar\pa\psi|^2+r^{p-1}|\overline{\pa_v}\psi|^2+r^{p}\tau_+^{-1-\f12\a}|L\psi|^2\\
&\quad +|L\psi|^2+r^{p-\a_2}\tau_+^{-1-\f12\a}|\phi|^2+|\phi|^2),
\end{split}
\end{equation*}
As $\delta_0$ is assumed to be small, the integral of the second term will be absorbed. The integral of all the other terms except the first one
will be bounded by using Gronwall's inequality. The first term can not be bounded by using the integrated energy inequality directly due to the positive weights in $\tau_+$. However
since the integrated energy is expected to decay like $(1+\tau_1)^{-1-\a}$, we can use Lemma \ref{lweightILE} to estimate it.
In Lemma \ref{lweightILE}, take
\[
f(\tau)=\int_{S_\tau}(1+r)^{-1-\ep}|\bar\pa\psi|^2dvd\om.
\]
Using the integrated energy inequality \eqref{ILE0} (also holds for $\a=\ep$), we can show that
\begin{equation*}
 \begin{split}
 &\int_{\tau_1}^{\tau_2}\int_{S_\tau}\tau_+^{1-\a}r_+^{-1-\ep}r^{-2}|\bar\pa\psi|^2d\vol\les
\int_{\tau_1}^{\tau_2}\int_{S_\tau}\tau_+^{1-\a}\frac{|\bar\pa\phi|^2}{(1+r)^{1+\ep}}dx d\tau\\
 &\les\tilde{E}^\a[\phi]_{\tau_1}^{\tau_2}+(\tau_1)_+^{1-\a}\tilde{E}[\phi](\tau_1)+
(\tau_1)_+^{1-\a}D^{\ep}[F]_{\tau_1}^{\tau_2}\\
&\quad+\int_{\tau_1}^{\tau_2}\tau_+^{-\a}D^{\ep}[F]_{\tau}^{\tau_2}d\tau+\delta_0(\tau_i)_+^{1-\a} S^\ep[\phi](\tau_i).
 \end{split}
\end{equation*}
Now in the above estimate for $r^2E(X)$, we use Lemma \ref{lempphi2} to control the integral of $r^{p-\a_2}\phi^2$ and use
Lemma \ref{lem2} to control the integral of $\phi^2$. We end up with
\begin{align}
 \notag
&\qquad|\int_{\tau_1}^{\tau_2}\int_{S_\tau}r^2 E(X)d\vol|\\
 \notag
&\les \delta_0\int_{\tau_1}^{\tau_2}\int_{S_\tau}r^{p-1}|\overline{\pa_v}\psi|^2
+(1+r^{p}\tau_+^{-1-\f12\a})|L\psi|^2+(1+r^{p-\a_2}\tau_+^{-1-\f12\a})|\phi|^2dvd\om d\tau\\
 \notag
&\qquad+\delta_0\int_{\tau_1}^{\tau_2}\int_{S_\tau}\tau_+^{1-\a}r_+^{-1-\ep}r^{-2}|\bar\pa\psi|^2d\vol\\
 \label{Exeasy}
&\les \delta_0 (\int_{\tau_1}^{\tau_2}\int_{S_\tau}r^{p-1}|\overline{\pa_v}\psi|^2dvd\om d\tau+G^{p, 1+\f12\a}[\phi]_{\tau_1}^{\tau_2}
+R^{p-\a_2}\tilde{E}^{1+\frac{\a}{2}}[\phi]_{\tau_1}^{\tau_2}+
\tilde{E}^{0}[\phi]_{\tau_1}^{\tau_2}\\
 \notag
&+(\tau_1)_+^{1-\a}(\tilde{E}[\phi](\tau_1)+D^{\ep}[F]_{\tau_1}^{\tau_2})+\int_{\tau_1}^{\tau_2}\tau_+^{-\a}D^{\ep}[F]_{\tau}^{\tau_2}d\tau
+(\tau_i)_+^{1-\a} S^\ep[\phi](\tau_i)).
\end{align}
This gives the estimate for the error term $E(X)$ in the above energy identity \eqref{pMenergy}. We use a similar idea to treat the inhomogeneous term
$F(X(\phi)+\chi\phi)$. Note that
\[
 |X(\phi)+\chi\phi|=|fL(\phi)+\chi\phi-2h^{\Lb A}A(\phi)+g^{\Lb\Lb}\Lb(\phi)|\les \chi |L\psi|+f\bar H|\pa\phi|+ fH|\overline{\pa_v}\phi|.
\]
For $p\leq 1+\a_1<1+2\a$, we can estimate
\[
 |Ff(\bar H |\pa\phi|+H|\overline{\pa_v}\phi|)|\les\delta_0 ((1+r)^{1+\ep}|F|^2+(1+r)^{-1-\ep}|\pa\phi|^2+|\overline{\pa_v}\phi|^2).
\]
For the main term $|F|\chi|L\psi|$, we use Cauchy-Schwartz's inequality to show that
\begin{align*}
 r^2 |F|\chi|L\psi|&\les \delta_0^{-1}\tau_+^{1+(p-1)\a_1^{-1}\ep}|F|^2r^{3+\a_1}+\delta_0|L\psi|^2r^{2p-1-\a_1}\tau_+^{-1-(p-1)\a_1^{-1}\ep}\\
&\les \delta_0^{-1}\tau_+^{1+(p-1)\a_1^{-1}\ep}|F|^2r^{3+\a_1}+\delta_0|L\psi|^2(r^p \tau_+^{-1-\ep}+1),\quad p=1 \textnormal{ or } 1+\a_1.
\end{align*}
Similarly, we use Lemma \ref{lweightILE} to estimate the first term on the right hand side of the above inequality. And we can show that
\begin{align}
\label{pWeF}
 |\int_{\tau_1}^{\tau_2}\int_{S_\tau}F(X(\phi)+\chi\phi)& d\vol|\les \delta_0(\int_{\tau_1}^{\tau_2}\tilde{E}[\phi](\tau)d\tau+ G^{p, 1+\ep}[\phi]_{\tau_1}^{\tau_2}+
I^\ep[\phi]_{\tau_1}^{\tau_2})\\
\notag
& +\int_{\tau_1}^{\tau_2}\tau_+^{(p-1)\a_1^{-1}\ep} D^{\a_1}[F]_{\tau}^{\tau_2}d\tau+(\tau_1)_+^{(p-1)\a_1^{-1}\ep}D^{\a_1} [F]_{\tau_1}^{\tau_2}.
\end{align}

\bigskip

Next we estimate the boundary terms in the energy identity \eqref{pMenergy}.
\begin{lem}
\label{lempbdest}
 Let $X$ be the vector field defined in line \eqref{defofX} on the region $\{r\geq
R\}$. Let $f=r^p \kappa_M(x)=r^p \kappa(\frac{|x|-R}{M})$ for the
cutoff function $\kappa$. Let
$\chi=r^{-1}f$. Then we have
\begin{align*}
&\left|\int_{S_\tau}i_{\tilde{J}^X[\phi]}d\vol-\int_{S_\tau}f(\pa_v\psi)^2\sqrt{-G}dvd\om\right|\les\delta_0(S^\ep[\phi](\tau)+R^{p-\a_2}
\tilde{E}[\phi](\tau)+g^p[\phi](\tau))
\end{align*}
for all $p\in[1+\a_1]$. For the special case when $p=0$, we have
\begin{align*}
 \left|\int_{S_\tau}i_{\tilde{J}^X[\phi]}d\vol\right|\les \tilde{E}[\phi](\tau)+\delta_0S^\ep[\phi](\tau).
\end{align*}
Here the implicit constants are independent of $M$.
\end{lem}
\begin{proof}
Recall vector field $\tilde{J}^X[\phi]$ defined in line
\eqref{mcurent}. On $S_\tau$, we use the formula \eqref{curStau} and
we can compute
\begin{align*}
\int_{S_\tau}i_{\tilde{J}^X[\phi]}d\vol=\int_{S_\tau}(-2\pa^{\Lb}\phi
X(\phi)+ X^{\Lb} \pa^\ga\phi\pa_\ga\phi+ \pa^{\Lb}\chi
\cdot\phi^2-\chi \pa^{\Lb}\phi^2)\sqrt{-G}r^2dvd\om.
\end{align*}
For the special case when $p=0$, note that
\[
 |\pa^{\Lb}\phi|\les |\overline{\pa_v}\phi|+\bar H|\pa\phi|,\quad  |X(\phi)|=|-2\pa^{\Lb}\phi+g^{\Lb\Lb}\Lb(\phi)|\les|\overline{\pa_v}\phi|+\bar H|\pa\phi|,
\quad |\pa^{\Lb}\chi|\les r^{-2}.
\]
Thus we have
\begin{align*}
 |\int_{S_\tau}i_{\tilde{J}^X[\phi]}d\vol|\les \int_{S_\tau}|\overline{\pa_{v}}\phi|^2 +\frac{|\bar\pa\phi|^2}{(1+r)^{1+\ep}} \quad r^2 dvd\om\les \tilde{E}[\phi](\tau)
+\delta_0 S^\ep[\phi](\tau).
\end{align*}

For general $p$, we expand the integral on $S_\tau$ according to the metric decomposition $g=h+m_0$. Recall that $X=f(-2\pa^{\Lb}+g^{\Lb\Lb}\Lb)$. In
particular, $X^{\Lb}=-fg^{\Lb\Lb}$ and the main part of the vector
field $X$ is $-2fm_0^{L\Lb}L=fL$. Since $\chi=r^{-1}f$, we can write
\[
\pa^{\Lb}\phi L(\phi)+\f12 \chi \pa^{\Lb} \phi^2=\chi\pa^{\Lb} \phi
\cdot r L(\phi)+\chi \phi\pa^{\Lb}\phi=\chi \pa^{\Lb}\phi\cdot
L(\psi),\quad \psi=r\phi.
\]
Therefore we have
\begin{align*}
&\quad r^2(-2\pa^{\Lb}\phi X\phi+ \pa^{\Lb}\chi
\cdot\phi^2-\chi
\pa^{\Lb}\phi^2)\\
&=-2rf\pa^{\Lb}\phi L\psi-\f12\psi^2
L\chi-2r^2\pa^{\Lb}\phi (X-fL)\phi+ h^{\Lb A}A(\chi)\psi^2\\
&=rf L(\phi) L(\psi)-\f12 \psi^2 L\chi-2rf h^{\Lb
A}A(\phi)L(\psi)-2r^2\pa^{\Lb}\phi (X-fL)\phi+ h^{\Lb
A}A(\chi)\psi^2\\
&=f|L\psi|^2-\f12 L(\chi\psi^2)+Er_4,
\end{align*}
where the error $Er_4$ can be bounded as follows
\begin{align*}
|Er_4|&\les f|L\psi|(H|\overline{\pa_v}\psi|+H|\phi|+\bar
H|\pa\psi|)+H|\chi'||\psi|^2\\
&\quad +f(|L\psi|+|\phi|+\bar
H|\pa\psi|+H|\nabb\psi|)(H|\overline{\pa_v}\psi|+H|\phi|+\bar
H|\pa\psi|)\\
&\les fH|L\psi||\overline{\pa_v}\psi|+f\bar
H|L\psi||\pa\psi|+fH\bar
H|\pa\psi||\overline{\pa_v}\psi|+fH^2|\overline{\pa_v}\psi|\\
&\quad +fH|\overline{\pa_v}\psi||\phi|+fH\bar H
|\pa\psi||\phi|+fH|\phi|^2+f\bar H^2 |\pa\psi|^2\\
&\les \delta_0(
r^{-1-\ep}|\pa\psi|^2+r^2|\overline{\pa_v}\phi|^2+r^{p}|L\psi|^2+r^{p-\a_2}|\phi|^2).
\end{align*}
Similarly, using the estimate \eqref{nullforest} to control the
null form $\pa^\ga\phi\pa_\ga\phi$, we have
\begin{align*}
|r^2 X^{\Lb}\pa^\ga\phi\pa_\ga\phi|&\les f\bar
H(|L\psi||\pa\psi|+|\phi||\pa\psi|+|\nabb\psi|^2+|\phi|^2+\bar
H|\pa\psi|^2+H|\overline{\pa_v}\psi||\pa\psi|)\\
&\les \delta_0(
r^{-1-\ep}|\pa\psi|^2+r^2|\overline{\pa_v}\phi|^2+r^{p}|L\psi|^2+r^{p-\a_2}|\phi|^2).
\end{align*}
After integrating over $S_\tau$ with measure $dvd\om$, the first term on the right hand side of the above two inequalities can be controlled by
$S^\ep[\phi](\tau)$. The integral of the second term gives the energy flux through $S_\tau$. The last term can be controlled by using Lemma
\ref{lempphi2}. Summarizing, we have shown that
\begin{align*}
&\left|\int_{S_\tau}i_{\tilde{J}^X[\phi]}d\vol-\int_{S_\tau}f |L\psi|^2\sqrt{-G}dvd\om+\f12\int_{S_\tau} L(\chi \psi^2)\sqrt{-G}dvd\om\right|\\
&\les\delta_0 (S^\ep[\phi](\tau)+R^{p-\a_2}
\tilde{E}[\phi](\tau)+g^p[\phi](\tau)).
\end{align*}
The Lemma then follows if we can control the
integral of the term $-\f12 L(\chi \psi^2)$. We use integration by
parts to pass the derivative $L=\pa_v$ to $\sqrt{-G}$. By the assumption
\eqref{HH}, $L(\sqrt{-G})$ decays better in $r$. More precisely,
using Lemma \ref{lem1} and the fact that $\chi$ has compact support,
we have
\begin{align*}
\left|\int_{S_\tau}L(\chi\psi^2)\sqrt{-G}dvd\om\right|&=\left|\int_{S_\tau}\chi\psi^2L
(\sqrt{-G})dvd\om-\left.\int_{\om}\chi\psi^2\sqrt{-G}d\om\right|_{r=R}\right|\\
&\les \delta_0\int_{S_\tau}r^{p-\a_2}\phi^2
dvd\om+R^p\tilde{E}[\phi](\tau)
\end{align*}
Then again using Lemma \ref{lempphi2}, we can conclude the lemma.
\end{proof}

\bigskip

The above lemma shows that the energy flux through $S_\tau$ is almost equal to $g^p[\phi](\tau)$. We proceed to estimate the other terms in the $p$-weighted energy identity
\eqref{pMenergy}. We will estimate the boundary term on $C_R$ later and now we rewrite the energy terms on
the right hand side of \eqref{pMenergy} in terms of $\psi=r\phi$. Since $\chi=r^{-1}f$, we have the identity
\begin{align*}
\f12 r^2f' |\pa_v\phi|^2-\f12 r^2\Delta\chi \cdot \phi^2&=\f12
f'|\pa_v\psi|^2-\f12 f'\pa_v(r\phi^2)-\f12 \pa_v f' \cdot
r\phi^2\\
&=\f12 f'|\pa_v\psi|^2-\f12 \pa_v(f' r\phi^2).
\end{align*}
Here $\Delta$ is the Laplacian operator on flat $\mathbb{R}^3$. The first term on the RHS of the above identity is what we want. We use
integration by parts to control the integral of the second term. As $f$ has compact support, we can show that
\begin{align}
\label{errterm0}
&-\f12\int_{\tau_1}^{\tau_2}\int_{S_\tau}\pa_v (f'r\phi^2)
\sqrt{-G}dvd\om d\tau\\
\notag
&=\f12\int_{\tau_1}^{\tau_2}\int_{S_\tau}f'
r\phi^2 \pa_v\sqrt{-G} dvd\om d\tau+\f12
\int_{C_R}f'r\phi^2\sqrt{-G}d\om d\tau.
\end{align}
The first term on the RHS is an error term and we will estimate it later. We now move the
second term to the left hand side of the $p$-weighted energy identity \eqref{pMenergy} and combine it
with the original boundary term on $C_R$. The new boundary term on $C_R$ can
be written as
\begin{align*}
&-\int_{C_R}i_{\tilde{J}^X[\phi]}d\vol-\f12
\int_{C_R}f'r\phi^2\sqrt{-G}d\om d\tau\\
&=-f(R)\int_{C_R}i_{J^Y[\phi]}d\vol +\f12 \int_{C_R}\left(r^2\pa^r \chi \cdot
\phi^2-fr\pa^r(\phi^2) -f' r\phi^2 \right)\sqrt{-G}d\om d\tau\\
&=-R^p(\int_{C_R}i_{J^Y[\phi]}d\vol+\f12 \int_{C_R}\pa^r(r\phi^2)\sqrt{-G}d\tau d\om)+\f12 \int_{C_R}(\pa^r f - f')r\phi^2 \sqrt{-G}d\om d\tau\\
&=R^p B(C_R)+\f12 \int_{C_R}(\pa^r f - f')r\phi^2 \sqrt{-G}d\om d\tau,
\end{align*}
where $\pa^r=\pa^L-\pa^{\Lb}$ and we have used $B(C_R)$ to denote the integral. We see that $B(C_R)$ is independent of the power $p$. Hence to
control the boundary term $B(C_R)$, it suffices to take $p=0$, which is essentially the energy estimates we have done in the previous section.
We will estimate the boundary term $B(C_R)$ later. We now group the error term on the boundary $C_R$ in the above inequality with the error term
in \eqref{errterm0} and denote
\[
 Er_5=\f12 \int_{C_R}(\pa^r f - f')r\phi^2 \sqrt{-G}d\om d\tau-\f12\int_{\tau_1}^{\tau_2}\int_{S_\tau}f'
r\phi^2 \pa_v\sqrt{-G} dvd\om d\tau.
\]
Since $|\pa_v\sqrt{-G}|\les \bar H$, using Lemma \ref{lem1} and Lemma \ref{lem2}, we can estimate $Er_5$
\begin{equation}
\label{Er5}
 |Er_5|\les \int_{\tau_1}^{\tau_2}(\int_{\om} \bar{H} R^{p}\phi^2 d\om+ \int_{S_\tau}f\bar H \phi^2 dvd\om )d\tau\les\delta_0 R^{p-1-2\a}\int_{\tau_1}^{\tau_2}
\tilde{E}[\phi](\tau)d\tau.
\end{equation}
Now from the $p$-weighted energy identity \eqref{pMenergy}, the above discussion leads to the following energy estimate
\begin{align}
\label{pMenergy1} &\int_{{S}_{\tau_1}}i_{\tilde{J}^X[\phi]}d\vol -
\int_{{S}_{\tau_2}}i_{\tilde{J}^X[\phi]}d\vol+R^p B(C_R)+Er_5\\
\notag
&=\int_{\tau_1}^{\tau_2}\int_{S_\tau}F(X(\phi)+\chi\phi)
 + \f12 f'r^{-2}|\pa_v\psi|^2+r^{-2}(\chi-\f12 f')|\nabb\psi|^2 +E(X)d\vol.
\end{align}
The boundary term on $S_\tau$ is almost equal to $g^p[\phi](\tau)$ by Lemma \ref{lempbdest}. If
$f=r^p$, the energy term on the right hand side will give us a positive sign. This will be made to be rigorous by taking the limit $M\rightarrow \infty$. Here we
recall that $M$ is the parameter in the cutoff function $\kappa_M$.

Finally in the above energy identity \eqref{pMenergy1}, we estimate the boundary term $B(C_R)$ by taking $p=0$. For this case $f\leq 1$, $|\chi|\leq r^{-1}$.
The inhomogeneous term $F(X(\phi)+\chi\phi)$ can be bounded by
\[
 |F(X(\phi)+\chi\phi)|\les |F||\bar\pa\phi|\les |F|^2(1+r)^{1+\ep}+(1+r)^{-1-\ep}|\bar\pa\phi|^2.
\]
Now if
\[
\int_{\tau_1}^{\tau_2}\tilde{E}[\phi](\tau)d\tau<\infty,
\]
then we have
\begin{align*}
 \lim\limits_{M\rightarrow\infty}|\int_{\tau_1}^{\tau_2}\int_{S_\tau}(\kappa_M)'r^{-2}(|\pa_v\psi|^2-|\nabb\psi|^2)d\vol|\les \lim\limits_{M\rightarrow\infty}
\frac{1}{M}\int_{\tau_1}^{\tau_2}\tilde{E}[\phi](\tau)d\tau=0.
\end{align*}
Therefore let $M\rightarrow \infty$ in the above energy identity \eqref{pMenergy1} with $p=0$. Using estimate \eqref{errEx0} to control the error term $E(X)$ and estimate \eqref{Er5}
to bound $Er_5$ and Lemma \ref{lempbdest} to control the boundary terms on $S_{\tau_i}$, $i=1, 2$, we then have the estimate for the boundary term $B(C_R)$
\begin{equation*}
\begin{split}
|B(C_R)|&\les \delta_0 R^{-1-2\a}\int_{\tau_1}^{\tau_2}\tilde{E}[\phi](\tau)d\tau+\tilde{E}[\phi](\tau_i)+\delta_0S^\ep[\phi](\tau_i)
+\int_{\tau_1}^{\tau_2}\tau_+^{-1-\a}\tilde{E}[\phi](\tau)d\tau\\
&\quad +D^{\ep}[F]_{\tau_1}^{\tau_2}+\int_{\tau_1}^{\tau_2}\int_{S_\tau}\frac{|\bar\pa\phi|^2}{(1+r)^{1+\ep}}+\frac{|\nabb\phi|^2}{1+r} dxd\tau.
\end{split}
\end{equation*}
If $\int_{\tau_1}^{\tau_2}\tilde{E}[\phi](\tau)d\tau=\infty$, then the above estimate for $B(C_R)$ holds automatically.
Now we use the integrated energy inequality \eqref{ILE0} and the energy inequality \eqref{eb} to improve the above estimate for the boundary term $B(C_R)$. We have
\begin{equation}
 \label{bdest}
\begin{split}
 |B(C_R)|&\les\delta_0 R^{-1-2\a}\int_{\tau_1}^{\tau_2}\tilde{E}[\phi](\tau)d\tau+\tilde{E}[\phi](\tau_1)+\delta_0 S^\ep[\phi](\tau_i)+D^{\ep}[F]_{\tau_1}^{\tau_2},
\end{split}
\end{equation}
where $S^\ep[\phi](\tau_i)=S^\ep[\phi](\tau_1)+S^\ep[\phi](\tau_2)$. This gives the estimate for the boundary term $B(C_R)$ in the above energy identity \eqref{pMenergy1}.

\bigskip

Now in the above energy identity \eqref{pMenergy1}, we have estimate \eqref{Exeasy} for the error term $E(X)$, estimate \eqref{pWeF} for the inhomogeneous term $F(X\phi+\chi\phi)$,
estimate \eqref{Er5} for the error term $Er_5$ and the above estimate \eqref{bdest} for the boundary term $B(C_R)$. The boundary term on $S_\tau$ has been discussed in Lemma
\ref{lempbdest}. As the function $f$, $\chi$ depends on the parameter $M$ in the cutoff function $\kappa_M$, we now argue that we can push the parameter $M$
to infinity and conclude the $p$-weighted
energy inequalities \eqref{pwe1}, \eqref{pwe1a}.

Without loss of generality, we can assume that
$\kappa$ is decreasing. We find that
\begin{align*}
&\f12 f'= \f12 pr^{p-1}\kappa_M+\f12 M^{-1}r^p
\kappa'(\frac{r-R}{M}),\\
&\chi-\f12 f'=(1-\frac{
p}{2})r^{p-1}\kappa_M-\f12 M^{-1}r^p\kappa'(\frac{r-R}{M})\geq (1-\frac{p}{2})r^{p-1}\kappa_M.
\end{align*}
Note that $\kappa'$ is supported on $[1, 2]$. We
conclude that if
\[
G^{p-1, 0}[\phi]_{\tau_1}^{\tau_2}=\int_{\tau_1}^{\tau_2}\int_{S_\tau}r^{p-1}|L\psi|^2dvd\om
d\tau
\]
is finite, then
\begin{align*}
&\lim\limits_{M\rightarrow\infty}|\int_{\tau_1}^{\tau_2}\int_{S_\tau}M^{-1}r^p\kappa'(\frac{r-R}{M})|L\psi|^2r^{-2}dvd\om d\tau\\
&\leq
\lim\limits_{M\rightarrow\infty}\int_{\tau_1}^{\tau_2}\int_{r\geq
M}r^{p-1}|L\psi|^2dvd\om d\tau=0.
\end{align*}
We first consider the $p$-weighted energy inequality when $p=1$. Note that
\[
G^{0, 0}[\phi]_{\tau_1}^{\tau_2}=\int_{\tau_1}^{\tau_2}\int_{S_\tau}r^2|L\phi|^2-L(r\phi^2)dvd\om
d\tau\les \int_{\tau_1}^{\tau_2}\tilde{E}[\phi](\tau)d\tau.
\]
For fixed $\tau_1<\tau_2$, it suffices to prove Proposition \ref{ILEthmpw} when
\[
\tilde{E}[\phi](\tau_1)+S^\ep[\phi](\tau_1)+S^\ep[\phi](\tau_2)+D^{\ep}[F]_{\tau_1}^{\tau_2}<\infty.
\]
Otherwise, all the estimates in Proposition \ref{ILEthmpw} hold automatically for $\tau_1<\tau_2$. In this case using the energy estimate \eqref{eb}, we have
\[
 \int_{\tau_1}^{\tau_2}\tilde{E}[\phi](\tau)d\tau\leq C(\tau_1, \tau_2) (\tilde{E}[\phi](\tau_1) +S^\ep[\phi](\tau_1)+S^\ep[\phi](\tau_2)+D^{\ep}[F]_{\tau_1}^{\tau_2})<\infty,
\]
where $C(\tau_1, \tau_2)$ is constant. This further implies that
$G^{0, 0}[\phi]_{\tau_1}^{\tau_2}$ is finite. Therefore by the argument above, we can let $M$ go to
infinity in the $p$-weighted energy inequality \eqref{pMenergy1} with $p=1$ and we can conclude from the estimates \eqref{Exeasy},
\eqref{pWeF}, \eqref{Er5}, \eqref{bdest} together with Lemma \ref{lempbdest} that
\begin{align}
\notag
 &g^1[\phi](\tau_2)+\int_{\tau_1}^{\tau_2}\int_{S_\tau}|\overline{\pa_v}\psi|^2dvd\om d\tau\les \delta_0 (g^1[\phi](\tau_2)+
\int_{\tau_1}^{\tau_2}\int_{S_\tau}|\overline{\pa_v}\psi|^2dvd\om d\tau)\\
\label{pwe100}
&+\delta_0(\int_{\tau_1}^{\tau_2}\tilde{E}[\phi](\tau)d\tau+  G^{1, 1+\ep}[\phi]_{\tau_1}^{\tau_2})+R(\tau_1)_+^{1-\a}\tilde{E}[\phi](\tau_1)+\delta_0 R(\tau_i)_+^{1-\a} S^\ep[\phi](\tau_i)\\
\notag
&+g^1[\phi](\tau_1)+\int_{\tau_1}^{\tau_2}D^{\a_1}[F]_{\tau}^{\tau_2}d\tau+ R(\tau_1)_+ D^{\a_1}
[F]_{\tau_1}^{\tau_2},
\end{align}
where we used the energy inequality \eqref{eb} to estimate $\tilde{E}^{1+\f12 \a}[\phi]_{\tau_1}^{\tau_2}$. For small
$\delta_0$ the first two terms can be absorbed. Usually $G^{1, 1+\ep}[\phi]_{\tau_1}^{\tau_2}$ can be bounded
by using Gronwall's inequality. However, due to the presence of
$S^\ep[\phi](\tau_i)$ on the right hand side, we can not use Gronwall's
inequality directly. However we can
take $\tau_2=\tau$ in the above inequality. Multiply both side by
$\tau_+^{-1-\ep}$ and then integrate
it with respect to $\tau$ from $\tau_1$ to $\tau_2$. We retrieve $G^{1, 1+\ep}[\phi]_{\tau_1}^{\tau_2}$ on the left hand side. The same term
will appear on the right hand side but with the small coefficient
$\delta_0$. We thus can estimate it.

Now to prove the $p$-weighted energy inequality \eqref{pwe1} with $p=1$, we
need to recover $\int_{\tau_1}^{\tau_2}\tilde{E}[\phi](\tau)d\tau$ on the
left hand side of \eqref{pwe100}. Note that
\[
 \int_{S_\tau}|\overline{\pa_v}\psi|^2dvd\om =\int_{S_\tau}|\overline{\pa_v}\phi|^2r^2 +L(r\phi^2)dvd\om=\int_{S_\tau}|\overline{\pa_v}\phi|^2r^2dvd\om-\left.\int_{\om}r\phi^2d
\om\right|_{r=R}.
\]
For $R\geq 1$, we have
\begin{equation*}
 R^3\int_{\om}\phi^2(\tau, R, \om)d\om=\int_{0}^{R}\int_{\om}\pa_r(r^3\phi^2)d\om dr\leq 3\int_{r\leq R}
 \phi^2dx + \int_{r\leq R}R^2|\pa_r\phi|^2+\phi^2dx.
\end{equation*}
 Hence we have
\[
 R\int_{\om}\phi^2(\tau, R, \om)d\om\leq 8\int_{r\leq R}|\pa\phi|^2+(1+R)^{-2}\phi^2 dx\leq 8 \int_{r\leq R}|\bar\pa\phi|^2dx.
\]
Now add both side of the above $p$-weighted energy inequality \eqref{pwe100} when $p=1$ with
\[
 \int_{\tau_1}^{\tau_2}\int_{r\leq R}|\pa\phi|^2dxd\tau+R\int_{\tau_1}^{\tau_2}\int_{\om}\phi^2(\tau, R, \om)d\om d\tau\les (1+R)^{1+\ep}I^\ep[\phi]_{\tau_1}^{\tau_2}.
\]
Then the left hand side of \eqref{pwe100} becomes
\[
 g^1[\phi](\tau_2)+\int_{\tau_1}^{\tau_2}\tilde{E}[\phi](\tau)d\tau.
\]
For small $\delta_0$ the first term in the second line of \eqref{pwe100} can be absorbed. Then using the integrated energy
inequality \eqref{ILE0} to control $I^\ep[\phi]_{\tau_1}^{\tau_2}$, we can conclude from \eqref{pwe100} the $p$-weighted energy inequality
\eqref{pwe1} for $p=1$.

\bigskip

Finally we prove the $p$-weighted energy inequality \eqref{pwe1a} when $p=1+\a_1$. Having the $p$-weighted energy inequality when $p=1$, which in particular gives the bound
for $g^1[\phi](\tau)$ (we may assume the right hand side of \eqref{pwe1a} is finite), we conclude that
$\int_{\tau_1}^{\tau_2}g^1[\phi](\tau)d\tau$ is finite.
In particular, we have
\[
 \int_{\tau_1}^{\tau_2}\int_{S_\tau}r^{\a_1}|L\psi|^2dvd\om d\tau\leq  \int_{\tau_1}^{\tau_2}g(1, \tau)d\tau<\infty.
\]
Therefore in the $p$-weighted energy inequality \eqref{pMenergy1} we can set $p=1+\a_1$ and then
let the parameter $M$ in the cutoff function go to infinity. Similar to the above $p$-weighted energy inequality when $p=1$, for small $\delta_0$, we can show that
\begin{align*}
& g^{1+\a_1}[\phi](\tau_2)+\int_{\tau_1}^{\tau_2}\int_{S_\tau}r^{\a_1}|\overline{\pa_v}\psi|^2dvd\om d\tau\les g^{1+\a_1}[\phi](\tau_1)+\delta_0\int_{\tau_1}^{\tau_2}\tilde{E}[\phi](\tau)d\tau
\\
&+ \delta_0 G^{1+\a_1, 1+\ep}[\phi]_{\tau_1}^{\tau_2}+R^{1+\a_1}(\tau_1)_+^{1-\a}\tilde{E}[\phi](\tau_1)+\delta_0 R^{1+\a_1}(\tau_i)_+^{1-\a}
 S^\ep[\phi](\tau_i)\\
&+\int_{\tau_1}^{\tau_2}(\tau)_+^{\ep}D^{\a_1}[F]_{\tau}^{\tau_2}d\tau+ R^{1+\a_1}(\tau_1)_+^{1+\ep}D^{\a_1}
[F]_{\tau_1}^{\tau_2}.
\end{align*}
The integral of the energy flux $\tilde{E}[\phi](\tau)$ from $\tau_1$ to $\tau_2$ can be controlled by the $p$-weighted energy inequality \eqref{pwe1} with $p=1$.
To estimate $G^{1+\a_1, 1+\ep}[\phi]_{\tau_1}^{\tau_2}$, we set $\tau_2=\tau$
in the above inequality and then integrate both side with weights $\tau_+^{-1-\ep}$ from $\tau_1$ to $\tau_2$. And then we can conclude \eqref{pwe1a} from the above inequality.
We thus finished the proof of Proposition \ref{ILEthmpw}.

\section{Integrated Local Energy Decay}

We have shown in the previous section the integrated energy estimates Proposition \ref{ILEthm} and the $p$-weighted energy inequalities
Proposition \ref{ILEthmpw} without using any vector fields with positive weights in $t$. We now argue that under appropriate
assumptions on the inhomogeneous term $F$ as well as the data on the initial hypersurface $\Si_0$ the energy flux $E[\phi](\tau)$ decays in
$\tau$.

We still consider the linear wave equation \eqref{LWAVEEQ} on $(\mathbb{R}^{3+1}, g)$ with metric $g$ satisfies the conditions in Proposition \ref{ILEthmpw}. Let $E_0$ denote the
size of the data on $\Si_0$
\[
 E_0:=\tilde{E}[\phi](0)+S^\ep[\phi](0)+g^{1+\a_1}[\phi](0),
\]
where $\ep$, $\a_1$ are small positive constants appeared in Proposition \ref{ILEthmpw}.
We always assume that $E_0$ is finite.

On asymptotically flat spacetimes, we are not able to show the decay of the energy flux $E[\phi](\tau)$. However, we can show that integrated local energy
$I^\ep[\phi]_{\tau}^{\infty}$ decays.
\begin{prop}
 \label{propILED}
Assume that the inhomogeneous term $F$ satisfies
\[
 D^{\a_1}[F]_{\tau_1}^{\tau_2}\leq C_1(\tau_1)_+^{-1-\a},\quad \forall \tau_2\geq \tau_1.
\]
Then we have the integrated local energy decay
\begin{equation*}
I^\ep[\phi]_{\tau_1}^{\tau_2}=\int_{\tau_1}^{\tau_2}\int_{\Si_\tau}\frac{|\bar\pa\phi|^2}{(1+r)^{1+\ep}}dxd\tau\leq C_{\a, R}(E_0+
C_1) (\tau_1)_+^{-1-\a},\quad\forall
 \tau_2\geq \tau_1.
\end{equation*}
Here the constant $C_{\a, R}$ depends on $\a$, $R$. And we recall here that the small positive constants $\ep$, $\a_1$ satisfy the relations in Proposition \ref{ILEthmpw}.
\end{prop}

Compared to the case when the metric is flat outside the cylinder with radius $R$, the main difficulty is
the presence of $S^\ep[\phi](\tau_2)$ on the right hand side of
estimates in Proposition \ref{ILEthm}, \ref{ILEthmpw}. The idea is that we first show that $I^\ep[\phi]_0^\infty$ is finite from the estimate
\eqref{boundILE}. And then we can extract a sequence $\{\tau_n\}$ such that $S^\ep[\phi](\tau_n)$ decays. This will lead to the decay of the integrate energy.

\begin{proof}
The assumption above in particular implies that $D^\ep[F]_{0}^{\infty}$ is finite. Therefore by the boundedness of the integrated energy estimate
\eqref{boundILE} we have $I^\ep[\phi]_0^{\infty}$ is finite. In particular, we can conclude that there is a sequence
$\tau_n\rightarrow \infty$ such that $S^\ep[\phi](\tau_n)$ is finite. Then from the energy inequality \eqref{eb}, we infer that $\tilde{E}[\phi](\tau_n)$ is finite. In particular
$E^N[\phi]_0^{\tau_n}$ is finite for all $n$. Since $\tau_n\rightarrow\infty$, we have $E^N[\phi]_0^\tau$ is finite for all $\tau$. By Lemma \ref{lem1} all the previous estimates
hold if we replace $\tilde{E}[\phi](\tau)$ with $E[\phi](\tau)$.

Denote $M=E_0+C_1$. Without loss of generality we may assume $M>1$. For some big constant $C_2$ depending only on $R$ and $\a$, assume
\begin{equation}
\label{ILEass} I^\ep[\phi]_{\tau_1}^{\tau_2}=\int_{\tau_1}^{\tau_2}\int_{\Si_\tau}\frac{|\bar\pa\phi|^2}{(1+r)^{1+\ep}}dxd\tau\leq C_2M
 (1+\tau_1)^{-\b},\quad \forall \tau_2\geq \tau_1\geq 0
\end{equation}
for some $\b\in[0, 1+\a]$. Since $I^\ep[\phi]_0^\infty$ is finite, the above estimate holds for $\b=0$. The proposition claims that it holds for
$\b=1+\a$. We define a nonempty set
\begin{equation}
 \label{setT}
T=\{\tau\left|S^\ep[\phi](\tau)\leq\int_{\Si_\tau}\frac{|\bar\pa\phi|^2}{(1+r)^{1+\ep}}dx\right.\leq 10C_2M\tau_+^{-1-\b}\}.
\end{equation}
Let $\tau_1=0$ in the $p$-weighted energy inequality \eqref{pwe1a} with weights $r^{1+\a_1}$. We obtain
\begin{equation*}
\label{S12a} \int_{S_\tau}r^{1+\a_1}(\pa_v\psi)^2dvd\om\les M,\quad\forall \tau\in T.
\end{equation*}
By the definition of $T$, we have
\begin{equation*}
 \int_{S_\tau}\frac{(\pa_v\psi)^2}{(1+r)^{1+\ep}}dvd\om\les M\tau_+^{-1-\b},\quad\forall \tau\in T.
\end{equation*}
Here recall that $\psi=r\phi$. Interpolate between the above two
inequalities. We get
\begin{equation*}
 \int_{S_\tau}r(\pa_v\psi)^2dvd\om\les M(1+\tau)^{-\th\a},\quad\forall \tau\in
 T,
\end{equation*}
where
\[
\th=\min\{1, \frac{(1+\b)\a_1}{(2+\a_1+\ep)\a}\}.
\]
Then the $p$-weighted energy inequality \eqref{pwe1} when $p=1$
implies that
\begin{equation*}
\begin{split}
\int_{\tau_1}^{\tau_2}E[\phi](\tau)d\tau&\les M(\tau_1)_+^{-\th\a}+(\tau_1)_+^{1-\a}E[\phi](\tau_1), \quad\forall \tau_1,\tau_2\in T.
\end{split}
\end{equation*}
Now the energy inequality  \eqref{eb} shows
\begin{equation*}
 E[\phi](\tau_2)\les E[\phi](\tau)+S^\ep[\phi](\tau)+(\tau_2)_+^{-1-\a}+M(1+\tau_2)^{-1-\b}, \forall \tau\leq \tau_2,\tau_2\in T.
\end{equation*}
In particular, we have
\[
 E[\phi](\tau_2)\les M,\quad \forall \tau_2\in T.
\]
Combine this with the previous two estimates. We can show that
\begin{equation}
\label{tauE}
\begin{split}
 (\tau_2-\tau_1)E[\phi](\tau_2)\les &  M(\tau_1)_+^{-\th\a}+ (\tau_1)_+^{1-\a}E[\phi](\tau_1)\\
&+(\tau_2-\tau_1) (\tau_2)_+^{-1-\a}+M(\tau_2-\tau_1)(\tau_2)_+^{-1-\b}+M(\tau_1)_+^{-\b}
\end{split}
\end{equation}
for all $\tau_1$, $\tau_2\in T$, $\tau_1\leq \tau_2$. Since $0\in T$, in particular, let $\tau_1=0$. We get
\begin{equation*}
 E[\phi](\tau_2)\les M(\tau_2)_+^{-1},\quad \forall \tau_2\in T.
\end{equation*}
Now, fix $\tau_1\in T$, $\tau_1\geq 1$. We can always choose $\tau_2\in T$ such that
\[
 2\tau_1\leq \tau_2\leq 4\tau_1.
\]
Otherwise by the definition of $\tau$, we have
\[
 10C_2M\int_{2\tau_1}^{4\tau_1}\tau_+^{-1-\b}d\tau=\frac{6C_2M}{\b}((1+2\tau_1)^{-\b}-(1+4\tau_1)^{-\b})< C_2M(1+2\tau_1)^{-\b}.
\]
This is impossible as $\b\leq 1+\a<2$, $\tau_1\geq 1$. For such $\tau_1$ and $\tau_2$, the estimate \eqref{tauE} then implies that
\begin{equation}
\label{intedecay1} E[\phi](\tau_2)\les C_1(\tau_2)_+^{-1-\a}+M(\tau_2)_+^{-1-\b}+M(\tau_2)_+^{-1-\th\a},\quad\forall \tau_2\in T.
\end{equation}
Therefore from the integrated energy estimate \eqref{ILE0}, we can improve the integrated energy
\begin{equation*}
I^\ep[\phi]_{\tau_1}^{\tau_2}\les M(\tau_1)_+^{-1-\a}+M(\tau_1)_+^{-1-\b}+M(\tau_1)_+^{-1-\th\a} , \quad \forall \tau_1, \tau_2\in T.
\end{equation*}
As the set $T$ contains arbitrarily large $\tau$, the above estimate holds for all $\tau_2\geq \tau_1$, $\tau_1\in T$.
For general $\tau_1\geq 4$, note that we can choose $\tilde{\tau}_1\in T$ such that
\[
 \f12 \tau_1\leq \tilde{\tau}_1\leq \tau_1.
\]
Hence the above improved integrated energy inequality holds for all $0\leq \tau_1\leq \tau_2$. In particular, as estimate \eqref{ILEass} holds
for $\b=0$, we conclude that
\[
 I^\ep[\phi]_{\tau_1}^{\tau_2}\les M(\tau_1)_+^{-1} , \quad \forall \tau_1\leq \tau_2.
\]
That is the estimate \eqref{ILEass} holds for $\b=1$. Now from the definition of $\th$ and estimate \eqref{intedecay1} we again can show that
\eqref{ILEass} holds for
\[
\b=1+\min\{\a, \frac{2\a_1}{2+\a_1+\ep}\}.
\]
Recall that
\[
\frac{2\a+\a\ep}{2-\a}\leq \a_1.
\]Therefore estimate \eqref{intedecay1} holds for
\[
\b=1+\min\{\a, \frac{2\a_1}{2+\a_1+\ep}\}=1+\a.
\]
We thus finished the proof for the proposition.
\end{proof}

\bigskip

\section{Bootstrap assumptions}
The semilinear term $F$ in the equation \eqref{QUASIEQ} has already been discussed in \cite{yang1}. The quasilinear part $g^{\mu\nu\ga}\pa_\ga\phi\pa_{\mu\nu}\phi$ satisfies
the null condition. Cubic or higher order terms are always easy to handle for nonlinear wave equations. To
simplify the proof of the main Theorem \ref{maintheorem} but without loss of generality, instead of the general equation \eqref{QUASIEQ}, we consider the
following simple model of quasilinear wave equations
\begin{equation}
 \label{QUASIEQsim}
\begin{cases}
 \Box_g \phi+g^{\mu\nu\ga}\pa_\ga\phi\cdot \pa_{\mu\nu }\phi=0,\\
\phi(0, x)=\phi_0(x), \quad \pa_t\phi(0, x)=\phi_1(x),
\end{cases}
\end{equation}
on the time dependent inhomogeneous background $(\mathbb{R}^{3+1}, g)$ where the metric $g$ satisfies the estimates \eqref{HHqu} and $g^{\mu\nu\ga}$
are constants satisfy the null condition. We have to point out here that although we write the quasilinear part as $g^{\mu\nu\ga}\pa_\ga\phi\pa_{\mu\nu}\phi$, the null structure will
never be used in the region $\{|x|\leq R\}$. In this region the nonlinear term can be any quadratic form of $\phi$, $\pa\phi$.

We assume $\delta_0$ is sufficiently small, depending only on $\a$,
such that Proposition \ref{ILEthm}, Proposition \ref{ILEthmpw} and Proposition \ref{propILED}
hold. Recall that the initial data
 $(\phi_0, \phi_1)$ are smooth and are supported on $\{|x|\leq R\}$.
We use bootstrap argument to prove the main Theorem \ref{maintheorem}.

First we fix the foliation $\Si_\tau$ by choosing the radius $R$ to be one
appeared in the assumption \eqref{HHqu} for the background metric
$g$. We start with the following bootstrap assumptions
 on the solution $\phi$. On $S_\tau$ ($r\geq R$), we assume
\begin{equation}
 \label{boostphiout}
\begin{split}
 &\sum\limits_{|k|\leq 4}\int_\om |\Lb Z^k\phi|^2d\om+\sum\limits_{|k|\leq 3}\int_\om |\pa\Lb Z^k\phi|^2d\om\leq 2  H^2,\\
&\sum\limits_{|k|\leq 4}\int_\om |\overline{\pa_v} Z^k\phi|^2d\om+\sum\limits_{|k|\leq 3}\int_\om |\pa\overline{\pa_v} Z^k\phi|^2d\om\leq 2\bar H^2,
\end{split}
\end{equation}
where
\[
 \bar H=\delta_0(1+|x|)^{-1-2\a},\quad H=\bar H+\delta_0(1+|x|)^{-\f12-2\a}(1+\tau)^{-\f12-\f12\a}
\]
and $\tau$ is the parameter of the foliation $\Si_{\tau}$. When
$r\leq R$, we assume
\begin{align}
 \label{boostphiinout}
&\sum\limits_{|k|\leq 4}\int_\om |\pa Z^k\phi|^2d\om+\sum\limits_{|k|\leq 3}\int_\om |\pa^2 Z^k\phi|^2d\om\leq 2 \bar H^2 , \quad 1 \leq r\leq R,\\
\label{boostphiinin}
&\sum\limits_{|k|\leq 3} |\pa Z^k\phi|^2+\sum\limits_{|k|\leq 2} |\pa^2 Z^k\phi|^2\leq 2\delta_0^2,\quad |x|\leq 1.
\end{align}
To close the above bootstrap assumptions, we commute the equation \eqref{QUASIEQ} with $Z$ for $k$ times and show the decay of the integrated local energy of $Z^k\phi$. We then use
Sobolev embedding when $r\geq 1$ and elliptic estimates when $r\leq 1$ to improve the above bootstrap assumptions. That is we will show that
the above bootstrap assumptions hold if we replace $2$ with $E_0 C$ for some constant $C$ depending only on
$R$, $\a$. Therefore if $E_0$ is sufficiently small, we can improve the bootstrap assumptions and conclude the main theorem.

\bigskip

Before we go to the estimates for the integrated energy decay for $Z^k\phi$, we prove several lemmas which will be used in the sequel. First
of all, we choose the small positive constant $\ep$, $\a_1$, $\a_2$ satisfying the conditions in Proposition \ref{ILEthmpw}, where $\ep$ is much
smaller than $\a$ and $0<\ep<\a<\a_1<\a_2$. All the implicit constants appeared in the sequel may depend on these small constants. However,
since the choice of $\ep$, $\a_1$, $\a_2$ depends only on $\a$, we can let $\a$ to be the representative for the dependence of the implicit constants
in the sequel. The only point we need to emphasize is that since we want to show that the smallness of $\delta_0$ is independent of $R$, we may
have to keep track of the dependence of $R$. From now on, unless we point it out, the implicit constant $A\les B$ depends only on $\a$.

We consider the solutions of the linear wave equation
\[
 \Box_g\phi+N(\phi)=F
\]
with the metric $g$ and $N$ satisfying the estimate \eqref{HH}. The first lemma will be used to show the pointwise decay of the solution when $r\leq 1$.
\begin{lem}
 \label{lemH2estin1}
Let $\phi$ be the solution of the linear wave equation
\eqref{LWAVEEQ}. Then we have
\begin{align*}
\int_{\tau_1}^{\tau_2}\int_{r\leq 2}|\pa^2
\phi|^2dxd\tau\leq C_{\a} (
D^{\a_1}[F]_{\tau_1}^{\tau_2}+I^\ep[Z\phi]_{\tau_1}^{\tau_2}+I^\ep[\phi]_{\tau_1}^{\tau_2})
\end{align*}
for some constant $C_\a$ depending only on $\a$.
\end{lem}
\begin{proof}
Note that $g^{ij}$ is uniformly elliptic for sufficiently
small $\delta_0$. From the equation \eqref{LWAVEEQ}, we derive by using elliptic estimates that
\begin{align*}
 &\int_{\tau_1}^{\tau_2}\int_{r\leq 1}|\pa^2\phi|^2dxd\tau\les\int_{\tau_1}^{\tau_2}\int_{r\leq 2}\sum\limits_{i, j=1}^{3}|\pa_{ij}\phi|^2dxd\tau+
\int_{\tau_1}^{\tau_2}\int_{r\leq 2}|\pa\pa_t\phi|^2dx\\
&\les\int_{\tau_1}^{\tau_2}\int_{r\leq 4}|\sum\limits_{i,
j=1}^{3}g^{ij}\pa_{ij}\phi|^2+|\pa\pa_t \phi|^2+|\phi|^2dxd\tau\\
&\les\int_{\tau_1}^{\tau_2}\int_{r\leq
4}(1+r)^{1+\a_1}|F|^2dxd\tau+\int_{\tau_1}^{\tau_2}\int_{r\leq
4}\frac{|\pa\pa_t \phi|^2+
|\bar\pa\phi|^2}{(1+r)^{1+\ep}}dxd\tau\\
&\les D^{\a_1}[F]_{\tau_1}^{\tau_2}+I^\ep[Z\phi]_{\tau_1}^{\tau_2}+I^\ep[\phi]_{\tau_1}^{\tau_2}.
\end{align*}
This proves the Lemma.
\end{proof}
For a symmetric two tensor $k^{\mu\nu}$, we may need to decompose the differential operator $k^{\mu\nu}\pa_{\mu\nu}$ with respect to the null frame $\{L, \Lb, S_1, S_2\}$.
\begin{lem}
 \label{lemopde}
Assume $k^{\mu\nu}=k^{\nu\mu}$. Then we have
\[
|k^{\mu\nu}\pa_{\mu\nu}\phi|\leq |k^{\Lb\Lb}||\Lb Z\phi|+|k|(|\overline{\pa_v}Z\phi|+|L\Lb\phi|+r^{-1}|\pa\phi|), \quad r=|x|\geq 1.
\]
where $k^{\Lb\Lb}=k^{\mu\nu}\f12\Lb_{\mu}\f12 \Lb_{\nu}$ and $|k|=\sum\limits_{\mu, \nu}|k^{\mu\nu}|$.
\end{lem}
\begin{proof}
 We decompose the derivative $\pa_\mu$ relative to the null frame $\{L, \Lb, S_1, S_2\}$
\[
 \pa_\mu=\nabb_\mu+\f12 \Lb_\mu\Lb+\f12 L_\mu L, \quad L_0=1, \quad L_i=\frac{x_i}{r},
\]
where $\nabb_\mu$ is a linear combination of $S_1$ and $S_2$. Note
that $L(\Lb_\mu)=L(L_\mu)=\Lb(\Lb_\mu)=\Lb(L_\mu)=0$. We can compute
\begin{align*}
\pa_{\mu\nu}&=(\f12 \Lb_\mu\Lb+\f12 L_\mu L+\nabb_\mu)(\f12
\Lb_\nu\Lb+\f12 L_\nu L+\nabb_\nu)\\
&=\f12\Lb_\mu\f12\Lb_\nu \Lb\Lb+\f12 \Lb_\mu L_\nu L\Lb+\f12
\Lb_\mu\Lb\nabb_\nu+\f12 L_\mu \f12\Lb_\nu L\Lb\\
&\quad +\f12 L_\mu \f12 L_\nu LL+\f12 L_\mu L\nabb_\nu
+\nabb_\mu(\f12 \Lb_\nu \Lb)+\nabb_\mu(\f12 L_\nu
L)+\nabb_\mu\nabb_\nu.
\end{align*}
Recall that $L=2\pa_t -\Lb$. We have
\begin{align*}
k^{\mu\nu}\pa_{\mu\nu}&=k^{\Lb\Lb}\Lb\Lb+k^{LL}LL+2k^{\Lb L}L\Lb+k^{\mu\nu}\nabb_\mu\nabb_\nu+k^{\mu\nu}\Lb_\mu \Lb\nabb_\nu
+k^{\mu\nu}L_\mu L\nabb_\nu\\
&\quad +\f12 k^{\mu\nu}(\nabb_\mu \Lb_\nu \cdot \Lb+\nabb_\mu L_\nu
\cdot L+\Lb_\nu [\nabb_\mu, \Lb]+L_\nu [\nabb_\mu, L])\\
&=2k^{\Lb\Lb}\Lb\pa_t+2k^{LL}L\pa_t+(2k^{\Lb L}-k^{LL}-k^{\Lb\Lb})L\Lb+k^{\mu\nu}\nabb_\mu\nabb_\nu+2k^{\mu\nu}\Lb_\mu
\nabb_\nu\pa_t\\
&\quad +k^{\mu\nu}(L_\mu-\Lb_\mu) L\nabb_\nu +\f12 k^{\mu\nu}(\nabb_\mu \Lb_\nu \cdot \Lb+\nabb_\mu L_\nu \cdot L+\Lb_\nu [\nabb_\mu, \Lb]+L_\nu
[\nabb_\mu, L]).
\end{align*}
Note that
\[
|\nabb_\mu \Lb_\nu|\leq r^{-1},\quad |[\nabb_\mu, \Lb]\phi|\leq
r^{-1}|\pa\phi|,\quad \nabb=r^{-1}\Om, \quad r\geq 1.
\]
The Lemma then follows.
\end{proof}

The following lemma gives the estimate for $L\Lb\phi$.
\begin{lem}
 \label{LLbphi}
We have
\begin{align*}
&|L\Lb\phi|\leq C_\a ( \frac{|\pa \phi|+|\pa
Z\phi|}{r}+ \tau_+^{-\f12-\f12\a}(|\overline{\pa_v}Z\phi|+|\overline{\pa_v}\phi|)+|F|),\quad (t, x)\in S_\tau;\\
&|L\Lb\phi|\leq C_\a (\frac{|\pa \phi|+|\pa Z\phi|}{r}+ |F|),\quad
1\leq r\leq R.
\end{align*}
\end{lem}
\begin{proof}
 Write the equation \eqref{LWAVEEQ} in null coordinates
\[
 -L\Lb \phi+\frac{2}{r}\pa_r \phi+\lap \phi+h^{\mu\nu}\pa_{\mu\nu}\phi+\tilde N(\phi)=F,
\]
where
\[
 \tilde{N}^{\mu}=N^\mu+\frac{1}{\sqrt{-G}}\pa_v(g^{\mu\nu}\sqrt{-G}).
\]
When $r\geq R$, we can show that
\[
 |\tilde{N}(\phi)|\les r^{-1}|\pa\phi|+\tau_+^{-\f12-\f12 \a}|\overline{\pa_v}\phi|.
\]
Using Lemma \ref{lemopde} to control $h^{\mu\nu}\pa_{\mu\nu}\phi$,
we have pointwise bound for $L\Lb \phi$
\begin{align*}
|L\Lb \phi|\les&\frac{|\pa\phi|+|\nabb
\Om\phi|}{r}+ |h^{\Lb\Lb}||\pa
Z\phi|+|h|(|\overline{\pa_v}Z\phi|+|L\Lb \phi|)+|\tilde{N}(\phi)|+|F|.
\end{align*}
Since $|h|\leq H+\bar H\leq \delta_0$, $|h^{\Lb\Lb}|\leq \bar H\leq r^{-1}$ and $\delta_0$ is small, the above inequality implies that
\begin{align*}
|L\Lb \phi|\les &\frac{|\pa \phi|+|\pa
Z\phi|}{r}+ |h||\overline{\pa_v}Z\phi|+|\tilde{N}(\phi)|+|F|.
\end{align*}
The Lemma then follows from the assumption \eqref{HH}.
\end{proof}

For any two functions $\Phi$, $\phi$, we
denote the null form
\[
 Q(\Phi, \phi)=g^{\mu\nu\ga}\pa_\ga\Phi \cdot \pa_{\mu\nu}\phi
\]
for constants $g^{\mu\nu\ga}$ satisfying the null condition. To simplify the notation, for another set of constants $\tilde{g}^{\mu\nu\ga}$ satisfying the null condition, we
still use $Q(\Phi,\phi)$ to denote $\tilde{g}^{\mu\nu\ga}\pa_\ga\Phi \cdot \pa_{\mu\nu}\phi$.
\begin{lem}
 \label{nullstr}
Let $g^{\mu\nu\ga}$ be constants satisfying the null condition. Then for any two smooth functions $\Phi$, $\phi$, we have
\begin{align*}
&|Q(\Phi, \phi)|\les |\overline{\pa_v}\Phi||\pa Z\phi|+|\pa\Phi|(
|\overline{\pa_v}Z\phi|+|L\Lb\phi|+r^{-1}|\pa\phi|) ,\quad |x|\geq 1,\\
&ZQ(\Phi, \phi)=Q(Z\Phi,\phi)+Q(\Phi,Z\phi)+Q(\Phi, \phi).
\end{align*}
The last term $Q(\Phi, \phi)$ should be interpreted as $\tilde{g}^{\mu\nu\ga}\pa_\ga \Phi\pa_{\mu\nu}\phi$ for new constants $\tilde{g}^{\mu\nu\ga}$ satisfying the null condition.
\end{lem}
\begin{proof}
The null condition $g^{\mu\nu\ga}\xi_\ga\xi_\mu\xi_\nu=0$ whenever $\xi_0^2=\xi_1^2+\xi_2^2+\xi_3^2$ implies that
 $\Lb\Phi\cdot \Lb\Lb\phi$ will not appear in the decomposition of the null from $Q(\Phi, \phi)$ relative to the null frame
$\{L, \Lb, S_1, S_2\}$. Using Lemma \ref{lemopde}, we can get the first inequality. For the second inequality, we note that
$Zr=0$, $[Z, \Lb]=[Z,L]=0$.
\end{proof}

\section{Integrated Local Energy Decay for $Z^k\phi$}
Since the initial data for the simplified quasilinear wave equation \eqref{QUASIEQsim} have compact support, from Proposition \ref{propILED} we
conclude that under the bootstrap assumptions \eqref{boostphiinin}, \eqref{boostphiinout}, \eqref{boostphiout} the integrated energy
$I^{\ep}[\phi]_{\tau_1}^{\tau_2}$ for $\phi$ decays in $\tau_1$. As having discussed in the previous section, to close the bootstrap assumptions,
we need to show the decay of the integrated energy for higher order derivatives of the solution. We thus can commute the equation with the
vector fields $Z=\{\Om, \pa_t\}$. However, after commuting the equation with $Z$, the resulting equation is not of the form in Proposition
\ref{propILED} (an addition second order derivative term $k^{\mu\nu}\pa_{\mu\nu}$ appears). That is we are not able to apply Proposition
\ref{propILED} directly to obtain the decay of the integrated energy for $Z\phi$. Below we consider the equations for $Z\phi$, and show that the
integrated energy for $Z\phi$ also decays.

Let $\phi$ be the solution of the following linear wave equation
\begin{equation}
\label{LWAVEEQh} \Box\phi+h^{\mu\nu}\pa_{\mu\nu}\phi+N^\mu\pa_\mu\phi=F
\end{equation}
where $h^{\mu\nu}$, $N^\mu $ satisfy the estimates \eqref{HH} (but with $\delta_1=\delta_0$). We have the equation for $Z\phi$
\begin{equation}
\label{LwaveeqZ} \Box Z\phi+h^{\mu\nu}\pa_{\mu\nu} Z\phi+\tilde{h}^{\mu\nu}\pa_{\mu\nu}\phi+N^\mu\pa_\mu(Z\phi)+\tilde{N}^\mu\pa_\mu\phi=ZF,
\end{equation}
where
\[
\tilde{h}^{\mu\nu}\pa_{\mu\nu}\phi=Z(h^{\mu\nu}\pa_{\mu\nu}\phi)-h^{\mu\nu}\pa_{\mu\nu}Z\phi,\quad \tilde{N}=[Z, N].
\]
We assume that $\tilde{h}$, $\tilde{N}^{\mu}$ satisfy the following estimates
\begin{equation}
\label{HHz}
\begin{split}
&|\tilde{N}^{\Lb}|+|\tilde{h}^{\Lb\Lb}|\leq \bar H, \quad |\tilde{N}^{\mu}|+|\tilde{ h}^{\mu\nu}|\leq H,\quad (t, x)\in S_\tau;\\
 &|\tilde{N}^{\mu}|+|\tilde{ h}^{\mu\nu}|\leq \bar H,\quad |x|\leq R.
\end{split}
\end{equation}
Here we recall that
\[
\tilde{h}^{\Lb\Lb}=\tilde{h}^{\mu\nu}\f12\Lb_\mu \f12 \Lb_\nu,\quad \tilde{N}^{\Lb}=\tilde{N}^{\mu}\f12 \Lb_{\mu}.
 \]
Denote
\[
E_0=\tilde{E}[\phi](0)+\tilde{E}[Z\phi](0)+S^\ep[\phi](0)+S^\ep[Z\phi](0).
\]
\begin{prop}
\label{propILEDz} Assume $h^{\mu\nu}$, $N^\mu$ satisfy the estimates \eqref{HH}. $\tilde{h}^{\mu\nu}$, $\tilde{N}^\mu$ are defined as above and
satisfy the similar estimates \eqref{HHz}. Let $\phi$ be the solution of the linear wave equation \eqref{LWAVEEQh}. Assume
\[
D^{\a_1}[F]_{\tau_1}^{\tau_2}+D^{\a_1}[ZF]_{\tau_1}^{\tau_2}\leq C_1(1+\tau_1)^{-1-\a},\quad \forall \tau_2\geq \tau_1
\]
for some constant $C_1$. If $\delta_0$ is sufficiently small, depending only on $\a$, then
\begin{align*}
&I^\ep
[Z\phi]_{\tau_1}^{\tau_2}+D^{\a_1}[\tilde{h}^{\mu\nu}\pa_{\mu\nu}\phi]_{\tau_1}^{\tau_2}+D^{\a_1}[\tilde{N}^{\mu}\pa_{\mu}Z\phi]_{\tau_1}^{\tau_2}+
E^{1+\a}[Z\phi]_{\tau_1}^{\tau_2}\\
&\leq C_{\a, R} (C_1+E_0)(\tau_1)_+^{-1-\a},\quad \forall \tau_2\geq \tau_1
\end{align*}
for some constants $C_{\a, R}$ depending on $\a$, $R$.
\end{prop}
\begin{proof}
For simplicity, in the proof we denote
 \[
 E_1=C_1+E_0,\quad \bar N=N^{\mu}\pa_\mu Z\phi,\quad F_h=\tilde{h}^{\mu\nu}\pa_{\mu\nu}\phi,\quad F_1=ZF-\tilde{N}(\phi).
 \]
We move $\bar{N}+F_h$ to the right hand side of the equation \eqref{LwaveeqZ} and treat it as inhomogeneous term. Using the smallness of
$\delta_0$ we will absorb $F_h$ and $\tilde{N}$. And then the decay of $I^\ep[Z\phi]_{\tau_1}^{\tau_2}$ follows from the same argument for
proving the decay of $I^\ep[\phi]_{\tau_1}^{\tau_2}$ in Proposition \ref{propILED}. The main difficulty is that we need to show that the
smallness of $\delta_0$ depends only on $\a$. Note that the implicit constants in the integrated energy estimate \eqref{ILE0} and the energy
estimate \eqref{eb} depend only on $\a$. We mainly rely on these two estimates to control $F_h$ and $\bar{N}$.

First using the estimates \eqref{HHz} and Proposition \ref{propILED} we can show that
\begin{align*}
D^{\a_1}[\tilde{N}(\phi)]_{\tau_1}^{\tau_2}&\leq C_\a( I^{\ep}[\phi]_{\tau_1}^{\tau_2}+E^{1+\a}[\phi]_{\tau_1}^{\tau_2})\leq E_1 C_{\a,
R}(\tau_1)_+^{-1-\a}(1+\int_{0}^{\tau_2}E[\phi](\tau)d\tau)\\
&\leq E_1 C_{\a, R}(\tau_1)_+^{-1-\a}.
\end{align*}
We have used the $p$-weighted energy inequality \eqref{pwe1} in the last step. In particular, we have
\begin{equation}
\label{DFh}
 D^{\a_1}[F_1]_{\tau_1}^{\tau_2}=D^{\a_1}[ZF-\tilde{N}(\phi)]_{\tau_1}^{\tau_2} \leq E_1 C_{\a, R}(\tau_1)_+^{-1-\a}.
\end{equation}
Using Lemma \ref{lemopde} and the assumption \eqref{HHz}, we can estimate
\begin{equation}
\label{estH1k}
 |F_h|+|\tilde{N}|\les
\begin{cases}
 \bar H|\pa Z\phi|+H(|\overline{\pa_v}Z\phi|+|L\Lb \phi|+r^{-1}|\pa\phi|), \quad |x|\geq R;\\
  \bar H(|\pa Z\phi|+|L\Lb \phi|+r^{-1}|\pa\phi|),\quad 1\leq r<R;\\
  \bar H(|\pa^2\phi|+|\pa\phi|).
\end{cases}
\end{equation}
Then using Lemma \ref{LLbphi} to bound $L\Lb \phi$, we can show that
\begin{align*}
 D^{\a_1}[F_h]_{\tau_1}^{\tau_2}+D^{\a_1}[\bar{N}]_{\tau_1}^{\tau_2}\les \delta_0^2&(\int_{\tau_1}^{\tau_2}\int_{r\leq 1}|\pa^2 \phi|^2 dxd\tau+
\int_{\tau_1}^{\tau_2}\int_{\Si_\tau\cap \{r\geq 1\}}
|L\Lb\phi|^2dxd\tau\\&+ E^{1+\a}[Z\phi]_{\tau_1}^{\tau_2}+I^\ep[Z\phi]_{\tau_1}^{\tau_2}+I^\ep[\phi]_{\tau_1}^{\tau_2})\\
\les \delta_0^2 &(D^{\a_1}[F]_{\tau_1}^{\tau_2}+ E^{1+\a}[Z\phi]_{\tau_1}^{\tau_2} +I^\ep[Z\phi]_{\tau_1}^{\tau_2}+I^\ep [\phi]_{\tau_1}^{\tau_2}).
\end{align*}
Here we have used Lemma \ref{lemH2estin1} to estimate $\pa^2\phi$ in $\{|x|\leq 1\}$. To estimate $E^{1+\a}[Z\phi]_{\tau_1}^{\tau_2}$ (see the definition in Section 2),
set $\tau_2=\tau$ in the energy inequality \eqref{eb} and multiply both side with $\tau_+^{-1-\a}$ and then integrate with respect to $\tau$
from $\tau_1$ to $\tau_2$. We can derive
\begin{equation*}
 E^{1+\a}[Z\phi]_{\tau_1}^{\tau_2}\les \tilde{E}[Z\phi](\tau_1)+S^\ep [Z\phi](\tau_1)+I^\ep[Z\phi]_{\tau_1}^{\tau_2}+D^{\ep}[F_1-F_h-\bar{N}]_{\tau_1}^{\tau_2}.
\end{equation*}
Here recall that $Z\phi$ satisfies the above linear wave equation \eqref{LwaveeqZ}. Now from the previous estimate we obtain
\[
 D^{\a_1}[F_1-F_h-\bar{N}]_{\tau_1}^{\tau_2}\les D^{\a_1}[F_1]_{\tau_1}^{\tau_2}+\delta_0^2(D^{\a_1}[F]_{\tau_1}^{\tau_2}+ E^{1+\a}[Z\phi]_{\tau_1}^{\tau_2}
+I^\ep[Z\phi]_{\tau_1}^{\tau_2}+I^\ep [\phi]_{\tau_1}^{\tau_2}).
\]
Let $\delta_0$ be sufficiently small depending only on $\a$ (the implicit constant depends only on $\a$). We can absorb $E^{1+\a}[Z\phi]_{\tau_1}^{\tau_2}$ and thus to derive
\begin{align}
\label{inteEn}
E^{1+\a}[Z\phi]_{\tau_1}^{\tau_2}&\les \tilde{E}[Z\phi](\tau_1)+S^\ep [Z\phi](\tau_1)+I^\ep[Z\phi]_{\tau_1}^{\tau_2}+\tilde{D}_{\tau_1}^{\tau_2}, \\
 \label{DFh3}
D^{\a_1}[F_1-F_h-\bar{N}]_{\tau_1}^{\tau_2}&\les\delta_0^2(\tilde{E}[Z\phi](\tau_1)+S^\ep [Z\phi](\tau_1)+I^\ep[Z\phi]_{\tau_1}^{\tau_2})+\tilde{D}_{\tau_1}^{\tau_2},
\end{align}
where we denote
\[
 \tilde{D}_{\tau_1}^{\tau_2}=D^{\a_1}[F_1]_{\tau_1}^{\tau_2}+D^{\a_1}[F]_{\tau_1}^{\tau_2}+I^\ep[\phi]_{\tau_1}^{\tau_2}.
\]
From estimate \eqref{DFh} and Proposition \ref{propILED}, we have
\begin{equation}
\label{tildeDdecay}
 \tilde{D}_{\tau_1}^{\tau_2}\leq E_1 C_{\a, R} (\tau_1)_+^{-1-\a}.
\end{equation}
We now use the above estimates \eqref{inteEn}, \eqref{DFh3} to simplify the integrated energy estimate \eqref{ILE0}, the energy estimate \eqref{eb} as well as the $p$-weighted energy
inequalities \eqref{pwe1a}, \eqref{pwe1} for $Z\phi$. For the integrated energy estimate and the energy estimate, from \eqref{DFh3}, for sufficiently small $\delta_0$, we have
\[
 E[Z\phi](\tau_2)+I^\ep[Z\phi]_{\tau_1}^{\tau_2}+\int_{\tau_1}^{\tau_2}\int_{S_\tau}\frac{|\nabb Z\phi|^2}{r}dxd\tau\les E[Z\phi](\tau_1)
+\delta_0 S^\ep [Z\phi](\tau_i)+\tilde{D}_{\tau_1}^{\tau_2}.
\]
Here since $E^N[Z\phi]_0^\infty$ is finite, all the estimates hold if we replace $\tilde{E}[Z\phi](\tau)$ with $E[Z\phi](\tau)$. Now from the $p$-weighted energy inequalities
\eqref{pwe1}, we obtain the $p$-weighted energy estimate when $p=1$ for $Z\phi$
\begin{align*}
&g^{1}[Z\phi](\tau_2)+\int_{\tau_1}^{\tau_2}E[Z\phi](\tau)d\tau\les g^1[Z\phi](\tau_1)+\delta_0^2\int_{\tau_1}^{\tau_2}(\tau)_+^{-\a}E[Z\phi](\tau)d\tau d\tau\\
&+C_R((\tau_1)_+^{1-\a}E[Z\phi](\tau_1)+(\tau_1)_+^{-\a}+\delta_0(\tau_i)_+^{1-\a}S^\ep[Z\phi](\tau_i)),
\end{align*}
where we use the estimate \eqref{DFh3} to bound the inhomogeneous term $F_1-F_h-\bar N$. Let $\delta_0$ to be sufficiently small, depending only on $\a$ (as the implicit constant
depends only on $\a$). We conclude from the above estimate that
\begin{equation}
\label{intE}
\begin{split}
 &\int_{\tau_1}^{\tau_2}E[Z\phi](\tau)d\tau\\&\les g^1[Z\phi](\tau_1)+C_R((\tau_1)_+^{1-\a}E[Z\phi](\tau_1)+(\tau_1)_+^{-\a}+\delta_0(\tau_i)_+^{1-\a}S^\ep[Z\phi](\tau_i)).
\end{split}
\end{equation}
Similarly, we obtain the $p$-weighted energy inequality when $p=1+\a_1$ for $Z\phi$
\begin{align*}
&g^{1+\a_1}[Z\phi](\tau_2)dvd\om d\tau \les g^{1+\a_1}[Z\phi](\tau_1)+\int_{\tau_1}^{\tau_2}E[Z\phi](\tau)d\tau\\
&+C_{R}((\tau_1)_+E[Z\phi](\tau_1)+\delta_0(\tau_i)_+ S^\ep[Z\phi](\tau_i)+E_1).
\end{align*}
Let $\tau_1=0$. From the previous estimate for the integral of the energy flux, we derive
\[
 g^{1+\a_1}[Z\phi](\tau)dvd\om d\tau \les C_{R}( E_1+\tau_+ S^{\ep}[Z\phi](\tau)).
\]
By Proposition \ref{ILEthm} we have
\[
I^{\ep}[Z\phi]_{0}^{\infty}\leq C_{\a, R}E_1.
\]
Then the above two $p$-weighted energy estimates for $Z\phi$ are sufficiently to prove the decay of the integrated energy for $Z\phi$ (the proof is then the same as the proof in
Proposition \ref{propILED}). That is
\[
I^\ep[Z\phi]_{\tau_1}^{\tau_2}\leq E_1 C_{\a, R}(\tau_1)_+^{-1-\a},\quad \forall \tau_2\geq \tau_1\geq 0.
\]
To finish the proof for Proposition \ref{propILEDz}, it suffices to prove the decay of $E^{1+\a}[Z\phi]_{\tau_1}^{\tau_2}$. Note that
\[
 E^{1+\a}[Z\phi]_{\tau_1}^{\tau_2}=\int_{\tau_1}^{\tau_2}\frac{E[Z\phi](\tau)}{(1+\tau)^{1+\a}}d\tau\leq (\tau_1)_+^{-1-\a}\int_{\tau_1}^{\tau_2}E[Z\phi](\tau)d\tau.
\]
Since we have shown
\[
\int_{\tau_1}^{\tau_2}S^\ep[Z\phi](\tau)d\tau\leq  I^\ep[Z\phi]_{\tau_1}^{\tau_2}\leq E_1 C_{\a, R}
 (\tau_1)_+^{-1-\a},
\]
we can choose $\tau_2$ arbitrarily large such that
\[
 S^\ep [Z\phi](\tau_2)\leq E_1 C_{\a, R}(\tau_2)_+^{-1-\a}.
\]
for some arbitrarily large $\tau_2$.
 Then in the $p$-weighted energy inequality \eqref{intE} set $\tau_1=0$, we derive
\[
\int_{0}^{\tau_2}E[Z\phi](\tau)d\tau\leq E_1 C_{\a, R}.
\]
This constant is independent of $\tau_2$. In particular, we have
\[
 \int_{\tau_1}^{\tau_2}E[Z\phi](\tau)d\tau\leq E_1 C_{\a, R}.
\]
Therefore we have
\[
 E^{1+\a}[Z\phi]_{\tau_1}^{\tau_2}\leq
 (\tau_1)_+^{-1-\a} \int_{\tau_1}^{\tau_2}E[Z\phi](\tau)d\tau\leq E_1 C_{\a, R}(\tau_1)_{+}^{-1-\a}.
\]
This finishes the proof for Proposition \ref{propILEDz}.
\end{proof}

We now consider the solution of the quasilinear wave equation \eqref{QUASIEQsim} under the bootstrap assumptions
\eqref{boostphiout}, \eqref{boostphiinout}, \eqref{boostphiinin}. We show that the
integrated energy for $Z^k\phi$, $k\leq 6$ decays.
\begin{prop}
\label{propILEDz6}
Let $\phi$ be the solution of \eqref{QUASIEQsim} with compactly supported initial data $\phi_0$, $\phi_1$ described in Theorem \ref{maintheorem}. Then
\[
I^\ep[Z^k\phi]_{\tau_1}^{\tau_2}\leq E_0 C_{\a, R} (1+\tau_1)^{-1-\a},\quad
\forall k\leq 6
\]
for some constant $C_{\a, R}$ depending on $R$, $\a$. Here $E_0$ is defined before Theorem \ref{maintheorem}.
\end{prop}

To show the integrated energy decay for $Z^k\phi$, we consider the equation of $Z^k\phi$ obtained by commuting the equation \eqref{QUASIEQsim} with $Z^k$. Let
$N$ be the vector field with components
\[
 N^\mu=\frac{1}{\sqrt{-G}}\pa_\nu(g^{\mu\nu}\sqrt{-G}).
\]
Then we can write the equation \eqref{QUASIEQsim} as
\[
 \Box\phi+Q(\phi, \phi)+h^{\mu\nu}\pa_{\mu\nu}\phi+N(\phi)=0
\]
Commute this equation with $Z^k$. We obtain the equation for $Z^k\phi$
\begin{equation}
\label{eqZphi} \Box Z^k\phi+Q(\phi, Z^k\phi)+h^{\mu\nu}\pa_{\mu\nu}
Z^k\phi+Q_1^k+H_1^k+Q(Z^k\phi, \phi)+N(Z^k\phi)=F^k_2.
\end{equation}
with the following definitions for $Q_1^k$, $H_1^k$, $F_2^k$. Using Lemma \ref{nullstr}, we let $Q_1^k$ be the collection of all
those terms containing $\pa_{\mu\nu}Z^{k-1}\phi$ in the expansion of
$Z^k Q(\phi, \phi)$. More precisely,
\[
Q_1^k=Q(Z\phi, Z^{k-1}\phi)+Q(\phi, Z^{k-1}\phi).
\]
We remark here that $Q$ denotes a general null form for constants $\tilde{g}^{\mu\nu\ga}$ satisfying the null condition. It may be different from
$g^{\mu\nu\ga}$ appeared in the equation \eqref{QUASIEQsim}. For example we in fact have
\[
Q(Z\phi, Z^{k-1}\phi)=kg^{\mu\nu\ga}\pa_\ga Z\phi\cdot\pa_{\mu\nu}Z^{k-1}\phi.
\]
Similarly, we let $H_1^k$ be the collection of all those terms in the expansion of $Z(h^{\mu\nu}\pa_{\mu\nu}\phi)$ containing $\pa_{\mu\nu}Z^{k-1}\phi$, which can be
given as follows
\[
 H_1^k=k(Z(h^{\mu\nu}\pa_{\mu\nu}Z^{k-1}\phi)-h^{\mu\nu}\pa_{\mu\nu}Z^k\phi).
\]
Finally, we denote
\begin{align*}
& Q^k_2=-\sum\limits_{k_2\leq k-2, k_1+k_2\leq k, k_1<k}Q(Z^{k_1}\phi,Z^{k_2}\phi),\quad H_2^k =-Z^k(h^{\mu\nu}\pa_{\mu\nu}\phi)+h^{\mu\nu}\pa_{\mu\nu}Z^k\phi+H_1^k\\
& Q^k=Q(Z^k\phi, \phi), N^k=N(Z^k\phi), F_2^k=Q_2^k+H_2^k-[Z^k, N]\phi,\quad F^k=F_2^k-Q_1^k-H_1^k.
 \end{align*}
We first check that $g^{\mu\nu\ga}\pa_\ga\phi$ satisfies the same estimates \eqref{HH}, \eqref{HHimp} as $h^{\mu\nu}$. Note that
\[
 g^{\Lb\Lb \ga}\pa_\ga\phi=g^{\mu\nu\ga}\f12 \Lb_\mu\f12\Lb_\nu\pa_\ga\phi,\quad g^{\mu\nu\ga}\Lb_\mu\Lb_\nu\Lb_\ga=0.
\]
The bootstrap assumption \eqref{boostphiout} together with the
Sobolev embedding on the unit sphere shows that
\[
 |g^{\Lb\Lb \ga}\pa_\ga\phi|+|\pa g^{\Lb\Lb \ga}\pa_\ga\phi|\leq 2\bar H,\quad |\nabb g^{\Lb\Lb \ga}\pa_\ga\phi|\leq |r^{-1}g^{\Lb\Lb \ga}\pa_\ga Z\phi|\leq 2 r^{-1}\bar H,
\quad r\geq R.
\]
The other estimates in \eqref{HH}, \eqref{HHimp} follow directly from the bootstrap assumptions
\eqref{boostphiout}, \eqref{boostphiinout}, \eqref{boostphiinin}
after using Sobolev embedding.

To apply Proposition \ref{propILEDz}, we can write the equation for $Z^{k-1}\phi$ as
\[
 \Box Z^{k-1}\phi+(g^{\mu\nu\ga}\pa_\ga \phi+h^{\mu\nu})\pa_{\mu\nu}Z^{k-1}\phi+(g^{\mu\nu\ga}\pa_{\mu\nu}\phi+N^\ga)\pa_\ga(Z^{k-1}\phi)=F^{k-1}.
\]
Then the equation for $Z^{k}\phi$ will be of the form \eqref{LwaveeqZ} if we denote $\tilde{h}^{\mu\nu}$,
be functions such that
\[
\tilde{h}^{\mu\nu}\pa_{\mu\nu}\phi=Q_1^k+H_1^k.
\]
The vector field $N^\mu$ there corresponds to $g^{\mu\nu\ga}\pa_{\mu\nu}\phi+N^\ga$ here. And $\tilde{N}$ is the $Z$ derivative of $N$. We
can check that $\tilde{h}^{\mu\nu}$, $\tilde{N}^{\mu}$, $N$ satisfy condition
\eqref{HHz}. In fact for $g^{\mu\nu\ga}\pa_\ga Z\phi$ or $g^{\nu\ga\mu}\pa_{\nu\ga}\phi$ contributed by the null
form $Q_1^k$, we can show that the condition \eqref{HHz} is satisfied by using Lemma \ref{nullstr} together with the bootstrap
assumptions. For the part from $H_1^k$, we have the
assumption \eqref{HHqu}. This implies that we can apply Proposition \ref{propILED} and \ref{propILEDz} to show the integrated energy decay of $Z^k\phi$.

In particular, when $k\leq 1$, we have $F_2^k=0$. Thus Proposition \ref{propILED} and \ref{propILEDz} imply that
\[
I^\ep[Z\phi]_{\tau_1}^{\tau_2}+D^{\a_1}[Q_1^1+H_1^1]_{\tau_1}^{\tau_2}+D^{\a_1}[Q^1+N^1]_{\tau_1}^{\tau_2}+E^{1+\a}[Z\phi]_{\tau_1}^{\tau_2}\leq E_0 C_{\a, R}
(\tau_1)_+^{-1-\a},\quad
k\leq 1.
\]
Here we recall that $Q^k=Q(Z^k\phi, \phi)$, $N^k=N(Z^k\phi)$.
We now use induction argument to show Proposition \ref{propILEDz6}.
We assume that for some fixed $k\leq 6$
\begin{equation}
\label{inductass}
\begin{split}
&I^\ep[Z^l\phi]_{\tau_1}^{\tau_2}+D^{\a_1}[Q_1^l+H_1^l]_{\tau_1}^{\tau_2}+D^{\a_1}[F_2^l]_{\tau_1}^{\tau_2}+D^{\a_1}[Q^l+N^l]_{\tau_1}^{\tau_2}+
E^{1+\a}[Z^l\phi]_{\tau_1}^{\tau_2}\\
&\leq E_0 C_{\a, R, k-1}(\tau_1)_+^{-1-\a},\quad \forall l\leq k-1.
\end{split}
\end{equation}
We have shown that this is true when $k=2$. Now we want to
show that the above estimate also holds for $l=k$.

First note that the induction assumption in particular implies that
\[
 D^{\a_1}[F^{l}]_{\tau_1}^{\tau_2}= D^{\a_1}[F_2^{l}-Q_1^{l}-H_1^{l}]_{\tau_1}^{\tau_2}\leq E_0C_{\a, R, k-1}(\tau_1)_+^{-1-\a},\quad \forall l\leq k-1.
\]
Therefore by Proposition \ref{propILEDz}, the estimate \eqref{inductass} holds for $k$ if we can show that
\begin{equation*}
D^{\a_1}[F^k_2]_{\tau_1}^{\tau_2}\leq D^{\a_1}[Q^k_2]_{\tau_1}^{\tau_2}+D^{\a_1}[H_2^k]_{\tau_1}^{\tau_2}+D^{\a_1}[[Z^k, N]\phi]_{\tau_1}^{\tau_2}\leq E_0 C_{\a, R,k}(\tau_1)_+^{-1-\a},
\end{equation*}
which follows from the following two lemmas.
\begin{lem}
 \label{D3aH2k}
Under the induction assumption \eqref{inductass}, we have
\[
D^{\a_1}[H_2^k]_{\tau_1}^{\tau_2} +D^{\a_1}[[Z^k, N]\phi]_{\tau_1}^{\tau_2}\leq E_0 C_{\a, R, k}(\tau_1)_+^{-1-\a}.
\]
\end{lem}
\begin{proof}
We use condition \eqref{HHqu} and Lemma \ref{lemopde}. We can show that
\begin{equation}
\label{H2kzkn}
 |H_2^k|+|[Z^k, N]\phi|\les
 \begin{cases}
 \sum\limits_{l\leq k-1}\bar H |\pa Z^{l}\phi|+H(|\overline{\pa_v}Z^l\phi|+r^{-1}|\pa Z^l\phi|+|L\Lb Z^{l-1}\phi|),r\geq 1;\\
\bar H  \sum\limits_{l\leq k-2}|\pa^2 Z^{l}\phi|+|\pa Z^{l+1}\phi|,\quad r<1.
\end{cases}
\end{equation}
Thus we have
\begin{equation*}
 \begin{split}
D^{\a_1}[H_2^k]_{\tau_1}^{\tau_2} +D^{\a_1}[[Z^k,N]\phi]_{\tau_1}^{\tau_2}\les&\sum\limits_{l\leq k-1} \int_{\tau_1}^{\tau_2}\int_{r\leq 1}|\pa^2 Z^{l-1}\phi|^2dxd\tau+I^\ep[Z^l\phi]_{\tau_1}^{\tau_2}\\
&+E^{1+\a}[Z^l\phi]_{\tau_1}^{\tau_2}+\int_{\tau_1}^{\tau_2}\int_{\Si_\tau\cap\{r\geq 1\}}|L\Lb Z^{l-1}\phi|^2dxd\tau.
 \end{split}
\end{equation*}
Here $Z^{-1}\phi=0$. Now using Lemma \ref{lemH2estin1} and the induction assumption \eqref{inductass}, we can estimate
\begin{align*}
 \int_{\tau_1}^{\tau_2}\int_{r\leq 1}|\pa^2 Z^{l}\phi|^2dxd\tau\les D^{\a_1}[F^l]_{\tau_1}^{\tau_2}+\sum\limits_{l_1\leq l+1} I^\ep[Z^{l_1}\phi]_{\tau_1}^{\tau_2}\leq
C_{\a, R, k} E_0(\tau_1)_+^{-1-\a},\forall l\leq k-2.
\end{align*}
Similarly using Lemma \ref{LLbphi} to control $L\Lb Z^{k-2}\phi$, we
get
\begin{align*}
 \int_{\tau_1}^{\tau_2}\int_{\Si_\tau\cap\{r\geq 1\}}|L\Lb Z^{l}\phi|^2dxd\tau&\les \sum\limits_{l_1\leq l+1}I^\ep[Z^{l_1}\phi]_{\tau_1}^{\tau_2}+E^{1+\a}
[Z^{l_1}]_{\tau_1}^{\tau_2}+D^{\a_1}[F^l]_{\tau_1}^{\tau_2}\\
&\leq C_{\a, R, k} E_0 (\tau_1)_+^{-1-\a},\quad \forall l\leq k-2.
\end{align*}
The proves the lemma.
\end{proof}
For $D^{\a_1}[Q_2^k]_{\tau_1}^{\tau_2}$, we have
\begin{lem}
 \label{D3aQ2k}
We have
\[
D^{\a_1}[Q_2^k]_{\tau_1}^{\tau_2} \leq E_0 C_{\a, R, k}(\tau_1)_+^{-1-\a}.
\]
\end{lem}
\begin{proof}
The proof will be the same to that of the previous lemma once we can estimate the null form. Recall the definition of $Q_2^k$ after equation \eqref{eqZphi}. It suffices to consider $Q(Z^{k_1}\phi, Z^{k_2}\phi)$ for some pair $(k_1, k_2)$ such that
$k_1+k_2\leq k$, $k_2\leq k-2$, $k_1\leq k-1$. We need a Sobolev embedding to estimate the null form.
We claim that for such pair $(k_1, k_2)$, we always have
\begin{equation*}
\begin{split}
 \int_{\om}|\pa Z^{k_1}\phi|^2|\pa^2 Z^{k_2}\phi|^2d\om\les & \sum\limits_{l\leq 3}\int_{\om}|\pa^2 Z^l\phi|^2d\om \cdot \sum\limits_{l\leq k-1}\int_{\om}|\pa Z^{l}\phi|d\om\\
&+
\sum\limits_{l\leq 4}|\pa Z^l\phi|^2d\om\cdot
\sum\limits_{l\leq k-2}\int_{\om}|\pa^2 Z^l\phi|^2d\om.
\end{split}
\end{equation*}
We only prove the above claim for the case $k_1+k_2=k=6$, $k_1\leq 5$, $k_2\leq 4$. If $k_2\leq 1$ or $k_1\leq 3$ , the above inequality follows from Sobolev
 embedding on the unit sphere.
If $k_1=4$, $k_2=2$, we use
\[
 \|\pa Z^{k_1}\phi\pa^2 Z^{k_2}\phi\|_{L^2(\mathbb{S}^2)}\les \|\pa Z^{k_1}\phi\|_{L^4(\mathbb{S}^2)}
\|\pa^2 Z^{k_2}\phi\|_{L^4(\mathbb{S}^2)}\les \|\pa
Z^{k_1}\phi\|_{H^1(\mathbb{S}^2)} \|\pa^2
Z^{k_2}\phi\|_{H^1(\mathbb{S}^2)}.
\]
Thus the above Sobolev embedding holds. Now using Lemma
\ref{nullstr} and the bootstrap assumptions \eqref{boostphiout},
\eqref{boostphiinout}, we can show that when $r\geq 1$
\begin{align}
\label{Qzk}
 &\int_{\om}|Q(Z^{k_1}\phi, Z^{k_2}\phi)|^2d\om\\
 \notag
&\les \sum\limits_{l\leq k-1} \int_{\om} H^2 |\overline{\pa_v} Z^{l}\phi|^2+\bar H^2|\pa Z^l\phi|^2+H^2|L\Lb Z^{l-1}\phi|^2
+H^2r^{-2}|\pa Z^l\phi|^2d\om.
\end{align}
When $r\leq 1$, note that $k_1+k_2\leq k\leq 6$. In particular, $k_1\leq 3$ or $k_2\leq 2$. Thus by the bootstrap assumption \eqref{boostphiinin}, we can get
\[
 |Q(Z^{k_1}\phi, Z^{k_2}\phi)|\les\bar H \sum\limits_{l\leq k-2}|\pa^2 Z^l\phi|+|\pa Z^{l+1}\phi|.
\]
Then the Lemma follows from the same argument for proving Lemma \ref{D3aH2k}.
\end{proof}

The above two lemmas together with Proposition \ref{propILEDz} implies that \eqref{inductass} holds for $l=6$. In particular we have Proposition \ref{propILEDz6}. From the proof,
we in fact can prove an important integrated energy
inequality for $L\Lb Z^k\phi$ with positive weights in $r$ which will be used to derive the pointwise decay of the derivative of the solution.
\begin{cor}
\label{corimprol} We have
\[
\int_{\tau_1}^{\tau_2}\int_{\Si_\tau\cap \{r\geq 1\}}r^{1-\ep}|L\Lb
Z^k\phi|^2dx d\tau\leq E_0 C_{\a, R}(\tau_1)_+^{-1-\a},\quad \forall
k\leq 5.
\]
\end{cor}
\begin{proof}
Using Lemma \ref{LLbphi}, we can show that for all $k\leq 5$
\begin{align*}
\int_{\tau_1}^{\tau_2}\int_{\Si_\tau\cap \{r\geq 1\}}r^{1-\ep}|L\Lb Z^k\phi|^2dx
d\tau&\les \int_{\tau_1}^{\tau_2}\int_{\Si_\tau}\frac{|\pa
Z^k\phi|^2+|\pa Z^{k+1}\phi|^2}{(1+r)^{1+\ep}}dx
d\tau+D^{\a_1}[F^{k}]_{\tau_1}^{\tau_2}\\
&\les \sum\limits_{l\leq 6}I^\ep [Z^l\phi]_{\tau_1}^{\tau_2}+D^{\a_1}[F^{k}]_{\tau_1}^{\tau_2}\leq E_0 C_{\a, R} (\tau_1)_+^{-1-\a}.
\end{align*}
\end{proof}

\section{Pointwise estimates}
The corollary in the end of the last section plays an important role in showing the pointwise
estimates for the solution when $\{r\geq 1\}$.
Next we use this
integrated energy decay estimate together with the $p$-weighted energy
estimates proven in Proposition \ref{ILEthm}, \ref{ILEthmpw} to obtain the pointwise
estimates for the solution $\phi$ and hence to close the bootstrap
assumptions \eqref{boostphiout}, \eqref{boostphiinin},
\eqref{boostphiinout}. We divide our argument into several steps.

In the argument below, the notation $A\les B$ means $A\leq C_{\a, R}B$ for some constant $C_{\a, R}$ depending only on $\a$, $R$.

First we estimate $Z^k\phi$ for $k\leq 5$. Since $Z$ can be $\pa_t$
or $\Om$, from Proposition \ref{propILEDz6}, we can bound
$S^\ep[Z^k\phi](\tau)$ as follows
\[
S^\ep[Z^k\phi](\tau)\leq I^{\ep}[Z^k\phi]_{\tau}^{\tau+1}+I^{\ep}[Z^{k+1}\phi]_{\tau}^{\tau+1}\les E_0\tau_+^{-1-\a},\quad
\forall k\leq 5.
\]
In particular, the set \eqref{setT} defined in the proof of
Proposition \ref{propILED} is $[0, \infty)$ for all $Z^k\phi$,
$k\leq 5$. Hence the proof there implies that the energy decays for
all $\tau\geq 0$. That is
\begin{equation} \label{energybd}
 E[Z^k\phi](\tau)\les E_0\tau_+^{-1-\a},\quad \forall k\leq 5,\quad \forall \tau\geq 0.
\end{equation}
Here we have used the estimate \eqref{inductass} which holds for all
$l\leq 6$. Now let $\tau_1=0$ in the $p$-weighted energy inequality
\eqref{pwe1a} when $p=1+\a_1$, we have
\[
 g^{1+\a_1}[Z^k\phi](\tau)=\int_{S_\tau}r^{1+\a_1}|L(rZ^k\phi)|^2 dvd\om\les E_0 ,\quad \forall k\leq 5.
\]
From Lemma \ref{lempphi2}, we can show
\begin{equation*}
\int_{S_\tau}r^{1-\ep}|Z^k\phi|^2dvd\om\les E[Z^k\phi](\tau)+g^{1+\a_1}[Z^k\phi](\tau) \les E_0 ,\quad \forall k\leq 5.
\end{equation*}
The good derivative of the solution decays better. Quantitatively, from the previous estimate, we derive
\begin{equation}
 \label{r1aphi}
 \int_{S_\tau}r^{3-\ep}|LZ^k\phi|^2dvd\om \les g^{1-\ep}[Z^k\phi](\tau)+\int_{S_\tau}r^{1-\ep}|Z^k\phi|^2dvd\om\les E_0 ,\quad \forall k\leq 5.
\end{equation}
Using the decay estimates for the energy $E[Z^k\phi](\tau)$, $k\leq
5$, Lemma \ref{lem1} quickly yields the spherical average estimate
for $Z^k\phi$
\begin{equation}
\label{zkphi}
 \int_{\om}|Z^k\phi|^2(\tau, r,\om)d\om\les E[Z^k\phi](\tau)\les E_0 r^{-1}\tau_+^{-1-\a}, \quad \forall k\leq 5.
\end{equation}
Recall that $\nabb=r^{-1}\Om$. The above estimate in particular
implies the improved decay estimates for the angular derivative of
the solution
\begin{equation}
 \label{nabbzkphi}
\int_{\om}|\nabb Z^k\phi|^2d\om\leq r^{-2}\int_{\om}|\Om
Z^k\phi|^2d\om\les E_0 r^{-3}\tau_+^{-1-\a},\quad \forall
k\leq 4.
\end{equation}

\bigskip

Next, we estimate $L Z^k\phi$ and $\Lb Z^k\phi$ which we rely on
Corollary \ref{corimprol}. Corollary \ref{corimprol}
implies that
\begin{align}
\notag \int_{\Si_\tau\cap\{r\geq 1\}}r^{3-\ep}|L\Lb Z^k\phi|^2dvd\om&\leq \int_{\tau}^{\tau+1}\int_{\Si_\tau\cap\{r\geq 1\}}r^{1-\ep}(|L\Lb
Z^k\phi|^2+|L\Lb\pa_t
Z^k\phi|^2)dx d\tau\\
\label{llbzkphi} &\les E_0 (1+\tau)^{-1-\a},\quad \forall k\leq 4.
\end{align}
Recall that $L=2\pa_t-\Lb$ and $Z$ can be $\pa_t$. We can also obtain estimates for $LL Z^k\phi$. In fact, from estimate \eqref{r1aphi}, we can show that
\begin{align*}
& \int_{S_\tau}r^{3-\ep}|LLZ^k\phi|^2dvd\om\\
&\leq\int_{S_\tau}r^{3-\ep}|L\Lb Z^k\phi|^2dvd\om+
2\int_{S_\tau}r^{3-\ep}|\pa_{v}\pa_t Z^k\phi|^2dvd\om\les
E_0 ,\quad \forall k\leq 4.
\end{align*}
This estimate together with estimate \eqref{r1aphi} implies that
\begin{equation}
 \label{pavzkphi}
\int_{\om}|L Z^k\phi|^2d\om\les E_0  r^{-3+\ep}, \quad r\geq R,
\quad \forall k\leq 4.
\end{equation}
Here note that $\pa_v=L=\pa_t+\pa_r$. This estimate together with
\eqref{nabbzkphi} gives the estimate for $\overline{\pa_v}Z^k\phi$
when $r\geq R$, $k\leq 4$. To close the bootstrap assumptions, we also need to estimate $\Lb Z^k\phi$.

We consider $\Lb Z^k\phi$ on the larger domain $\Si_\tau\cap\{r\geq
1\}$. We first argue that
\begin{equation}
\label{decaynull} \liminf\limits_{v\rightarrow \infty}\int_{\om}|\Lb(Z^k\phi)|^2(u, v, \om) d\om=0, \quad k\leq 5.
\end{equation}
This follows immediately from the fact that
\[
\int_{S_\tau}\frac{|\pa Z^k\phi|^2}{(1+r)^{1+\ep}}r^2dvd\om \les E_0 \tau_+^{-1-\a},\quad k\leq 5.
\]
We can also see this as solutions of linear wave equations decays at null infinity. To solve our nonlinear equations, we see from the last
section of the previous chapter that we use Picard iteration process and approximate the solution by linear solutions. As linear solution decays
at null infinity, we have \eqref{decaynull}. Hence on $\Si_\tau\cap\{r\geq 1\}$, from the estimates \eqref{llbzkphi}, we can show that
\begin{align*}
 \int_{\om}|\Lb Z^k\phi|^2(\tau, v, \om)d\om&\leq\int_{v}^{\infty}\int_{\om}|L\Lb Z^k\phi|^2r^{3-\ep}dvd\om  \cdot
\int_{v}^{\infty}\int_{\om}r^{-3+\ep}dvd\om\\
&\les (2v-\tau+R)^{-2+\ep}\int_{\Si_\tau\cap\{r\geq 1\}}|L\Lb Z^k\phi|^2r^{1-\ep}dx d\om\\
&\les E_0 (1+\tau)^{-1-\a}r^{-2+\ep},\quad \forall k\leq 4,
\end{align*}
where we recall that $v=\frac{t+r}{2}=r+\frac{\tau-R}{2}$. We must
remark here that the above argument holds for $(\tau, v, \om)\in
S_\tau$. When $1\leq r\leq R$, splitting the integral into two
parts: integral on $S_\tau$ and integral on $r\leq R$, we can get
the same estimates. This gives the estimate for $\Lb Z^k\phi$ when
$r\geq 1$. Since $\Lb$, $Z=\{\pa_t, \Om\}$ can form a basis of the
tangent space at any point where $r\geq1$, a weaker estimate for
$\pa Z^k\phi$ can be that
\begin{equation}
\label{pazkphi}
\begin{split}
 \int_{\om}|\pa Z^k\phi|^2d\om &\les \int_{\om}|\pa_t Z^k\phi|^2+|\Lb Z^k\phi|^2d\om\\
 &\les E_0  r^{-1}(1+\tau)^{-1-\a},\quad r\geq
1,\quad \forall k\leq 4.
\end{split}
\end{equation}
This can be used to estimate $\pa Z^k\phi$ when $1\leq r\leq R$ (as $r\leq R$ we can improve the decay in $r$ to be $r^{-3+\ep}$ up to a
constant depending only on $R$).

\bigskip

The above discussion gives us the pointwise estimates (after using Sobolev embedding on the unit sphere) for the first order derivatives of the
solution. To close our bootstrap assumptions, we also need to estimate the second order derivative of the solution. We first consider the case
when $r\geq 1$. Note that
 \[
 |\pa^2 Z^k\phi|\les |L\Lb Z^k\phi|+|\pa
 Z^k\phi|+r^{-1}|Z^k\phi|,\quad r\geq 1.
 \]
 Thus to estimate the full second order derivative of the solution, it suffices to estimate $L\Lb Z^k\phi$. We rely on Lemma
 \ref{LLbphi}, which shows that it suffices to estimate $F^k$ (see the equation \eqref{eqZphi} for $Z^k\phi$) and notations thereafter.
 We see that $F^k$
 consists of $H_2^k-[Z^k, N]\phi$ satisfying the estimate \eqref{H2kzkn} in the proof of Lemma
 \ref{D3aH2k}, $Q_2^k$ satisfying estimate \eqref{Qzk} in the proof of Lemma
 \ref{D3aQ2k},
 $Q_1^k+H_1^k$ satisfying estimates given in the line
 \eqref{estH1k}. Therefore, we can estimate $F^k$ as follows
\begin{align*}
&\int_{\om}|F^k|^2d\om\les \sum\limits_{l\leq k}\int_{\om}r^{-2}|\pa
Z^l\phi|^2+|L\Lb
Z^{l-1}\phi|^2+H^2|\overline{\pa_v}Z^l\phi|^2d\om,\quad r\geq R;\\
&\int_{\om}|F^k|^2d\om\les \sum\limits_{l\leq k}\int_{\om}r^{-2}|\pa
Z^l\phi|^2+|L\Lb Z^{l-1}\phi|^2d\om,\quad 1\leq r\leq R.
\end{align*}
Note that we already have estimates for $\overline{\pa_v}Z^k\phi$ (see \eqref{pavzkphi}, \eqref{nabbzkphi}) and estimate for $\pa Z^k\phi$
(inequality \eqref{pazkphi}). Then from Lemma \ref{LLbphi}, we can bound
\begin{align*}
\int_{\om}|L\Lb Z^k\phi|^2d\om\les E_0 r^{-3}\tau_+^{-1-\a}+\sum\limits_{l\leq k-1}\int_{\om}|L\Lb Z^l\phi|^2d\om, \quad r\geq 1,\quad \forall
k\leq 3.
\end{align*}
As $F^0=0$, $Z^{-1}\phi=0$, we conclude from the above inequality (simply by an induction argument) that
\begin{equation*}
\int_{\om}|L\Lb Z^k\phi|^2d\om\les E_0 r^{-3}\tau_+^{-1-\a},\quad r\geq 1, \quad \forall k\leq 3.
\end{equation*}
These estimates are sufficient to obtain all the necessary $C^2$ estimates of the solution when $\{r\geq 1\}$. In fact, from the above
discussions, we have shown that
\begin{align*}
 &\int_{\om}|\pa\Lb Z^k\phi|^2d\om\les \int_{\om}|\Lb Z^{k+1}\phi|^2d\om+\int_{\om}|L\Lb Z^k\phi|^2d\om
\les E_0 r^{-2+\ep}(1+\tau)^{-1-\a},\\
&\int_{\om}|\pa \overline{\pa_v}Z^k\phi|^2d\om\les\int_{\om}|\overline{\pa_v}Z^{k+1}\phi|^2+ |L\Lb Z^k\phi|^2+|\Lb \nabb Z^k\phi|^2d\om\les E_0
r^{-3+\ep},\quad k\leq 3.
\end{align*}
Summarizing, we have shown that when $r\geq 1$
\begin{equation}
\label{pointwise1b}
\begin{split}
 &\sum\limits_{|k|\leq 4}\int_\om |\Lb Z^k\phi|^2d\om+\sum\limits_{|k|\leq 3}\int_\om |\pa\Lb Z^k\phi|^2d\om\leq C_{\a, R}E_0  (1+r)^{-2+\ep}\tau_+^{-1-\a},\\
&\sum\limits_{|k|\leq 4}\int_\om |\overline{\pa_v} Z^k\phi|^2d\om+\sum\limits_{|k|\leq 3}\int_\om |\pa\overline{\pa_v} Z^k\phi|^2d\om\leq C_{\a,
R} E_0
 (1+r)^{-3+\ep}\\
\end{split}
\end{equation}
for some constant $C_{\a, R}$ depending only on $R$ and $\a$. Here $\tau$ is the parameter for the foliation $\Si_\tau$. We can let
$\a<\frac{1}{10}$, $\ep<\frac{\a}{4}$. If we let
\[
 \ep_0=\frac{\delta_0}{\sqrt{C_{\a, R} }}, \quad E_0\leq \ep_0,
\]
then the bootstrap assumptions \eqref{boostphiout}, \eqref{boostphiinout} can be improved.

\bigskip

Finally we need to close the bootstrap assumption \eqref{boostphiinin} when $r\leq 1$. Since the angular momentum $\Om$ vanishes when $r=0$, we
can not get much information of the solution by commuting the equation with the angular momentum for small $r$. We instead rely on the vector
field $\pa_t$. Since we have estimates for $\pa_t\phi$, $\pa_{tt}\phi$, for fixed time $t$, we can consider the elliptic equation for $\phi$ and
use the robust elliptic theory to obtain the $C^2$ estimates for the solution on the compact region $r\leq 1$ (when $t$ is fixed). We will use
Schauder's estimates to show that the solution is bounded in $C^2$. We first use Sobolev embedding and $L^p$ elliptic theory to show the $C^1$
estimate for the solution.

Fix $\tau$. Lemma \ref{lemH2estin1} and Proposition \ref{propILEDz6}
imply that
\begin{equation}
\label{ZkphiH2} \|Z^k\phi\|_{H^2(B_2)}^2\les \int_{\tau}^{\tau+1}\|\pa_t Z^k\phi\|_{H^2(B_2)}^2d\tau\les E_0 (1+\tau)^{-1-\a},\quad \forall
k\leq 4.
\end{equation}
In particular using Sobolev embedding, we have
\begin{equation}
 \label{Cf12}
 \begin{split}
\|Z^k\phi\|_{C^\f12(B_2)}^2\les & E_0  (1+\tau)^{-1-\a},\quad
\forall k\leq 4.
\end{split}
\end{equation}
Here $B_{r}$ stands for the ball in $\mathbb{R}^3$ with radius $r$. Next we show the $C^1$ estimates. Let $\nabla=(\pa_{x_1}, \pa_{x_2},
\pa_{x_3})$. Commute the equation \eqref{eqZphi} with $\nabla$. We have the elliptic equation for $\nabla Z^k\phi$, $k\leq 3$
\begin{align}
\notag
 \Delta (\nabla Z^k\phi)+&(g^{ij\ga}\pa_\ga\phi+h^{ij}) \cdot \pa_{ij}(\nabla Z^k\phi)=\nabla (F^k-Q^k-N^k)+\pa_{tt}(\nabla Z^k\phi)\\
\notag
 &-
 Q(\nabla\phi, Z^k\phi)-\nabla h^{\mu\nu}\pa_{\mu\nu}Z^k\phi-2h^{i0}\pa_i \nabla \pa_t Z^k\phi-h^{00}\pa_t\pa_t
Z^k\phi\\
\label{elleqquasi} &-2g^{0i\ga}\pa_\ga\phi\cdot \pa_i \nabla\pa_t Z^k\phi-g^{00\ga}\pa_\ga\phi\cdot \nabla \pa_t\pa_t Z^k\phi.
\end{align}
Here $\Delta$ is the Laplacian operator in $\mathbb{R}^3$. The bootstrap assumptions \eqref{boostphiinin}, \eqref{boostphiinout} on $\phi$ as
well as the assumption \eqref{HHqu} on $h^{\mu\nu}$ imply that
 \[
m_0^{kl}+g^{ij\ga}\pa_\ga\phi+h^{ij}
\] is uniformly elliptic.
We want to show that the right hand side of the above elliptic equation is bounded in $L^2(B_{r_k})$ for some $r_k\in(1, 2)$. For $\nabla
(F^k-Q^k-N^k)$, by the definitions, it consists of two parts: null form which is quadratic in $Z^l\phi$ and $H_1^k+H_2^k-N^k$ contributed by the
metric perturbation $h^{\mu\nu}$. The later one is easy to estimate as the bound on $Z^l h^{\mu\nu}$ is given. Since we have the $H^2(B_2)$
bound for $Z^k\phi$, $k\leq 4$ (estimate \eqref{ZkphiH2}), we can show that
\[
\|\nabla (H_1^k+H_2^k-N^k)\|_{L^2(B_{r_k})}^2\les E_0(1+\tau)^{-1-\a}+\sum\limits_{l<k}\|\nabla Z^l\phi\|_{H^2(B_{r_k})}^2,\quad k\leq 3.
\]
For the quadratic terms from the null form, using Lemma \ref{nullstr}, it suffices to consider $Q(Z^{k_1}\phi, Z^{k_2}\phi)$, $k_1+k_2\leq k\leq
3$, $k_2\leq k-1$ (this also includes $Q^k=Q(Z^k\phi, \phi)$). We first have
\[
|\nabla Q(Z^{k_1}\phi, Z^{k_2}\phi)|\les |\pa \nabla
Z^{k_1}\phi||\pa^2 Z^{k_2}\phi|+|\pa Z^{k_1}\phi|| \pa^2\nabla
Z^{k_2}\phi|.
\]
Since $k_1+k_2\leq 3$, without loss of generality, we may assume that $k_1\leq 1$. Then the bootstrap assumptions \eqref{boostphiinin},
\eqref{boostphiinout} show that $|\pa^2 Z^{k_1}|\les 1$. As $k_2\leq k-1\leq 2$, we conclude that $\pa^2 Z^{k_2}$ is bounded in $L^2(B_2)$. That
is the first term is bounded in $L^2(B_2)$. For the second term, we always have
\[
\int_{\om}|\pa Z^{k_1}\phi|^2|\pa^2 \nabla Z^{k_2}\phi|^2d\om\les
\sum\limits_{l\leq 3}\int_{\om}|\pa Z^l\phi|^2d\om\cdot
\sum\limits_{l\leq k-1}\int_{\om}|\pa^2 \nabla Z^l\phi|^2d\om.
\]
In any case, we can show that
\[
\|\nabla Q(Z^{k_1}\phi, Z^{k_2}\phi)\|_{L^2(B_{r_k})}^2\les
E_0 (1+\tau)^{-1-\a}+\sum\limits_{l\leq k-1}\|\nabla
Z^l\phi\|_{H^2(B_{r_k})}
\]
This gives the estimate for $\nabla (F^k-Q^k-N^k)$. For the other terms on the right hand side of the above elliptic equation \eqref{elleqquasi}
for $\nabla Z^k\phi$, their $L^2(B_{r_k})$ norm  can be bounded by $\sqrt{E_0}\tau_+^{-\f12-\f12\a}$ (up to a constant depending only on $\a$,
$R$) as $Z^k\phi$ is bounded in $H^2(B_{2})$. Therefore the elliptic theory shows that
\begin{align*}
 \|\nabla Z^k\phi\|_{C^\f12(B_{r_{k}'})}^2 &\les\|\nabla Z^k\phi\|_{H^2(B_{r_k}')}^2\\
 &\les
\sum\limits_{l\leq k-1}\|\nabla Z^{l}\phi\|_{H^2(B_{r_k})}^2+E_0 \tau_+^{-1-\a}, \quad 1<r_k'<r_k<2.
\end{align*}
If we choose $1<r'_k<r_k<r'_{k-1}\leq 2$, $r_0=2$ then the above
estimate implies that
\[
 \|\pa Z^k\phi\|_{C^\f12(B_{r_3'})}^2 \les E_0 (1+\tau)^{-1-\a},\quad \forall k \leq
 3,\quad r_3'>1.
\]
In particular, this gives the $C^1$ estimates for the solution when
$\{r\leq 1\}$.

Finally, we use the $C^{1, \f12}(B_{r_3'})$ estimates for $Z^k\phi$,
$k\leq 3$ to show the $C^2$ estimates of the solution. We now
consider the elliptic equation for $Z^k\phi$, $k\leq 2$
\begin{align*}
 \Delta Z^k\phi+&(g^{ij\ga}\pa_\ga\phi+h^{ij})\cdot \pa_{ij}Z^k\phi=
 F^k-Q^k-N^k+\pa_{tt}Z^k\phi\\&-2(g^{0i\ga}\pa_{\ga}\phi+h^{0i})\cdot \pa_i\pa_t Z^k\phi
 -(g^{00\ga}
\pa_\ga\phi+h^{00})\cdot \pa_{tt}Z^k\phi.
\end{align*}
Similarly, by the definition of $F^k$, $Q^k$, $N^k$, we can estimate their $C^{\f12}$ norm as follows
\[
\|F^k-Q^k-N^k\|_{C^\f12(B_{s_k})}^2\les E_0 \tau_+^{-1-\a}+\sum\limits_{l\leq k-1}\|Z^l\phi\|_{C^{2, \f12}(B_{s_k})}^2,\quad 1<s_k<r_3'.
\]
For the other terms on the right hand side of the above elliptic
equation for $Z^k\phi$, we already have estimates of $Z^k\phi$,
$k\leq 3$ in $C^{1, \f12}(B_{r_3'})$ and estimates of $Z^k\phi$,
$k\leq 4$ in $C^{\f12}(B_{r_3'})$. Hence Schauder's estimates imply
that for all $s_k'<s_k$
\[
 \|Z^k\phi\|_{C^{2, \f12}(B_{s_k'})}^2\les \sum\limits_{l\leq k+2}\|Z^k\phi\|_{C^\f12(B_{s_k})}^2+\sum\limits_{l\leq k-1}\|Z^{l}\phi\|_{C^{2, \f12}(B_{s_k})}
+\sqrt{E_0}\tau_+^{-\f12-\f12 \a}.
\]
If we choose $s_0=r_3'$, $1<s_k'<s_k<s_{k-1}'\leq r_3'$, then we
have
\[
 \|Z^k\phi\|_{C^{2, \f12}(B_{s_2'})}^2\les E_0 \tau_+^{-1-\a},\quad \forall k\leq 2.
\]
In particular, this yields the $C^2$ estimates for $Z^k\phi$, $k\leq
2$ when $r\leq 1<s_2'$. To summarize, we have shown that
\begin{equation}
\label{pointwise1l}
 \sum\limits_{|k|\leq 3}|\pa Z^k\phi|^2+\sum\limits_{|k|\leq 2}|\pa^2 Z^k\phi|^2\leq C_{\a, R} E_0 \tau_+^{-1-\a},\quad r\leq 1
\end{equation}
for some constant $C_{\a, R}$ depending only on $R$ and $\a$. Without loss of generality, we may assume this constant is the same as the one in
\eqref{pointwise1b}. If
\[
 \ep_0\leq \frac{\delta_0}{\sqrt{C_{\a, R}}},\quad E_0\leq \ep_0,
\]
then the bootstrap assumptions, \eqref{boostphiinin}, \eqref{boostphiinout}, \eqref{boostphiinin} are improved. We thus closed all the bootstrap
assumptions.

\section{Proof of the main theorem}
Since the initial data have compact support, the finite speed of propagation for solutions of wave equations implies that the solution of the
quasilinear wave equation \eqref{QUASIEQsim} vanishes when $r\geq t+R$. Like what we did in the end of \cite{yang1}, we can run
the same Picard iteration process and the above argument shows that limiting solution (may be local in time) $\phi$ of the quasilinear wave
equation \eqref{QUASIEQsim} is bounded in $C^2$. Then by using the fact see e.g. \cite{hormander} that as long as the solution is
bounded in $C^2$, the solution exists globally. We thus proved the small data global existence result for quasilinear wave equations. Moreover,
the solution $\phi$ satisfies the estimates \eqref{pointwise1b}, \eqref{pointwise1l}. Using Sobolev embedding, we conclude that the solution
$\phi$ satisfies the estimates as claimed in Theorem \ref{maintheorem}. We thus proved Theorem \ref{maintheorem}.

\bigskip

For the global stability of large solutions, we first note that it
is reduced to a small data global existence result for the following
quasilinear wave equation
\[
\Box\phi+g^{\mu\nu\ga}\pa_\ga\phi\pa_{\mu\nu}\phi+g^{\mu\nu\ga}\pa_\ga\Phi\pa_{\mu\nu}\phi+g^{\mu\nu\ga}\pa_{\mu\nu}\Phi\pa_\ga\phi=0.
\]
For sufficiently small initial data, we can always solve this equation up to
time $t_0$ if the given solution $\Phi$ satisfies condition
\eqref{stabcondweak}. Then starting from time $t_0$, condition
\eqref{stabcondweak} together with Lemma \ref{nullstr} shows that
the assumptions of Theorem \ref{maintheorem} hold. Since the first
order linear term $g^{\mu\nu\ga}\pa_{\mu\nu}\Phi\pa_\ga\phi$ has a
null structure, from Theorem \ref{maintheorem}, we can
conclude Theorem \ref{stabquasi}.

\bibliography{shiwu}{}

\begin{thebibliography}{10}

\bibitem{alinhac-example}
S.~Alinhac.
\newblock An example of blowup at infinity for a quasilinear wave equation.
\newblock {\em Ast\'erisque}, (284):1--91, 2003.
\newblock Autour de l'analyse microlocale.

\bibitem{alinhac-freedecay}
S.~Alinhac.
\newblock Free decay of solutions to wave equations on a curved background.
\newblock {\em Bull. Soc. Math. France}, 133(3):419--458, 2005.

\bibitem{alinhac-sls}
S.~Alinhac.
\newblock Stability of large solutions to quasilinear wave equations.
\newblock {\em Indiana Univ. Math. J.}, 58(6):2543--2574, 2009.

\bibitem{blueHidden}
L.~Andersson and P.~Blue.
\newblock {Hidden symmetries and decay for the wave equation on the Kerr
  spacetime}.
\newblock 2009.
\newblock ar{X}iv: 0908.2265.

\bibitem{stefanos}
S.~Aretakis.
\newblock {The Wave Equation on Extreme Reissner-Nordström Black Hole
  Spacetimes: Stability and Instability Results}.
\newblock 2010.
\newblock ar{X}iv:1006.0283.

\bibitem{ChDNull}
D.~Christodoulou.
\newblock Global solutions of nonlinear hyperbolic equations for small initial
  data.
\newblock {\em Comm. Pure Appl. Math.}, 39(2).

\bibitem{kcg}
D.~Christodoulou and S.~Klainerman.
\newblock {\em The global nonlinear stability of the Minkowski space},
  volume~41 of {\em Princeton Mathematical Series}.
\newblock Princeton University Press, Princeton, NJ, 1993.

\bibitem{jnotes}
M.~Dafermos and I.~Rodnianski.
\newblock {Lectures on black holes and linear waves}.
\newblock 2008.

\bibitem{dr3}
M.~Dafermos and I.~Rodnianski.
\newblock The redshift effect and radiation decay on black hole spacetimes.
\newblock {\em Comm. Pure Appl. Math.}, 62(7):859--919, 2009.

\bibitem{newapp}
M.~Dafermos and I.~Rodnianski.
\newblock A new physical-space approach to decay for the wave equation with
  applications to black hole spacetimes.
\newblock In {\em X{VI}th {I}nternational {C}ongress on {M}athematical
  {P}hysics}, pages 421--432. World Sci. Publ., Hackensack, NJ, 2010.

\bibitem{Igorsurvey}
M.~Dafermos and I.~Rodnianski.
\newblock {The black hole stability problem for linear scalar perturbations}.
\newblock 2010.
\newblock ar{X}iv: 2010.5137.

\bibitem{newapp3}
M.~Dafermos and I.~Rodnianski.
\newblock {Decay for solutions of the wave equation on Kerr exterior spacetimes
  {III}: {T}he case {$|a|<M$}}.
\newblock in preparation.

\bibitem{hormander}
L.~H{\"o}rmander.
\newblock {\em Lectures on Nonlinear Hyperbolic Differential Equations}.
\newblock Springer-Verlag, Berlin, 1997.

\bibitem{johnblowup}
F.~John.
\newblock Blow-up for quasilinear wave equations in three space dimensions.
\newblock {\em Comm. Pure Appl. Math.}, 34(1):29--51, 1981.

\bibitem{kl-johnalmostge}
F.~John and S.~Klainerman.
\newblock Almost global existence to nonlinear wave equations in three space
  dimensions.
\newblock {\em Comm. Pure Appl. Math.}, 37(4):443--455, 1984.

\bibitem{soggestar}
M.~Keel, H.~Smith, and C.~Sogge.
\newblock Global existence for a quasilinear wave equation outside of
  star-shaped domains.
\newblock {\em J. Funct. Anal.}, 189(1):155--226, 2002.

\bibitem{soggestaralm}
M.~Keel, H.~Smith, and C.~Sogge.
\newblock Almost global existence for quasilinear wave equations in three space
  dimensions.
\newblock {\em J. Amer. Math. Soc.}, 17(1):109--153 (electronic), 2004.

\bibitem{klgex}
S.~Klainerman.
\newblock Global existence for nonlinear wave equations.
\newblock {\em Comm. Pure Appl. Math.}, 33(1):43--101, 1980.

\bibitem{klNullc}
S.~Klainerman.
\newblock Long time behaviour of solutions to nonlinear wave equations.
\newblock In {\em Proceedings of the {I}nternational {C}ongress of
  {M}athematicians, {V}ol.\ 1, 2 ({W}arsaw, 1983)}, pages 1209--1215, Warsaw,
  1984. PWN.

\bibitem{klinvar}
S.~Klainerman.
\newblock Uniform decay estimates and the lorentz invariance of the classical
  wave equation.
\newblock {\em Comm. Pure Appl. Math.}, 38(3):321--332, 1985.

\bibitem{klNull}
S.~Klainerman.
\newblock The null condition and global existence to nonlinear wave equations.
\newblock In {\em Nonlinear systems of partial differential equations in
  applied mathematics, Part 1 (1984)}, volume~23 of {\em Lectures in Appl.
  Math.}, pages 293--326. Amer. Math. Soc., Providence, RI, 1986.

\bibitem{kl-ponce}
S.~Klainerman and G.~Ponce.
\newblock Global, small amplitude solutions to nonlinear evolution equations.
\newblock {\em Comm. Pure Appl. Math.}, 36(1):133--141, 1983.

\bibitem{klmulti}
S.~Klainerman and T.~Sideris.
\newblock On almost global existence for nonrelativistic wave equations in 3d.
\newblock {\em Comm. Pure Appl. Math.}, 49(3).

\bibitem{gx-lindblad}
H.~Lindblad.
\newblock Global solutions of nonlinear wave equations.
\newblock {\em Comm. Pure Appl. Math.}, 45(9):1063--1096, 1992.

\bibitem{gx-lindblad2}
H.~Lindblad.
\newblock Global solutions of quasilinear wave equations.
\newblock {\em Amer. J. Math.}, 130(1):115--157, 2008.

\bibitem{igorStabvacuum}
H.~Lindblad and I.~Rodnianski.
\newblock Global existence for the {E}instein vacuum equations in wave
  coordinates.
\newblock {\em Comm. Math. Phys.}, 256(1):43--110, 2005.

\bibitem{SMigor}
H.~Lindblad and I.~Rodnianski.
\newblock The global stability of {M}inkowski space-time in harmonic gauge.
\newblock {\em Ann. of Math. (2)}, 171(3):1401--1477, 2010.

\bibitem{improvLuk}
J.~Luk.
\newblock Improved decay for solutions to the linear wave equation on a
  {S}chwarzschild black hole.
\newblock {\em Ann. Henri Poincar\'e}, 11(5):805--880, 2010.

\bibitem{smallglobLuk}
J.~Luk.
\newblock {The Null Condition and Global Existence for Nonlinear Wave Equations
  on Slowly Rotating Kerr Spacetimes}.
\newblock 2010.
\newblock ar{X}iv: 1009.4109.

\bibitem{sogge-metcalfe-nakamura}
J.~Metcalfe, M.~Nakamura, and C.~D. Sogge.
\newblock Global existence of solutions to multiple speed systems of
  quasilinear wave equations in exterior domains.
\newblock {\em Forum Math.}, 17(1):133--168, 2005.

\bibitem{sogge-metcalfe2}
J.~Metcalfe and C.~Sogge.
\newblock Long-time existence of quasilinear wave equations exterior to
  star-shaped obstacles via energy methods.
\newblock {\em SIAM J. Math. Anal.}, 38(1):188--209 (electronic), 2006.

\bibitem{sogge-metcalfe}
J.~Metcalfe and C.~Sogge.
\newblock Global existence of null-form wave equations in exterior domains.
\newblock {\em Math. Z.}, 256(3):521--549, 2007.

\bibitem{mora2}
C.~S. Morawetz.
\newblock Time decay for the nonlinear klein-gordon equations.
\newblock {\em Proc. Roy. Soc. Ser. A}, 306:291--296, 1968.

\bibitem{ge-shatah}
J.~Shatah.
\newblock Global existence of small solutions to nonlinear evolution equations.
\newblock {\em J. Differential Equations}, 46(3):409--425, 1982.

\bibitem{sideris-multispeed}
T.~Sideris and S.-Y. Tu.
\newblock Global existence for systems of nonlinear wave equations in 3{D} with
  multiple speeds.
\newblock {\em SIAM J. Math. Anal.}, 33(2):477--488 (electronic), 2001.

\bibitem{soggemulti}
C.~Sogge.
\newblock Global existence for nonlinear wave equations with multiple speeds.
\newblock In {\em Harmonic analysis at {M}ount {H}olyoke ({S}outh {H}adley,
  {MA}, 2001)}, volume 320 of {\em Contemp. Math.}, pages 353--366. Amer. Math.
  Soc., Providence, RI, 2003.

\bibitem{lodecTataru}
D.~Tataru.
\newblock {Local decay of waves on asymptotically flat stationary space-times}.
\newblock 2009.
\newblock ar{X}iv: 0910.5209.

\bibitem{glsoChengbo}
C.~Wang and Y.~Xin.
\newblock {Global existence of null-form wave equations on small asyptotically
  Euclidean manifolds}.
\newblock 2012.
\newblock ar{X}iv:1207.5218.

\bibitem{yang2}
S.~Yang.
\newblock {Global stability of large solutions to nonlinear wave equations}.
\newblock 2012.
\newblock ar{X}iv: 1205.4216.

\bibitem{yang1}
S.~Yang.
\newblock Global solutions of nonlinear wave equations in time dependent
  inhomogeneous media.
\newblock {\em Arch. Ration. Mech. Anal.}, 209(2):683--728, 2013.

\bibitem{yang5}
S.~Yang.
\newblock {Global solutons of nonlinear wave equations with large energy}.
\newblock preprint.

\end{thebibliography}
\bibliographystyle{plain}

\bigskip

DPMMS, Centre for Mathematical Sciences, University of Cambridge,
Wilberforce Road, Cambridge, UK CB3 0WA

\textsl{Email address}: S.Yang@damtp.cam.ac.uk
 \end{document}